\documentclass[11pt]{article}
\usepackage{appendix}
\usepackage{mathtools}
\usepackage[T1]{fontenc}
\usepackage{amsfonts}
\usepackage{amsmath}
\usepackage{amssymb}
\usepackage{amsthm}
\usepackage{thmtools}
\usepackage{bbm}
\usepackage{bm}
\usepackage{mathrsfs}
\usepackage{color}
\usepackage{pdfsync}
\usepackage{enumitem}
\usepackage{subcaption}
\usepackage{xcolor} %
\definecolor{darkgreen}{rgb}{0,0.5,0} %
\definecolor{darkbkue}{rgb}{1,0,0} 
\usepackage{natbib}
\setcitestyle{authoryear,aysep={,}}  %
\usepackage[colorlinks=true, linkcolor=blue, citecolor=blue, urlcolor=blue]{hyperref}
\makeatletter
\renewcommand\@makefnmark{%
  \hbox{\textsuperscript{\textcolor{darkgreen}{\@thefnmark}}}}
\makeatother
\usepackage{tikz}
\usetikzlibrary{calc, positioning,arrows.meta}
\usepackage[most]{tcolorbox}
\usepackage{booktabs}  
\usepackage[scaled]{helvet}
\DeclareMathOperator*{\argmin}{argmin}
\DeclareMathOperator*{\argmax}{argmax}

\newcommand{\PP}{\mathbb{P}}

\newcommand{\cP}{\mathcal{P}}

\usepackage[margin=31 mm]{geometry}
\newcommand{\mykill}[1]{}
\usepackage[capitalize, nameinlink]{cleveref}

\newlist{hypotheses}{enumerate}{1}
\setlist[hypotheses]{
    label=\textcolor{darkgreen}{(H\arabic*)},
    ref=(H\arabic*)
}

\usepackage{amsfonts,bbm,bm}

\newcommand{\BF}{\mathbf{F}}

\newcommand{\BI}{\mathbf{I}}

\newcommand{\BQ}{\mathbf{Q}}

\newcommand{\BS}{\mathbf{S}}
\newcommand{\BT}{\mathbf{T}}

\newcommand{\CC}{\mathcal{C}}

\newcommand{\cB}{\mathcal{B}}
\newcommand{\cC}{\mathcal{C}}

\newcommand{\cE}{\mathcal{E}}

\newcommand{\cG}{\mathcal{G}}
\newcommand{\cH}{\mathcal{H}}

\newcommand{\cL}{\mathcal{L}}

\newcommand{\cN}{\mathcal{N}}

\newcommand{\cS}{\mathcal{S}}

\newcommand{\cU}{\mathcal{U}}

\newcommand{\cW}{\mathcal{W}}

\newcommand{\R}{\mathbb{R}} 
\newcommand{\N}{\mathbb{N}} 
\newcommand{\E}{\mathbb{E}} 
\newcommand{\BB}{\mathbb{B}}
\newcommand{\BV}{\mathbb{V}}
\newcommand{\BL}{\mathbb{L}}
\newcommand{\tr}{\mathrm{tr}} 

\newcommand{\diver}[1]{\mathrm{div}(#1)}

\newcommand{\vae}{\varepsilon}

\newcommand{\rom}[1]{%
  \textup{\uppercase\expandafter{\romannumeral#1}}%
}

\newcommand{\inner}[2]{\langle #1,#2\rangle}

\theoremstyle{plain}
\newtheorem{theorem}{Theorem}[section]
\newtheorem{proposition}[theorem]{Proposition}
\newtheorem{lemma}[theorem]{Lemma}
\newtheorem{corollary}[theorem]{Corollary}
\theoremstyle{definition}
\newtheorem{definition}[theorem]{Definition}
\newtheorem{remark}[theorem]{Remark}

\newtheorem{conjecture}[theorem]{Conjecture}
\theoremstyle{remark}

\newcommand*{\indic}[1]{\mathbf{1}_{#1}}
\newcommand*{\rd}{\mathrm{d}}
\newcommand*{\dd}{\, \rd}

\newlist{myenum}{enumerate}{3}
\setlist[myenum,1]{label={\rm (H\arabic*)},
                   ref  ={\rm (H\arabic*)}}
\crefname{myenumi}{property}{properties}

\renewenvironment{thebibliography}[1]{%
\begin{oldthebibliography}{#1}%
\setlength{\baselineskip}{.9em}
\linespread{1}
\small
\setlength{\parskip}{.40ex}%
\setlength{\itemsep}{.30em}%
}%
{%
\end{oldthebibliography}%
}

\definecolor{grape}{rgb}{0.43, 0.17, 0.71}

\begin{document}

\title{
  \resizebox{0.95\textwidth}{!}{The Influence Function of Transport-based Quantiles}
}
\date{\today}
\author{  
  Alberto Gonz{\'a}lez-Sanz \and  Shunan Sheng \and  Bohan Wu \and Marco Avella Medina\thanks{Department of Statistics, Columbia University}   
  }
  
\maketitle 
\vspace{-1.5em}

\begin{abstract}
Transport-based quantiles provide a canonical way to extend univariate quantiles to
multivariate distributions through optimal transport. In this paper, we study the
influence function of the transport-based quantile map \(\BQ_P\), defined as the transport map
pushing a fixed reference measure \(\mu\) forward to a target distribution \(P\). For
Huber-type perturbations \(P_t=(1-t)P+t\delta_{x_0}\), we establish the existence and
uniqueness of the first-order limit
\[
\BI(x_0;\BQ_P(z))
:=
\lim_{t\downarrow0}
\frac{\BQ_{(1-t)P + t \delta_{x_0}}(z) - \BQ_P(z)}{t}
\]
away from the point \(x_0=\BQ_P(z)\). We show that this influence function admits the
representation $ \BI(x_0;\BQ_P(z))=\nabla G_{x_0}(z),$
where the potential \(G_{x_0}\) is characterized by a uniformly elliptic equation with
a Dirac source and Neumann boundary condition.  Next, we show that,
in dimension \(d\geq 2\), the influence function has a pole-type singularity. More
precisely, for fixed \(z\in \operatorname{int}(\Omega_\mu)\), the influence function
remains bounded when the perturbation point \(x_0\) is such that \(\BF_P(x_0)\) stays
away from \(z\),  where \(\BF_P=\BQ_P^{-1}\) is the transport-based distribution
function. However,  the influence function diverges as \(x_0\to \BQ_P(z)\), equivalently as
\(\BF_P(x_0)\to z\). In fact,
$\|\BI(x_0;\BQ_P(z))\|
\asymp
\frac{1}{\|z - \BF_P(x_0)\|^{d-1}}$. This singular behavior contrasts with the bounded influence function of
univariate quantiles and implies that the influence-function random variable
\(\BI(X;\BQ_P(z))\), \(X\sim P\), has infinite second moment. We also provide numerical
evidence suggesting that empirical transport quantiles may exhibit stable-type, non-Gaussian fluctuations.
\end{abstract}

 \vspace{1em}

{\small

\noindent \emph{MSC2020 subject classifications. Primary: } 62G35, 62G30.  
}
 \vspace{.2em}
 \section{Introduction}
Given probability measures $\mu$ and $P$ that are absolutely continuous with respect to the Lebesgue measure $\cL_d$ and have finite second moments, the optimal transport (OT) map from $\mu$ to $P$ under the squared cost is defined by
\begin{align}\label{eq:OT}%
T_{\mu\to P} 
= \argmin_{T:\, T\#\mu=P}
\int_{\mathbb{R}^d} \|z - T(z)\|^2 \, \mu(\rd z).
\vspace{-2mm}
\end{align}
Here $T\#\mu$ denotes the \emph{push-forward} measure, that is, the measure defined by $T\#\mu(A) := \mu\bigl(T^{-1}(A)\bigr)$ for each measurable set $A \subset \mathbb{R}^d$. The solution $T_{\mu\to P}$ to the OT problem~\eqref{eq:OT} is well defined and unique $\mu$-a.e. Moreover, it can be represented as the gradient of a convex function~\citep{Cuesta1989NotesOT, Ruschendorf.Rachev.1990, brenier1991polar}.

Among the recent applications of optimal transport (see \cite{PeyreCuturi.19,CheNilRig25OT}
and references therein), the definition of multivariate quantile functions via optimal
transport is a prominent contribution to classical statistics
\citep{ChernozhukovEtAl.17.AOS,HallinDelBarrioCuestaAlbertosMatran.21}.
The \emph{transport quantile} function $\BQ_P := T_{\mu\to P}$ of a probability measure
$P$ is defined as the optimal transport map from a fixed reference measure $\mu$ to $P$.\footnote{More precisely, the transport-based quantile function
$\mathbf Q_P$ is defined as the unique $\mu$-almost everywhere gradient of a
convex function satisfying $(\mathbf Q_P)_{\#}\mu=P$. This definition does not
require $P$ to have finite moments. The finite-second-moment assumption is only
needed when interpreting $\mathbf Q_P$ as the solution of \eqref{eq:OT}.}
When $P$ is absolutely continuous with respect to $\cL_d$ as well, the inverse map
$\mathbf{F}_P$ of $\BQ_P$ exists and is called the \emph{transport distribution} function
of $P$. Compared to classical notions of multivariate quantiles, such as those based on
Tukey depth or spatial depth, transport quantiles induce distribution-free
center-outward ranks and signs  and provide
a full characterization of the target distribution \citep{HallinDelBarrioCuestaAlbertosMatran.21}. The latter property fails
for Tukey depth \citep{Nagy2021}, but has recently been established for spatial
depth \citep{GonzalezSanzKonen2026}.

The growing range of applications of the transport quantile function calls for a systematic investigation of its statistical properties; see, e.g., \cite{ ghosal2022multivariate, HallinVecchiaLiu2022, deb2024distribution,huang2025multivariate, BarrioSanzHallin2025, ShiDrtonHallinHan2025,GonzalezSanzHallinYao2025} 
and references therein. Robustness  is a fundamental aspect in many applications~\citep{ronchetti2023}, which boils down to studying the \emph{stability} of transport quantiles under perturbations of the target measure $P$ that may arise from data contamination~\citep{HuberRonchetti2009,hampeletal1986,maronnaetal2019}. %
The two most classical quantifiers of robustness are 
the \emph{breakdown point} and the \emph{influence function}. The breakdown point characterizes the maximal proportion of contamination that an estimator can tolerate before becoming arbitrarily unstable, whereas the influence function quantifies the first-order effect of infinitesimal perturbations of the target measure. %

Recent work by  \citet{paindaveine:passeggeri2026,Avella-Gonzalez.26,gonzalezsanz:avellamedina2026} showed that transport quantiles enjoy a high breakdown point under suitable conditions. However, despite its fundamental importance, the influence function of transport quantiles has not yet been systematically investigated. 

\subsection{Main contributions}

In this paper, we initiate the study of the influence function of transport quantiles when $d \ge 2$.\footnote{When $d=1$, the transport quantile is given by $F_P^{-1}\circ F_\mu$, where $F_\mu$ and $F_P$ are the distribution functions of $\mu$ and $P$, respectively. Standard univariate quantiles are recovered by taking $\mu={\rm Unif}(0,1)$ and our derivation recovers the influence function for univariate quantiles.}  Let $\theta: \cP(\R^d) \mapsto \R^m$ be a statistical functional of interest.  To follow the classical definition of \cite{HuberRonchetti2009}, the influence function (IF) of $\theta$ at a point $x_0 \in \R^d$ is defined by
\begin{equation*}
    \BI(x_0;\theta,P)
=
\lim_{t\downarrow0}
\frac{1}{t}\Big(\theta((1-t)P + t \delta_{x_0}) - \theta(P)\Big),
\end{equation*}
whenever the limit exists. In our setting, the functional of interest is $\theta(P) = \BQ_P(z)$ for a fixed transport quantile level $z$ in the support of $\mu$. We show in \Cref{Thm:InfluenceQuantiles-main}(i) that the influence function
\[
\BI(x_0;\BQ_P(z))
:=
\lim_{t\downarrow0}
\frac{\BQ_{(1-t)P + t \delta_{x_0}}(z) - \BQ_P(z)}{t}
\]
exists for all $x_0 \neq \BQ_P(z)$. Moreover, \Cref{Thm:InfluenceQuantiles-main}(ii) shows that $\BI(x_0;\BQ_P(z))$ admits the representation
\[
\BI(x_0;\BQ_P(z))
=
\nabla G_{x_0}(z),
\]
where the potential $G_{x_0}$ is characterized as the weak solution to a uniformly elliptic partial differential equation (PDE) with a Dirac mass on the right-hand side and Neumann boundary conditions.

Furthermore, the influence function shows a \emph{pole-type singularity} at $\BF_P(x_0)$ (see \Cref{thm:Unbounded-IF}). More precisely, for $x_0$ sufficiently close to $\BQ_P(z)$, we have
\begin{equation*}
\|\BI(x_0;\BQ_P(z))\|
\asymp
\frac{1}{\|z - \BF_P(x_0)\|^{d-1}},
\end{equation*}
where $\BF_P = (\BQ_P)^{-1}$ is the transport distribution function of $P$. The observed singular behavior contrasts with the boundedness of the influence function in the one-dimensional case and arises from the singular nature of the Dirac mass in the associated PDE. 

As a byproduct of our proof strategy, we also characterize the shape of the set (see~\Cref{prop:shape})
\[
K_{x_0,t}
=
\big(\BQ_{(1-t)P + t\delta_{x_0}}\big)^{-1}(x_0),
\]
that is, the set of points mapped to the perturbation location $x_0$ at time $t>0$. We show that $K_{x_0,t}$ behaves asymptotically as $t\to 0$ like a ball of radius $t^{1/d}$ for $d\ge 3$.

Our derivation of the influence function suggests that the empirical optimal transport map may exhibit a heavy-tailed limiting distribution; see~\cref{sec:examples} for a more extended discussion of this point. %

\subsection{Related work and technical challenges}
Our analysis relies on the PDE characterization of the optimal transport map. A similar idea was adopted in \citep{Loeper.09, gonzalez2024linearization} to study the stability of the optimal transport map under \textit{smooth} perturbations of the target measure. This builds on a line of work about the quantitative and qualitative stability of optimal transport maps, including \cite{loeper2005regularity, loeper2006fully, DePhilippis2013, delBarrioLoubes.19, segers2022graphical, manole2023central, delBarrioGonzalezLoubes.24, balakrishnan2025stability}.

However, when the perturbation is of \emph{Huber type} with a mass point, that is, when the target measure $P$ is contaminated by the point mass at a singleton, little is known about the corresponding stability problem. In fact, the nonsmooth nature of the perturbation poses a serious technical challenge. Denote by $P_t = (1-t)P + t\delta_{x_0}$ the Huber-contaminated measure at time $t>0$. In this setting, the transport quantile $\BQ_{P_t}$ fails to be differentiable, even when the unperturbed transport quantile $\BQ_P$ is differentiable (see  \Cref{prop:Example-uniform}~(i)). 
 As a result, classical approaches relying on the implicit function theorem for deriving the influence function do not apply. Instead, we analyze the problem by first characterizing the limit in the $L^q$ sense for some $q \in (1,2)$, and then upgrading the convergence to the local space $\cC^{2,\alpha}_{\rm Loc}$ containing functions that belong to $\cC^{2,\alpha}(K)$ for every compact $K \subset \Omega$. 
 The pole-type singularity is subsequently derived via a localization argument, in which the analysis is carried out on regions away from the singularity point $\BF_P(x_0)$.

The proof of \Cref{Thm:InfluenceQuantiles-main} builds upon an $L^2$ estimate (see \Cref{thm:L2-estimate}), which quantifies the $L^2$ rate of convergence of $\BQ_{P_t} - \BQ_P$. Nonetheless, our quantity of interest is 
\[
\nabla \xi_t := \frac{1}{t}(\BQ_{P_t} - \BQ_P),
\]
which diverges under the current $L^2$ control. We therefore refine the estimate and prove that $\{\xi_t\}_{t \le t_0}$ is uniformly bounded in $L^q$ for some $q \in (1,2)$ and sufficiently small $t_0>0$ (see~\Cref{Prop:Lq-estimate}). From this point, we exploit the local maximum principle, obtained via \textit{Moser iteration}, followed by \textit{Schauder estimates}, to derive $\cC_{\rm Loc}^{2,\alpha}$ bounds and complete the proof; we refer the reader to~\cite{GilbargTrudinger.Book} for the background in PDE.

 We summarize the logical structure of the argument and present a roadmap toward the main results in \Cref{fig:dependency-diagram}. As illustrated there, the $L^2$ estimate serves as the cornerstone of our analysis. For intuition, we provide in \Cref{prop:Example-uniform}~(ii) a direct derivation of the $L^2$ estimate in the special case $\mu = P = \cU(\BB(0,1))$.

The robustness of multivariate quantile functions has been a longstanding topic in the statistical literature. In the case of Tukey median, \citet{Donoho.Gasko.AoS.1992} proved that the breakdown point of the Tukey median is bounded between $1/(d+1)$ and $1/3$ for general distributions. Complementing this, \citet{ChenTyler2002IF} provided an explicit characterization of the influence function of the Tukey median in the halfspace-symmetric setting. Robustness has also been studied for multivariate quantiles beyond depth-based notions: \citet{KonenPaindaveine2025} analyzed general multivariate $M$-quantiles, and \citet{Konen.Paindaveine.AIHP.2026} did so for spatial/geometric quantiles.
Furthermore, \citet{Konen.BJ.2025} showed that any probability measure is fully characterized by its geometric distribution function through the fractional Laplacian operator $F \mapsto (-\Delta)^{(d-1)/2}\nabla\cdot F$.

\begin{figure}[t]
\centering
\begin{tikzpicture}[
  thm/.style={draw, rectangle, inner sep=3pt, align=center},
  arr/.style={-Latex, line width=0.9pt},
  node distance=14mm and 34mm
]
\node[thm] (t42) {{\Large \Cref{thm:L2-estimate}}\\{\normalsize ($L^2$ estimate)}};
\node[thm, right=23mm of t42] (l54) {{\Large \Cref{lemma:test-linearization}}\\{\normalsize (linearization)}};
\node[thm, right=23mm of l54] (p53) {{\Large \Cref{Prop:Lq-estimate}}\\{\normalsize ($L^q$ estimate)}};
\node[thm, below=of l54] (p55) {{\Large \Cref{prop:Uniform-estimate}}\\{\normalsize (local $\cC^{2,\alpha}$ estimate)}};
\node[thm, below=of t42] (p34) {{\Large \Cref{prop:shape}}\\{\normalsize (shape control)}};

\node[thm, below=of p53] (t32) {{\Large \Cref{Thm:InfluenceQuantiles-main}}
\\{\normalsize (characterization of IF)}};
\node[thm, below=of p55] (t33) {{\Large \Cref{thm:Unbounded-IF}}\\{\normalsize (pole-type singularity of IF)}};

\draw[arr] (t42) -- (l54);
\draw[arr] (t42) -- (p34);
\draw[arr] (l54) -- (p53);
\draw[arr] (p53) -- (p55);

\draw[arr] (p53) -- (t32);
\draw[arr] (t32) -- (t33);

\draw[arr] (p55) -- (t32);
\draw[arr] (p34) -- (p55);

\end{tikzpicture}
\caption{Flow Chart for the Proof of Main Results.}
\label{fig:dependency-diagram}
\end{figure}
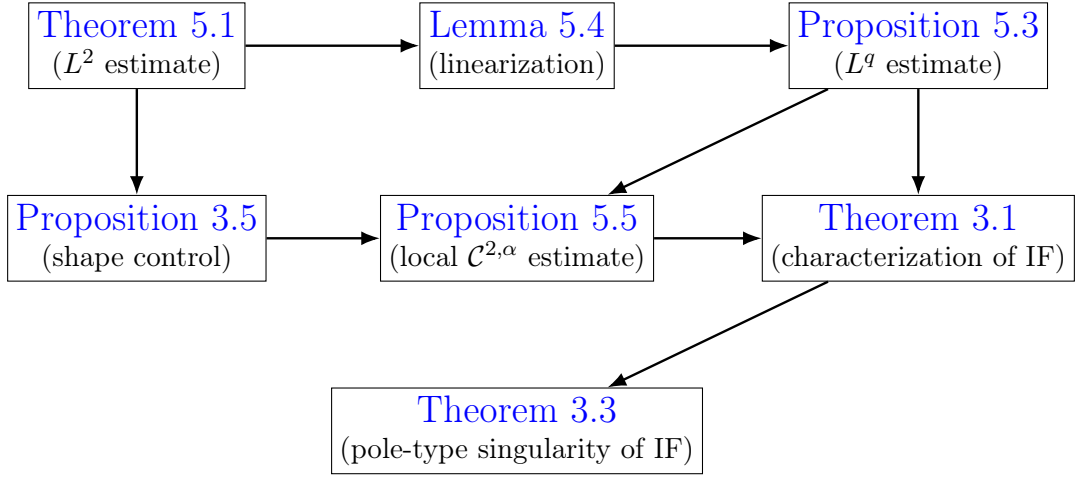

\paragraph{Organization} The remainder of the paper is organized as follows. \Cref{sec:background} introduces the notation and reviews the background on optimal transport. \Cref{sect:Main-rsult} states our main results on the influence function and its pole-type singularity. \Cref{sec:examples} presents several examples and numerical experiments. \Cref{sec:proofs} contains the proofs of the main results. Finally, \Cref{sec:solvability} collects the PDE results used in the proofs.

\section{Notation and Preliminaries}\label{sec:background}
\subsection{Notation}\label{Sect:Notation}
Let $\cP(\R^d)$ denote the space of probability measures on $\R^d$ and $\cP_2(\R^d)$ the subspace of measures with finite second moment. For any set $A \subset \mathbb{R}^d$, we write $\overline{A}$ for its topological closure and $\mathrm{int}(A)$ for its interior.  We write $\BB_R(x_0)$ for the open ball of radius $R > 0$ centered at $x_0 \in \R^d$.  For a probability measure $\nu$, we denote its support by $\Omega_\nu$ and Lebesgue density (if it exists) by $f_\nu$.  The $k$-dimensional Hausdorff measure is denoted by $\cH^{k}$. For symmetric matrices $A, B \in \mathbb{R}^{d \times d}$, the notation $A \preceq B$ indicates that $B - A$ is positive semidefinite. The $d \times d$ identity matrix is denoted by $\mathrm{I}_d$. We use $\nabla f$ for the gradient of a function $f$ and $\nabla^2 f$ for its Hessian matrix of second derivatives. The space $\cC^{k,\alpha}(\Omega)$ consists of functions whose derivatives up to order $k$ are $\alpha$-Hölder continuous on $\Omega$; we simply write $\cC^{k,\alpha}$ when the domain is clear from context. For $f\in \cC^{k,\alpha}(\Omega)$, its norm is given by $
\|f\|_{\cC^{k,\alpha}(\Omega)}
:=
\sum_{|\beta|\leq k}\|D^\beta f\|_{L^\infty(\Omega)}
+
\sum_{|\beta|=k}[D^\beta f]_{\cC^{0,\alpha}(\Omega)}$, 
where
$[g]_{\cC^{0,\alpha}(\Omega)}
:=
\sup_{\substack{x,y\in \Omega,\, x\neq y}}
|g(x)-g(y)|/\|x-y\|^\alpha$ and $D^{\beta}=\frac{\partial^{|\beta|}}{\partial x_1^{\beta_1}\dots \partial x_d^{\beta_d}}$, $|\beta|=\sum_{j=1}^d\beta_j$. We write $\cC^{k}(\Omega)=\cC^{k,0}(\Omega)$.  
The local space $\cC^{k,\alpha}_{\mathrm{Loc}}(\Omega)$ contains functions that belong to $\cC^{k,\alpha}(K)$ for every compact $K \subset \Omega$. For $L^p$-spaces, $L^p(\nu)$ denotes the space with respect to a probability measure $\nu$, while $L^p(\Omega)$ refers to the space with respect to Lebesgue measure on a domain $\Omega \subset \mathbb{R}^d$; the domain may be omitted when the context is clear. Finally, $W^{k,p}(\Omega)$ denotes the Sobolev space of functions with weak derivatives up to order $k$ in $L^p(\Omega)$, equipped with its standard norm. Let $\varphi(t)=e^{t^2}-1$. The Orlicz space associated with $\varphi$,
denoted by $L_\varphi(\Omega)$, is the set of measurable functions
$f:\Omega\to\R$ such that
\[
\int_\Omega \varphi\!\left(\frac{|f(x)|}{r}\right)\dd x<\infty
\]
for some $r>0$. The corresponding Luxemburg norm is given by
\[
\|f\|_{L_\varphi(\Omega)}
:=
\inf\left\{
r>0\,\middle|\,
\int_\Omega
\varphi\!\left(\frac{|f(x)|}{r}\right)\dd x
\leq 1
\right\}.
\]
A set $\Omega\subset \R^d$ is said to be a $\CC^{k,\alpha}$ domain $(k \in \N, \alpha \in [0,1])$ if for every $x_0\in \partial \Omega$, the boundary of $\Omega$, there exists a neighborhood $B=B(x_0)$ of $x_0$ in $\mathbb{R}^d$ and a diffeomorphism $\Psi : B \rightarrow D:=\Psi(B)\subset \R^d$ such that
 \begin{equation*}
     \label{eq:domain}
     \text{(i) $\Psi(B\cap \Omega)\subset \R_+^d$; \quad (ii) $\Psi(B\cap \partial \Omega) \subset\partial \R_+^d$;
\quad (iii) $\Psi\in \CC^{k,\alpha}(B), \Psi^{-1}\in \CC^{k,\alpha}(D)$}, 
 \end{equation*}
where
\[
\mathbb{R}^d_+ := \left\{x=(x_1,\ldots,x_d)\in\mathbb{R}^d : x_d>0\right\},
\]
and
\[
\partial \mathbb{R}^d_+
=
\left\{x=(x_1,\ldots,x_d)\in\mathbb{R}^d : x_d=0\right\}
=
\mathbb{R}^{d-1}\times\{0\}.
\]
For real numbers $a,b$, we write $a \lesssim b$ or $a = O(b)$ if there exists a constant $C > 0$ (which may depend on other parameters depending on context) such that $a \leq C b$. If both $a\lesssim b$ and $b \lesssim a$, we write that $a \asymp b$. We use $A \Subset B$ to mean the set $A$ is compactly contained in $B$ i.e., if $\overline{A}$ is compact and $\overline{A}\subset {\rm int}(B)$.

\subsection{Preliminaries on Optimal Transport}
We provide a brief overview of the optimal transport (OT) problem over $\R^d$ with the squared Euclidean cost function, and direct interested readers to the  expositions of \cite{Villani2003, Santambrogio2015,CheNilRig25OT} for the full treatment of this topic. In our exposition, we place an emphasis on the results of optimal transport relevant to statistical quantile estimation. 

Problem~\eqref{eq:OT} is known as the \textit{Monge formulation} of optimal transport. Alternatively, one may consider the relaxed version, known as the \textit{Kantorovich formulation}, 
\begin{equation}\label{eq:kantorovich}
\cW_2^2(\mu,P):=\inf_{\pi\in \Pi(\mu, P)}\int \|z-x\|^2\,\pi(\rd z,\rd x),
\end{equation}
where $\Pi(\mu,P)$ is the set of probability measures (couplings) on $\R^d\times\R^d$ with marginals $\mu$ and $P$. We call $\cW_2(\mu, P)$ the $2$-Wasserstein distance. Moreover, the optimal coupling $\pi^\star$ for \eqref{eq:kantorovich} exists and is given by the graph of the OT map $\pi^\star = ({\rm id} \times \BQ_P)_\# \mu$.

A central result in optimal transport is Brenier's theorem (see, e.g., \cite{Cuesta1989NotesOT}, \cite{Ruschendorf.Rachev.1990}, \cite{brenier1991polar},  \cite[Theorem~1.2]{GangboMcCann1996}, or \cite[Theorem~2.12]{Villani2003}), which characterizes the optimal transport map and shows that the transport-based quantile function $\BQ_P$ can be expressed as the gradient of a convex function.
\begin{theorem}[Brenier's theorem]
There exists a convex function $\varphi:\R^d\to\R\cup\{+\infty\}$ such that $\BQ_P=\nabla\varphi$. If $P\ll \cL_d$ as well, then $\BQ_P$ is a.e.~invertible and $\BF_P:= [\BQ_P]^{-1} = \nabla \varphi^*$, where $\varphi^*(x):=\sup_{z\in \R^d} \left(\inner{x}{z} - \varphi(z)\right)$ is the convex conjugate of $\varphi$.
\end{theorem}
Throughout the paper, we work with the following class of \textit{regular} probability measures. For any such measure $\nu$, we adopt the convention that $\Omega_\nu$ denotes $\operatorname{int}(\Omega_\nu)$ whenever it is used as the domain of a Sobolev space. In contrast, Hölder spaces on $\Omega_\nu$ are interpreted up to the boundary.

\begin{definition}\label{defn:regular-measures}
Fix $\alpha \in (0,1)$. We say that a probability measure $\nu\in \cP(\R^d)$ is {\it regular} if
\begin{enumerate}[label = (\alph*)]
\item Its support $\Omega_\nu =\operatorname{supp}(\nu)$ is bounded, convex, and ${\cC^{2,1}}$; 
\item $\nu\ll \cL_d$ with density $f_{\nu}:=\frac{\dd \nu}{\dd \cL_d} \in \cC^{1,\alpha}(\Omega_\nu)$;
\item There exist constants $0<\lambda\leq  \Lambda$ such that  
\begin{equation}\label{upperlowwer}
 \lambda  \leq f_\nu(x) \leq 
 \Lambda \quad\text{ for all $x\in \Omega_\nu$}.
\end{equation}
\end{enumerate}
\end{definition}
Many useful probability measures are regular. Let $\Omega$ be a closed, bounded, convex, and $\cC^{2,1}$ subset of $\R^d$.  Regular measures include the uniform measure on $\Omega$, denoted by $\cU(\Omega)$, and Gibbs measures with density $f_\nu(x) \propto e^{-g(x)} \indic{\Omega}(x)$, where $g \in \cC^{1, \alpha}(\Omega)$ is bounded. The latter class further includes truncated exponential family distributions that are absolutely continuous with respect to the Lebesgue measure. Moreover, finite mixtures of regular probability measures on $\Omega$ remain regular.

Finally, we recall the classical result due to \cite{caffarelli1992boundary,Caffarelli.96} and an improvement by~\cite{chen2021global}, see also \cite[Theorem~12.50]{villani2009optimal} for an exposition. 

\begin{theorem}[Caffarelli's theorem]\label{thm:Cafarrelli}
Assume that $\mu$ and $P$ are regular probability measures. Then there exists a constant 
\[
C
=
C\!\left(
d,\alpha,\lambda,\Lambda,
\Omega_\mu,\Omega_P,
\|f_\mu\|_{\cC^{1,\alpha}(\Omega_\mu)},
\|f_P\|_{\cC^{1,\alpha}(\Omega_P)}
\right)
>0
\]such that
\begin{equation*}
   \|\BF_P\|_{\cC^{2, \alpha}(\Omega_P)}+     \|\BQ_P\|_{\cC^{2, \alpha}(\Omega_\mu)}\leq C \quad \text{and} \quad \frac{1}{C} {\rm I}_d  \leq  \nabla \BQ_P(z)  \leq C{\rm I}_d ,\quad \forall z\in \Omega_\mu.
\end{equation*}
\end{theorem}

\Cref{thm:Cafarrelli} is a key technical tool for our theoretical analysis. It allows us to ``lift'' the regularity of  $\mu$ and $P$ to the regularity properties of the quantile map $\BQ_P$ and its inverse $\BF_P$, which paves the way to the use of techniques from partial differential equations (PDE); see~\Cref{sec:solvability} for such an application.

\section{Main Results}\label{sect:Main-rsult}
Recall that $d\geq 2$. In this section, we present our main theoretical results on the influence function of the optimal transport quantiles. \Cref{Thm:InfluenceQuantiles-main} establishes the existence and uniqueness of this influence function and characterizes it in terms of a second-order elliptic PDE. \Cref{thm:Unbounded-IF} shows that the influence function admits a pole-type singularity around the point of evaluation $z$, thus the transport-based quantile function is infinitesimally sensitive to  \emph{inliers} --- points $x$ such that $\BF_P(x)$ is close to $z$. An interesting side product of these main proofs is  \Cref{prop:shape} which characterizes the size of the preimage of the perturbation point $x_0$ in $\Omega_\mu$ for small perturbation weight $t $. We believe this result to be interesting in its own right.

\begin{theorem}[Existence of the Influence Function]\label{Thm:InfluenceQuantiles-main}
    Let  $P$ and $\mu$ be regular probability measures according to~\Cref{defn:regular-measures}. Fix $x_0\in \Omega_P$ and define $z_0 = \BF_P(x_0)$  and $P_t=(1-t)P+ t \delta_{x_0} $ for $t \in [0,1]$.  Then the following holds: 
    \begin{enumerate}[label = (\roman*)]
        \item  There exists a function
    $$G_{x_0}\in \cC^{2}_{Loc}({\rm int}(\Omega_\mu) \setminus \{z_0\}) \cap  L^q(\Omega_\mu),$$ 
    for all $q \in \left(1,\; \frac{d(d+2)}{d^{2}+2d-4}\right)$, such that, for every open set $A$ contained in a compact subset of $ {\rm int}(\Omega_\mu) \setminus \{z_0\}$, there exists $t_0$ such that $\BQ_{P_t}\in \cC^{1}(\overline{A})$
    $$  \left\| \frac{ \BQ_{P_t}-\BQ_{P}}{t} -\nabla G_{x_0}\right\|_{\cC^{1}(\overline{A})} \longrightarrow 0, \quad {\text{as}}\quad t\downarrow 0.  $$
   In addition, writing $ \nabla\xi_t=\frac{ \BQ_{P_t}-\BQ_{P}}{t} $,  the convergence $\xi_t\to G_{x_0}$ holds in $L^q(\Omega_\mu)$ modulo constants. 
    \item  Furthermore, \(G_{x_0}\) is the unique function in \(L^q(\Omega_\mu)\) such that 
\begin{equation}\label{eq: compatibility condition}
\int_{\Omega_\mu}  G_{x_0} \dd \mu=0
\end{equation}
and 
\begin{equation}
    \label{eq:weak-sense}
    \int_{\Omega_\mu}
G_{x_0}\,
\operatorname{div}\!\left(
    f_\mu [\nabla \BQ_P]^{-1}\nabla \varphi
\right) \dd z
=
\int_{\Omega_\mu}\varphi\dd\mu-\varphi(z_0),
\end{equation}
for every \(\varphi\in W^{2,\frac{q}{q-1}}(\Omega_\mu)\) satisfying the homogeneous
Neumann condition
\[
\left\langle
    f_\mu[\nabla \BQ_P]^{-1}\nabla\varphi,
    n_\mu
\right\rangle = 0
\qquad\text{on }\partial\Omega_\mu 
\]
in the trace sense. 
    \end{enumerate}
\end{theorem}
We call $\nabla G_{x_0}$ \emph{the influence function} of the transport map $\BQ_P$ at $x_0$. It quantifies the local sensitivity of the optimal transport map under a point mass perturbation at $x_0$, as in classical robust statistics \citep{hampel1974influence,HuberRonchetti2009}.  \Cref{Thm:InfluenceQuantiles-main} states that this influence function of the optimal transport map exists and is identified by the solution of~\eqref{eq:PDE-main}.

Below we remark on a few fine-grained details about this result. 
\begin{remark}\label{rmk:existence}
\begin{enumerate}[label = (\alph*)]
\item The range $q \in \left(1, d(d+2) /\left(d^{2}+2d-4\right)\right)$ of function norms in ~\Cref{Thm:InfluenceQuantiles-main}~(i) is specifically  chosen to obtain a uniform $L_q(\Omega_\mu)$ estimate of the influence function; see \Cref{Prop:Lq-estimate}.  In fact, the influence function $\nabla G_{x_0}$ is integrable in the $L^q$ sense for all $q\in (1, d/(d-1))$, and the $L^q(\Omega_\mu)$ norm of $G_{x_0}$, denoted $\|G_{x_0} \|_{L^q(\Omega_\mu)}$ can be controlled uniformly over all $x_0 \in \Omega_P$; see~\Cref{lemma:Lp-estimates-of-G}.
\item Assume that $z_0={\bf F}_P(x_0)\in {\rm int}(\Omega_\mu)$. 
We note that $\nabla \BQ_P$ is symmetric as the Hessian of a convex potential function. By \Cref{lemma:Lp-estimates-of-G}, $G_{x_0}$ solves
\begin{equation}
    \label{eq:PDE-main}
 {\rm div} (f_\mu [\nabla\BQ_{P}]^{-1} \nabla G ) =\mu-\delta_{z_0} 
\end{equation}
in the distributional sense.

\item Moreover, we may formally take $g=G_x$ (resp.\ $g=G_{\BQ_P(z)}$) in~\eqref{eq:weak-sense} with $z_0=z$ (resp.\ $z_0=\BF_P(x)$), and obtain the identity %
\begin{equation}\label{eq:symmetry}
    G_x(z) =  G_{\BQ_P(z)}(\BF_P(x)),\quad x\neq \BQ_P(z).
\end{equation}
See~\cref{lem:green-symmetry} for a rigorous proof. The form above suggests that one can compute the influence function at a fixed point $z\in {\rm int}(\Omega_\mu)$ with the following steps:
1) determine the function $G_{\BQ_P(z')}$ for $z'$ near $z$, that is, the solution of \eqref{eq:PDE-main} with perturbation point $\BQ_P(z')$;
2) use the symmetry in~\eqref{eq:symmetry} to recover $G_{(\cdot)}(z')$ for all $z'$ near $z$;
3) approximate $\nabla G_{(\cdot)}(z)$ using a finite-difference scheme.  Finally, we obtain the influence function by noting that $\BI(\,\cdot\,;\BQ_{P}(z)) =\nabla G_{(\cdot)}(z)$.

\end{enumerate}
\end{remark}

In the one-dimensional setting,   the exact formula for the influence function of the quantile $\BQ_P(z)$ (with the standard reference $\mu={\rm Uniform}(0,1)$) is
\begin{equation*}
    \BI(x_0;\BQ_P(z))=
\frac{1}{f_P(\BQ_P(z))}\Bigl\{z-\mathbf{1}\{x_0\le \BQ_P(z)\}\Bigr\}, 
\end{equation*}
which is bounded but has a jump discontinuity at $x_0=\BQ_P(z)$ \cite[p. 56]{HuberRonchetti2009}. This formula satisfies the PDE characterization in ~\eqref{eq:PDE-main} in one dimension. 

While the influence function of a one-dimensional quantile is uniformly bounded, our next result (\Cref{thm:Unbounded-IF}) shows that for $d \ge 2$, the  influence function $\BI(x_0;\BQ_P(\cdot))$ admits a pole-type singularity at $\BF_P(x_0)$. In other words, it is both  unbounded around and discontinuous at $\BF_P(x_0)$. We note that this result seems counterintuitive in light of the good breakdown properties of transport quantiles obtained in \cite{paindaveine:passeggeri2026,Avella-Gonzalez.26,gonzalezsanz:avellamedina2026}. We remark, however, that while unbounded influence functions are commonly associated with $0$ breakdown point functionals,  unbounded influence functions do not preclude a high breakdown point and there are well-known examples where this occurs for one-dimensional rank-based location estimators \cite[Examples~3.9 and 3.11]
{HuberRonchetti2009}.

\begin{theorem}[Unboundedness of the Influence Function]\label{thm:Unbounded-IF}   Let  $P$ and $\mu$ be regular. For $z\in {\rm int}(\Omega_\mu)$, the following holds: 
\begin{enumerate}
    \item There  exists a constant $C>0$ such that, for every $x_0\in \Omega_P\setminus\{\BQ_P(z)\}$, 
$$   \|\BI(x_0;\BQ_{P}(z))\| \leq  \frac{C}{\|z-\BF_P(x_0)\|^{d-1}} .$$
\item There exist $r,C>0$ such that, for every $x_0\in \BB_{r}(\BQ_P(z))\setminus \{\BQ_P(z)\}$,   
$$   \frac{1}{C\|z-\BF_P(x_0)\|^{d-1}} \leq \|\BI(x_0;\BQ_{P}(z))\| \leq  \frac{C}{\|z-\BF_P(x_0)\|^{d-1}} .$$
\end{enumerate}  
\end{theorem}

Theorem~\ref{thm:Unbounded-IF} gives two different types of estimates.
The upper bound in part~{\rm (i)} is global: it holds for every
$x_0\in\Omega_P\setminus\{\BQ_P(z)\}$, including perturbation points
for which $\BF_P(x_0)$ lies on $\partial\Omega_\mu$. In contrast, the
two-sided comparison in part~{\rm (ii)} is local. More precisely, as
$x_0\to\BQ_P(z)$, equivalently as $\BF_P(x_0)\to z$, we have
\[
\|\BI(x_0;\BQ_P(z))\|
\asymp
\frac{1}
{\|z-\BF_P(x_0)\|^{d-1}}.
\]
Thus, the transport-based quantile $\BQ_P(z)$ is particularly
sensitive to perturbations at inliers, namely points $x_0$ for which
$\BF_P(x_0)$ is close to $z$. The singularity in part~{\rm (ii)} is an interior singularity.
Indeed, since $z\in{\rm int}(\Omega_\mu)$, whenever $x_0$ is
sufficiently close to $\BQ_P(z)$, the point
$z_0=\BF_P(x_0)$ remains a fixed positive distance from
$\partial\Omega_\mu$. In this regime, $G_{x_0}$ satisfies a uniformly
elliptic equation with an interior Dirac source at $z_0$, and the
full-space frozen fundamental solution is therefore the appropriate
local model. 

\begin{figure}
    \centering
        \begin{subfigure}[t]{0.48\linewidth}
        \centering
        \includegraphics[width=\linewidth]{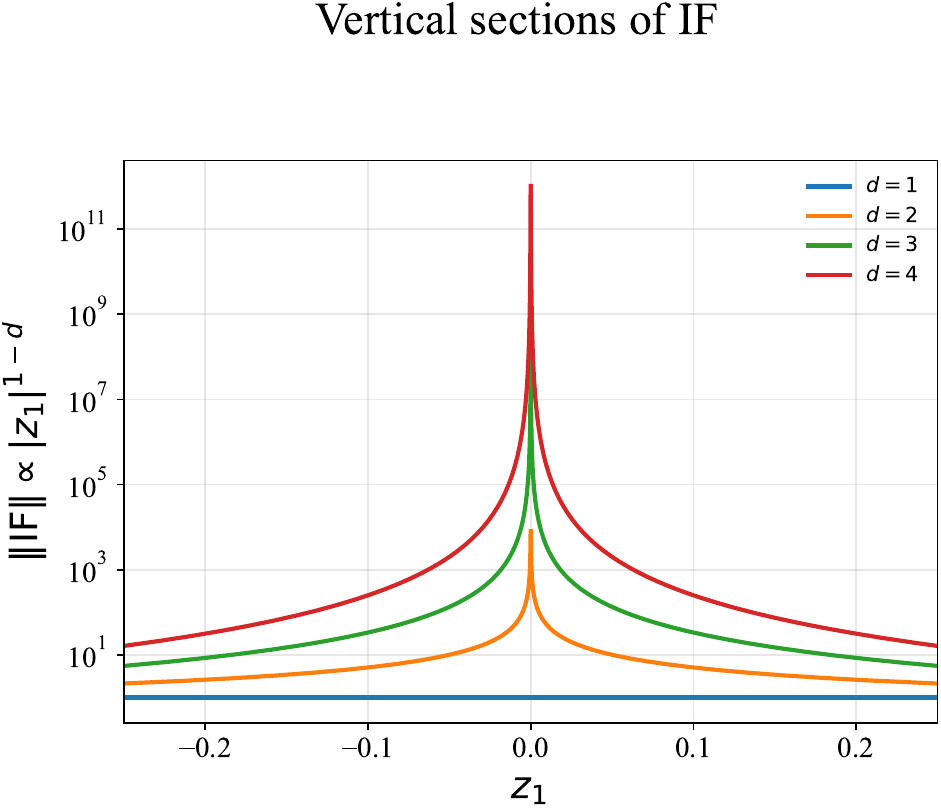}
        \label{fig:singularity-sections}
    \end{subfigure}
        \hfill
    \begin{subfigure}[t]{0.48\linewidth}
        \centering
        \includegraphics[width=\linewidth]{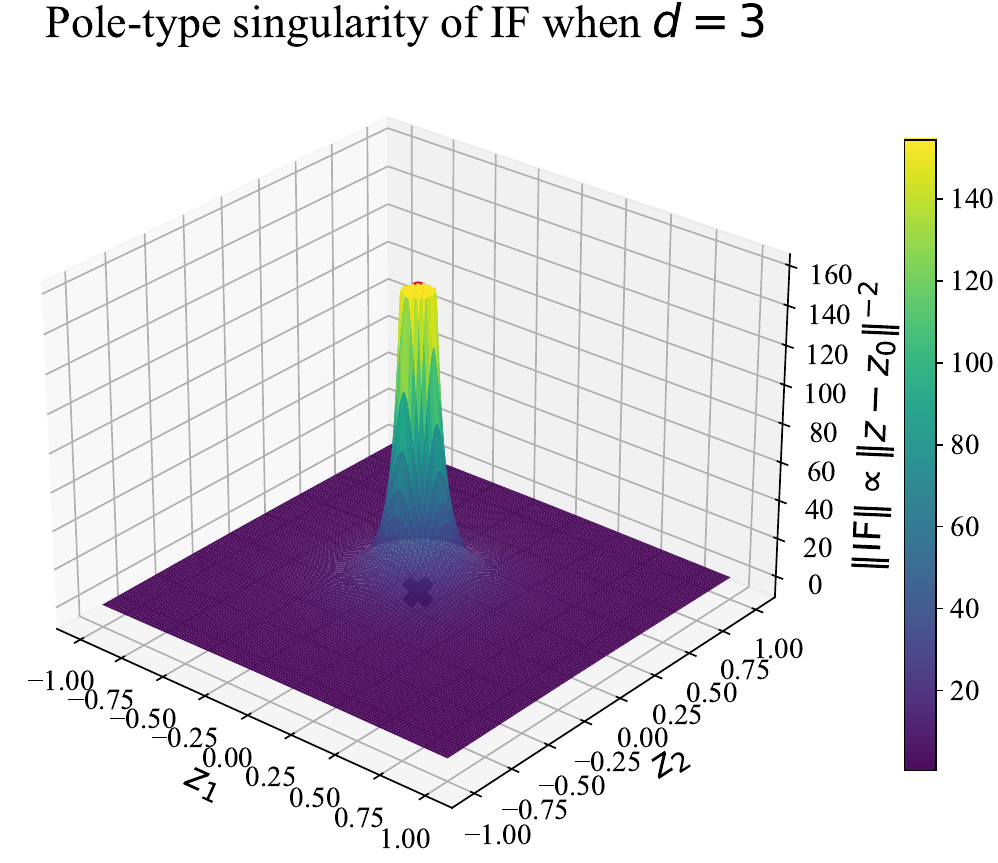}
        \label{fig:singularity-main-d=3}
    \end{subfigure}
    \caption{The left panel shows the norm of the influence function of $\BQ_P(0)$ computed in \Cref{prop:Example-uniform} along the line $(z_1,0,\ldots, 0)$, for $d=1,2,3,4$. The right panel shows the norm of the same influence function in dimension $d=3$. Notice that the influence function is $0$ at $z_0=0$, but has a singularity around it. }
    \label{fig:sensitivity_and_singularity}
\end{figure}
\begin{remark}
\begin{enumerate}[label = (\alph*)]
    \item When $x_0=\BQ_P(z)$, we have $\BQ_{P_t}(z) = \BQ_P(z)$ for all $t\in [0,1]$. Therefore,~$ \BI(x_0;\BQ_{P}(z)) = 0$ if $x_0 =  \BQ_{P}(z)$. In other words, perturbing the OT quantile $\BQ_P(z)$ exactly at the same point does not change its location. 
\item By direct computation,  we derive that the influence function is also unbounded in the $L_2$-sense, i.e., $\BI(\cdot;\BQ_P(z)) \not\in L^2(P)$ for any $z$. See~\Cref{coro:IF-property} in~\Cref{sec:asymptotic}. 
\item The mapping $\cG:\Omega_P \times \Omega_\mu\to \R$ given by
\begin{equation} \label{rmk:green-function}
    \cG(x,z)=G_x(z), 
\end{equation}
 where $G_x$ is given in \Cref{Thm:InfluenceQuantiles-main}, can be recognized as the \textit{Green function} associated with this uniformly elliptic PDE, which is well known to exhibit a pole-type singularity; see~\cite{LittmanStampacchiaWeinberger1963Green,GruterWidman1982GreenFunction} for detailed discussions. We also note that the connection between the PDE and the Green function is briefly mentioned in \cite[Lemma~35]{manole2023central}.
\end{enumerate}
\end{remark}

Let $\varphi_{t,x_0}$ be the Brenier potential of the transport map $\BQ_{P_t}$, and let $K_{t,x_0}:=(\partial \varphi_{t,x_0})^{-1}(x_0)$ be its preimage in $\Omega_\mu$. We now characterize the size of the preimage of $x_0$ under $\BQ_{P_t}$.  The next result shows, for small $t$, the set $K_{t,x_0}$ behaves like a ball for $d\geq 3$. This  shape characterization is a byproduct of the proof of \Cref{Thm:InfluenceQuantiles-main}.  
\begin{proposition}\label{prop:shape}
Let $P$ and $\mu$ be regular probability measures and $x_0\in \Omega_P$. Let $Z_{t,x_0}$ be the barycenter of the set $K_{t,x_0}$ defined above. There exist $R,t_0,C>0$ such that, for every $t\in(0,t_0)$, the following holds:
\begin{enumerate}[label=(\roman*)]
    \item If $d\ge 3$, then
    $\BB_{R^{-1}t^{\frac{1}{d}}}(Z_{t,x_0}) \subset K_{t,x_0} \subset \BB_{Rt^{\frac{1}{d}}}(Z_{t,x_0})$ and $\inf_{w\in K_{t,x_0}}\| \mathbf{F}_P (x_0)-w \|\leq C t^{\frac{1}{d}} $. 
    \item If $d=2$, then $K_{t,x_0}\subset \BB_{R\sqrt{t|\log(t)|}}(Z_{t,x_0})$ and $\inf_{w\in K_{t,x_0}}\| \mathbf{F}_P (x_0)-w \|\leq C \sqrt{t|\log(t)|} $.
\end{enumerate}
\end{proposition}

\section{Examples}\label{sec:examples}
In this section, we present a simple example based on the uniform measure that illustrates various key properties of the influence function.  
Using this result, we further formulate a conjecture that the empirical OT quantile $\BQ_{P_n}(z)$ converges to its population counterpart
$\BQ_P(z)$ at a nonparametric rate, and that the limiting distribution is heavy-tailed. 

\subsection{An Explicit Example}\label{sec:Explicit}
In this section, we explicitly compute the transport quantile and its influence function when $\mu=\cU(\BB_1(0))$ and
$P_t=(1-t)\cU(\BB_1(0))+t\delta_0$ for $t\in(0,1)$. The first two points of the following result characterize the optimal transport map to the point-perturbed measure and bound its $L^2$ deviation from the identity map. Together, they motivate the proof strategy for \Cref{Thm:InfluenceQuantiles-main}, which builds on the critical $L^2$ estimate (see \Cref{thm:L2-estimate}). Moreover, as the third point makes clear, the influence function exists and exhibits a pole-type singularity at the origin; see \Cref{fig:sensitivity_and_singularity}.
\begin{proposition}\label{prop:Example-uniform}
Let $\mu=\cU(\BB_1(0))$. Then the following hold:
    \begin{enumerate}
        \item For $t\in (0,1)$, define the mapping
    \begin{equation}
    \begin{aligned}
    T_t(z) = \begin{cases}
    0 , &\text{if } z \in \BB_{t^{\frac{1}{d}}}(0), \\
    \dfrac{z}{\|z\|} \left(\dfrac{\|z\|^d - t}{1 - t}\right)^{\frac{1}{d}}, & \text{otherwise}.
    \end{cases}
    \end{aligned}
    \end{equation}
    Then $T_t$ is the optimal transport map pushing ${\cal U}(\BB_1(0))$ forward to $(1-t) {\cal U}(\BB_1(0)) + t \delta_0 $, where ${\cal U}(\BB_1(0))$ denotes the uniform probability measure on the unit ball; 
         \item For sufficiently small $t$, the $L^2$ estimate holds:
        $$ \int_{\BB_1(0)} \|T_t(z) -z\|^2 \dd z \lesssim \begin{cases}
      t^{1+\frac{2}{d}} & {\rm if} \ d\geq 3\\
t^{2}{|\log(t)|}     & {\rm if}\ d=2.
  \end{cases}   $$

    \item For any $x, z \in \BB(0,1)$ with $x\neq z $,  the influence function exists and is given by
    \begin{equation*}
        \BI(x;z)=
-\dfrac{1}{d}\dfrac{z-x}{\|z-x\|^{d}}
+
\int_0^{\|x\|}
\left(\left(s\langle z,x^\star\rangle - 1\right)
\frac{s z - x^\star}{\|s z - x^\star\|^{d+2}} -\frac{x^\star}{d\|s z - x^\star\|^{d}}\right) \dd s
+\frac{1}{d} z,
    \end{equation*}
    where $x^\star := x/\|x\|$ for $x\neq 0$ and $x^\star =0$ otherwise. 
    When $z=0$, it follows that
    \begin{equation}\label{eq:IF-uniform-origin}
         \BI(x;0) = \begin{cases}
            0,&\quad x=0\\
             \frac{1}{d}\left( \|x\|^{-d}- (1-d)\right)x,&\quad x\neq 0.
         \end{cases}
    \end{equation}
    \end{enumerate}
\end{proposition}
\begin{remark}
When $d =1$, the right-hand side of~\eqref{eq:IF-uniform-origin} becomes $ {\rm sgn}(x)$, recovering the influence function of the median of $\cU(\BB(0,1))$.
\end{remark}

\subsection{Asymptotic Representation of the Empirical OT Quantile}\label{sec:asymptotic}

Let $\BI(\cdot;\BQ_P(z))$ be the influence function identified in \Cref{Thm:InfluenceQuantiles-main}.   If one had $\sqrt{n}$-asymptotic normality, one would expect  the following asymptotic linear representation to hold
$$ \sqrt{n}\bigl(\BQ_{P_n}(z)-\BQ_P(z)\bigr)
=
\frac1{\sqrt n}\sum_{i=1}^n \tilde{\BI}(X_i;\BQ_P(z)) + o_P(1).$$
The following result shows that the influence function  is not square-integrable, providing evidence
against a standard $\sqrt{n}$-asymptotically linear representation with
a square-integrable influence function. This is an immediate consequence of~\Cref{Thm:InfluenceQuantiles-main} and~\Cref{thm:Unbounded-IF}. 
\begin{corollary}\label{coro:IF-property}
Assume the setting of~\Cref{Thm:InfluenceQuantiles-main}. Let $z\in {\rm int}(\Omega_\mu)$. Then the following hold: 
    \begin{enumerate}[label = (\roman*)]
        \item $\BI(\cdot;\BQ_P(z)) \in L^{q}(P)$ for any $q\in (1,d/(d-1))$;
        \item $\BI(\cdot; \BQ_P(z)) \not\in L^2(P)$;
        \item $\int_{\Omega_P}\BI(x;\BQ_P(z))\, dP(x)=0$.
    \end{enumerate}
\end{corollary}
\Cref{coro:IF-property} shows that $\|\BI(\cdot;\BQ_P(z))\|_{L^2(P)} = \infty$, due to the singular behavior of the influence function, thus $\BI(\cdot;\BQ_P(z))$ is not an efficient influence function in the sense of the convolution theorem \citep[Theorem 25.20]{VdV2000} and the local asymptotic minimax theorem \citep[Theorem 25.21]{VdV2000}. This provides evidence that empirical transport quantiles do not admit a standard $\sqrt n$-asymptotically linear expansion with a square-integrable influence function.

Instead, we hypothesize a different scaling.
\begin{conjecture}[Asymptotic linear form]\label{conjecture:AL-form}
Under the setting of~\Cref{Thm:InfluenceQuantiles-main}, we conjecture that 
    \begin{equation}\label{eq:d-rate}
  r_n\bigl(\BQ_{P_n}(z)-\BQ_P(z)\bigr)
    =
    \frac{r_n}{n}\sum_{i=1}^n \BI(X_i;\BQ_P(z)) + o_{\PP}(1),
\end{equation}
where
\[
r_n = \begin{cases}
\sqrt{n/\log(n)},&\quad d =2\\
    n^{1/d},&\quad d\geq 3. 
\end{cases}
\]
\end{conjecture}

Our conjecture is motivated by stable-limit theory; see
\cite{Mijnheer1975Stable,SamorodnitskyMurad1994}. Indeed,
\Cref{thm:Unbounded-IF} implies the two-sided tail estimate
\[
\PP_{X\sim P}\!\left(
\|\BI(X;\BQ_P(z))\|>\lambda
\right)
\asymp
\lambda^{-\gamma},
\qquad
\gamma=\frac{d}{d-1}.
\]
This identifies $\gamma$ as the natural candidate tail index and suggests
the normalization $n^{1/\gamma}$ for sums of the influence-function
terms. %
This tail behavior suggest that under appropriate conditions one could expect a stable domain-of-attraction theorem such as \cite[Theorem~1.8.1]{SamorodnitskyMurad1994}, whereby there exist centering constants $b_n$ and normalizing constants $a_n=n^{1/\gamma}L_0(n),$ 
where $L_0$ is slowly varying, such that 
\[
a_n^{-1}
\left(
\sum_{i=1}^n
\|\BI(X_i;\BQ_P(z))\|
-b_n
\right)
\overset{d}{\longrightarrow}
Z_\gamma,
\]
where $Z_\gamma$ is a $\gamma$-stable random variable. This scalar
heuristic motivates the scaling in our conjecture. Establishing the
corresponding vector-valued stable limit would additionally require
multivariate regular variation of $\BI(X;\BQ_P(z))$.

Proving the pointwise expansion~\eqref{eq:d-rate} is challenging because the influence function has a singularity. The next proposition establishes a weaker but rigorous version: the conjectured linearization holds after testing against any $\cC^{1,\eta}$ function.

\begin{proposition}\label{prop:asym-linear-form-test}
   Under the setting of~\cref{Thm:InfluenceQuantiles-main}, it follows that,
    \begin{equation*}\label{eq:integrated-linear-form}
         r_n\left|\int_{\Omega_\mu} \inner{\BQ_{P_n}-\BQ_P}{[\nabla \BQ_P]^{-1}\nabla g}\dd \mu -  \int_{\Omega_\mu}  \left\langle \frac{1}{n}\sum_{i=1}^n \BI(X_i;\BQ_P(\cdot)), [\nabla \BQ_P]^{-1}\nabla g\right\rangle \dd \mu  \right|  = o_{\PP}(1)
    \end{equation*}
      for all $g\in \cC^{1,\eta}(\Omega_\mu)$, $\eta\in (0,1]$.
 Moreover, for $d =2,3$,  $\eta \in ((d-2)/2,1]$, it holds that for any $g\in \cC^{1,\eta}(\Omega_\mu)$,
    \[
    \sqrt{n}\int_{\Omega_\mu} \inner{\BQ_{P_n}-\BQ_P}{[\nabla \BQ_P]^{-1}\nabla g}\dd \mu \overset{d}{\to } \cN(0,\sigma_g^2),\qquad \sigma_g^2:={\rm Var}_\mu(g(Z))<\infty.
    \]
\end{proposition}

To further corroborate this conjecture, \cref{fig:AN} suggests that the hypothesized scaling is indeed correct when $d=3$. We vary the sample sizes $n \in \{250, 1000, 4000, 16000\}$
and observe that the distribution of the empirical OT median becomes increasingly
aligned with the linear statistic provided by the influence function 
in \Cref{conjecture:AL-form} as $n$ increases. This gives  numerical support to the conjecture that this asymptotic linear form
characterizes the empirical distribution of the OT quantile, which exhibits
exploding variance, heavy tails, and a convergence rate of $r_n^{-1}$.

The finite sample discrepancy in the variance may stem from the fact that our estimator $\BQ_{P_n}$ is computed by  solving the two-sample problem, i.e., the discrete OT problem for the empirical distributions $\mu_n$ and $P_n$, thereby introducing an additional finite-sample approximation term.

\begin{figure}
    \centering
        \includegraphics[width=1\linewidth]{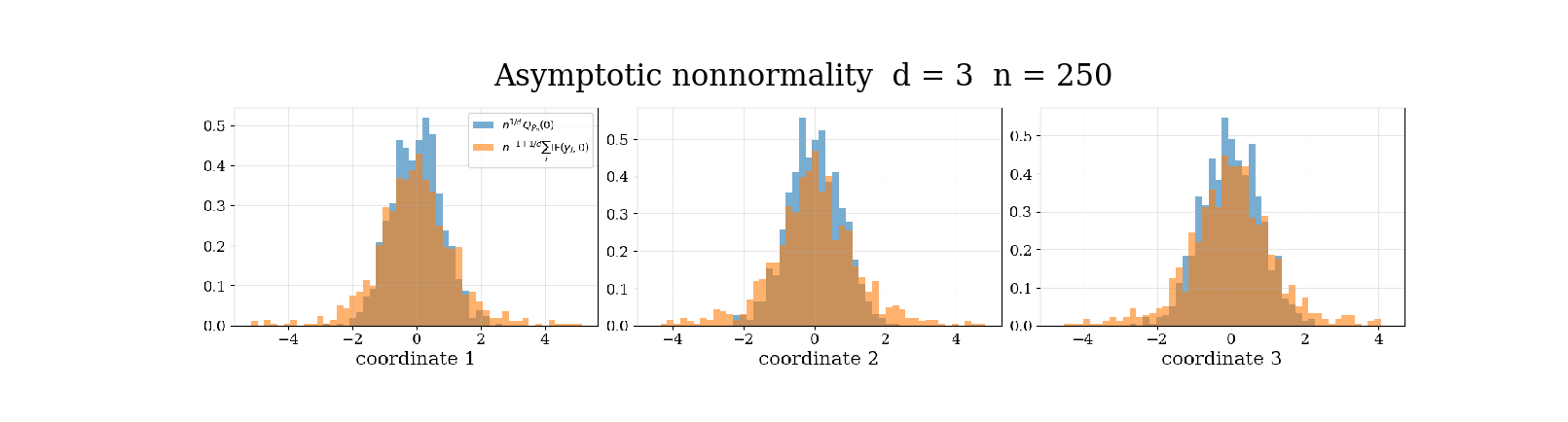}
\includegraphics[width=1\linewidth]{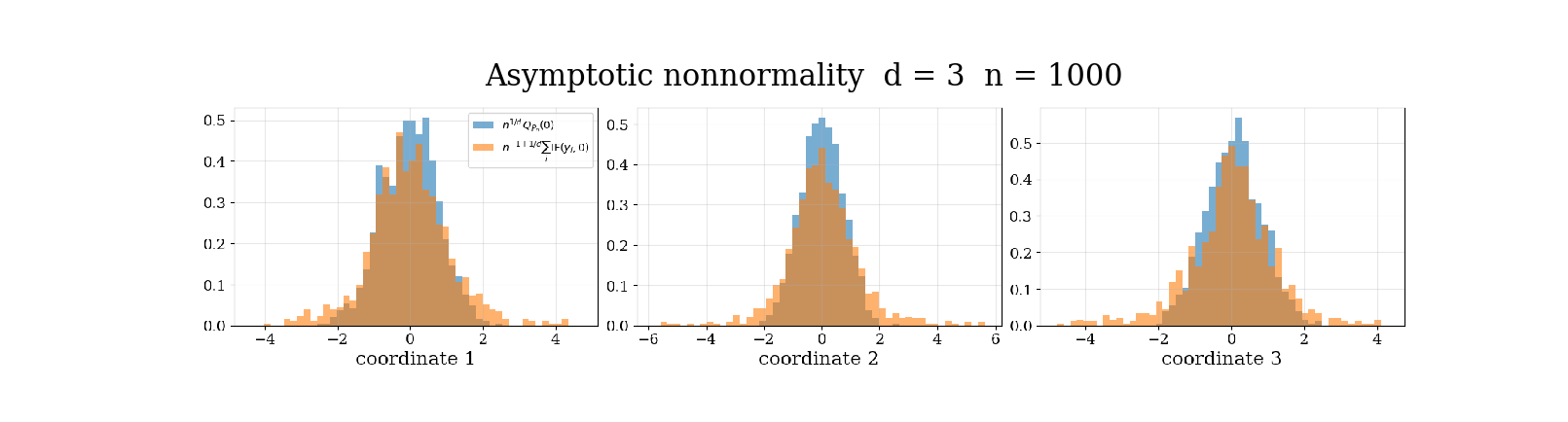}
    \includegraphics[width=1\linewidth]{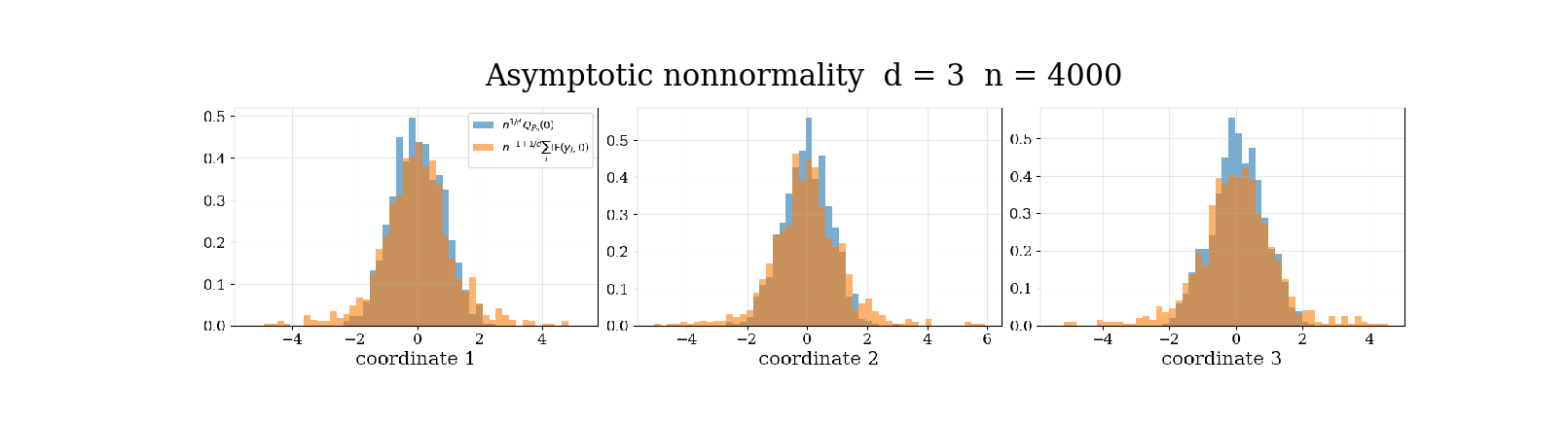}
    \includegraphics[width=1\linewidth]{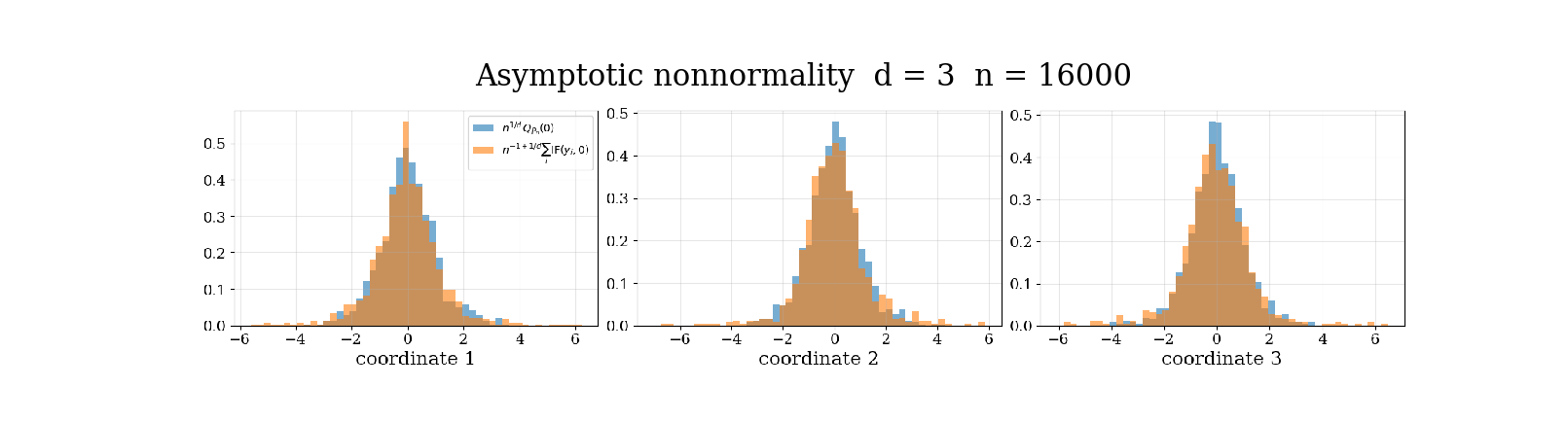}
\caption{Empirical distributions of $n^{1/d}\BQ_{P_n}(0)$ (blue) and the
         linear statistics $n^{-1+1/d}\sum_i \BI(X_i;\BQ_P(0))$ (orange),
         for $n \in \{250, 1000, 4000, 16000\}$ (rows 1--4), $d=3$, and $R=1000$
          replications.  The two distributions remain stable
         as $n$ grows and provide finite-sample evidence consistent with the scaling in~\cref{conjecture:AL-form}.
         While the OT median looks normal for small sample sizes, as $n$ increases
         the distribution of the empirical OT medians becomes aligned with the distribution of linear statistics, which show heavy tails. This  behavior
         provides empirical evidence for \cref{conjecture:AL-form}.}
    \label{fig:AN}
\end{figure}

To further assess the quality of the asymptotic approximation in \Cref{conjecture:AL-form}, we focus on the case where $n = 16000$, $d = 3$, and $R = 1000$  replications, and compare
the scaled empirical quantile $    A_n = r_n\bigl(\BQ_{P_n}(0) - \BQ_P(0)\bigr)$ against the linear form $    B_n = \frac{r_n}{n}\sum_{i=1}^n \BI(X_i;\BQ_P(0))$ from \Cref{conjecture:AL-form} using two types of tests.

\begin{enumerate}
    \item \textbf{Two-sample Kolmogorov--Smirnov test (KS).}
    For each coordinate $j \in \{1,2,3\}$, we apply the two-sample Kolmogorov--Smirnov test to the
    empirical distributions of $(A_n)_j$ and $(B_n)_j$ across the $R$ replications.
    Let $F_{(A_n)_j}$ and $F_{(B_n)_j}$ denote the corresponding empirical CDFs.
    The test statistic is
    \begin{equation}
        D_j = \sup_{t \in \mathbb{R}} \left| F_{(A_n)_j}(t) - F_{(B_n)_j}(t) \right|,
    \end{equation}
    which measures the largest pointwise discrepancy between the two empirical
    distributions. The null hypothesis $H_0: (A_n)_j \overset{d}{=} (B_n)_j$
    is rejected for large values of $D_j$.  A large $p$-value means that the test does not detect a coordinatewise difference between the two empirical distributions.

    \item \textbf{Jarque--Bera normality test (JB).}
    The Jarque--Bera test \citep{jarque1980efficient} tests whether a sample has skewness and excess kurtosis
    consistent with a Gaussian distribution. For a sample $x_1, \ldots, x_R$, the
    test statistic is
    \begin{equation}
        \mathrm{JB} = \frac{R}{6}\left( \hat{S}^2 + \frac{(\hat{K}-3)^2}{4} \right),
    \end{equation}
    where
    \begin{equation}
        \hat{S} = \frac{\frac{1}{R}\sum_{b=1}^R (x_b - \bar{x})^3}
                       {\left(\frac{1}{R}\sum_{b=1}^R (x_b - \bar{x})^2\right)^{3/2}},
        \qquad
        \hat{K} = \frac{\frac{1}{R}\sum_{b=1}^R (x_b - \bar{x})^4}
                       {\left(\frac{1}{R}\sum_{b=1}^R (x_b - \bar{x})^2\right)^{2}}
    \end{equation}
    are the sample skewness and kurtosis respectively. Under $H_0$: the series is
    normally distributed, $\mathrm{JB} \xrightarrow{d} \chi^2(2)$ as $R \to \infty$.
    We apply the test separately to $(A_n)_j$ and $(B_n)_j$ on each coordinate $j$.
\end{enumerate}

The per-coordinate diagnostics at $n = 16000$ are reported in \Cref{tab:AN_diag}.
     The coordinatewise tests fail to reject on all three coordinates ($p \in \{0.37, 0.466, 0.164\}$) at the sample size $n = 16000$. These results are consistent with %
 \Cref{conjecture:AL-form}. Both $A_n$ and $B_n$ individually
fail the JB test on all coordinates (all $p < 10^{-7}$), providing statistical evidence against Gaussian marginals. This is not a contradiction: \Cref{conjecture:AL-form} posits that $A_n$
and $B_n$ share the same limiting distribution, not that this limit is
Gaussian. The \emph{non-Gaussianity} of $B_n = \frac{r_n}{n}\sum_i \BI(X_i;\BQ_P(0))$
at $n=16000$ reflects the slow convergence for the influence function
terms, whose tails are heavy near the origin of $\BB(0,1)$ by \Cref{thm:Unbounded-IF}.

\begin{table}[h]
\centering
\begin{tabular}{cccc}
\toprule
$j$ & KS $p$-val ($A_n$ vs $B_n$) & JB $p$-val ($A_n$) & JB $p$-val ($B_n$) \\
\midrule
1 & $0.370$ & $7.64 \times 10^{-9}$  & $\approx 0$ \\
2 & $0.466$ & $1.27 \times 10^{-8}$  & $\approx 0$ \\
3 & $0.164$ & $1.09 \times 10^{-13}$ & $\approx 0$ \\
\bottomrule
\end{tabular}
\caption{Per-coordinate diagnostic tests at $n=16000$, $d=3$, $R=1000$ 
         replications. Here $A_n = r_n(\BQ_{P_n}(0)-\BQ_P(0))$ is the scaled empirical
         quantile and $B_n = r_n n^{-1}\sum_i \BI(X_i;\BQ_P(0))$ is the scaled
         linear form from \Cref{conjecture:AL-form}, with $r_n = n^{1/d}$ for $d=3$.
         KS: two-sample Kolmogorov--Smirnov test ($H_0$: $A_n \overset{d}{=} B_n$
         coordinatewise). JB: Jarque--Bera normality test ($H_0$: Gaussian marginal),
         applied separately to $A_n$ and $B_n$. The KS test fails to reject $H_0$
         on all coordinates, which is consistent with \Cref{conjecture:AL-form}; moreover, both
series exhibit non-Gaussian marginals, as suggested by the heavy-tailed
influence function established in \Cref{thm:Unbounded-IF}.}
\label{tab:AN_diag}
\end{table}

\section{Proofs}\label{sec:L-2and-Example-and-Shape}\label{sec:proofs}
This section is organized as follows. In~\Cref{Sect:L-2Estimates}, we establish the $L^2$ control of $\BQ_{P_t} - \BQ_P$ for sufficiently small $t$.  In  \Cref{sec:Lq-estimate},  we apply a 
linearization argument in~\Cref{lemma:test-linearization} to show that the sequence $\{(\BQ_{P_t} - \BQ_P)/t\}_{t\to 0}$ is uniformly bounded in $L^q$ for $q\in \left(1,\frac{d(d+2)}{d^2+2d-4}\right)$, which also characterizes the limit. 

 The $L^q$ estimate is improved to a local $L^\infty$ bound in \Cref{sec:holder} using the local maximum principle, which implies a local $\cC^{2,\alpha}$ bound via Schauder's estimates. 
 
 In~\Cref{sec:proof-main},  we combine our results to prove~\Cref{Thm:InfluenceQuantiles-main}. The proof of~\Cref{thm:Unbounded-IF} builds on the PDE characterization of the influence function and follows from a localization  over the regions that are away from the singularity point $z_0 =\BF_P(x)$. Finally, in \Cref{sec:proof-shape}, we present \Cref{prop:shape} and its proof using John's lemma, which may be of independent interest.

\subsection{The $L^2$ Estimates}\label{Sect:L-2Estimates}
In this section, we  prove an $L^2$ bound of the optimal transport map. As we illustrate in the roadmap in \Cref{fig:dependency-diagram}, the $L^2$ bound plays an instrumental role in all subsequent results. 
\begin{theorem}\label{thm:L2-estimate}
 Let  $P$ and $\mu$ be regular. Let $P_t = (1-t)P+ t \delta_{x_0}$ for some $x_0 \in \Omega_P$.
   For sufficiently small $t$, it follows that 
\begin{equation}
    \label{claim:L2bound}
   \left\|\BQ_{P_t}-\BQ_{P} \right\|_{L^2(\mu)}^2 \lesssim  \begin{cases}
      t^{1+\frac{2}{d}} & {\rm if} \ d\geq 3,\\
t^{2}{|\log(t)|}     & {\rm if}\ d=2.
  \end{cases} 
\end{equation}
\end{theorem}
The result establishes an $L^2$ stability bound for the optimal transport map under Huber contamination. In $L^2(\mu)$, the OT map is stable at a rate slightly faster than linear: there is no first-order $L^2$ effect, and the squared $L^2(\mu)$ displacement scales strictly better than $t$ as $t\to 0$.

The rest of the section is devoted to the proof of \Cref{thm:L2-estimate}. When $P, \mu$ are regular, \cite[Theorem~6]{Manole.2024.Aos} shows that
\begin{equation}
    \label{claim:L2bound-W2}
   \left\|\BQ_{P_t}-\BQ_{P} \right\|_{L^2(\mu)}^2 \lesssim  \mathcal{W}_2^2(P_t,P),\quad \forall t>0,
\end{equation}
where the implicit constant depends only on the regularity of $\mu$ and $P$. 

To show  \Cref{thm:L2-estimate}, we require \eqref{claim:L2bound-W2} and the following lemma. 
As a key ingredient, we would like to extend \Cref{prop:Example-uniform} to any point of perturbation in $\overline{\BB}_1(0)$, not necessarily only at the origin.

\begin{lemma}\label{lemma:W2-estimate}
For sufficiently small  $t$, the following holds
  $$ \mathcal{W}_2^2(P_t,P)\lesssim \begin{cases}
      t^{1+\frac{2}{d}} & {\rm if} \ d\geq 3,\\
t^{2}{|\log(t)|}     & {\rm if}\ d=2.
  \end{cases}  $$
\end{lemma}
\begin{proof}[Proof of \Cref{lemma:W2-estimate}]
    Let $T$ be the OT map from $P$ to the uniform  probability measure $\nu$ on the unit ball, $\dd\nu= (\cL_d(\BB_1(0)))^{-1}{\bf 1}_{\BB_1(0)}\dd z$. 
    
    Write $z_0=T(x_0)\in \overline{\BB}_1(0)$.  From  the theory of Caffarelli (see \cite{Caffarelli.96}), we know that $T^{-1}$ is Lipschitz with constant $L$.  If a transport map $R_t$ from $\nu$ to $(1-t)\nu + t \delta_{z_0}$ satisfies
\begin{equation} \label{cond:Rt}
        \int_{\BB_1(0)}\|R_t(z)-z\|^2\dd z\lesssim \begin{cases}
      t^{1+\frac{2}{d}} & {\rm if} \ d\geq 3,\\
t^{2}{|\log(t)|}    & {\rm if}\ d=2,
  \end{cases}
\end{equation}
we obtain the desired result
\begin{align*}
     \mathcal{W}_2^2(P, (1-t)P+ t \delta_{x_0})&\leq \int_{\Omega_P} \|T^{-1}(R_t (T (x))) -x\|^2\,P(\rd x )\\
     &= \int_{\BB_1(0)} \|T^{-1}(R_t (z)) -T^{-1}(z)\|^2 \dd z \leq L^2\int_{\BB_1(0)}\|R_t(z)-z\|^2\dd z\,.
\end{align*}

We now proceed to  construct a map $R_t$ that satisfies~\eqref{cond:Rt}. Since $\mathbb{B}_1(0)$ has  Lipschitz boundary, it satisfies a uniform interior cone condition. Hence,   for $t$ sufficiently small, there exists a bounded closed convex cone $\mathbb{V}$ with angle $\theta>0$, radius $r(t)>0$ and vertex $z_0$ such that $\mathbb{V}\subset \overline{\BB}_1(0)$ and $\nu(\BV) = t$.

As $\nu$ is uniform over $\BB_1(0)$, $\nu(\BV)=t$ implies that  $r(t) \asymp t^{\frac{1}{d}}$.  Moreover, since both measures  have total mass $ 1-t $, there exists an OT map $T_t$ from $\nu\indic{\BB_1(0)\setminus \BV}$ to $(1-t)\nu$. 

Define the Borel map $$R_t(z) =\begin{cases}
        z_0 & {\rm if}\ z\in \mathbb{V},\\
        T_t(z) & {\rm if}\ z\in  \BB_1(0)\setminus \mathbb{V}.
    \end{cases}$$
Then, $R_t$ pushes forward $\nu$ to $(1-t)\nu+ t \delta_{z_0}$. Furthermore,
\begin{equation}\label{eq:Rt}
    \begin{aligned}
    \int_{\BB_1(0)}\|R_t(z)-z\|^2\dd z  =& \int_{\BV} \|z_0-z\|^2 \dd z + \int_{\BB_1(0)\setminus \mathbb{V}} \|T_t(z) - z\|^2\dd z\\
    = & \int_{\BV} \|z_0-z\|^2 \dd z + (1-t) \cW_2^2\left(\frac{\nu}{1-t}\indic{\BB_1(0)\setminus \BV}, \nu\right).
\end{aligned}
\end{equation}
Since $\BV$ has radius $r(t)\asymp t^{\frac{1}{d}}$ and angle $\theta$, we know that 
\begin{equation}\label{eq:V}
        \int_{\BV} \|z_0-z\|^2 \dd z \lesssim \int_{0}^{r(t)} r^{d+1} \dd r \lesssim t^{1+\frac{2}{d}}.
\end{equation}
 Denote by $\nu_t = \frac{\nu}{1-t}\indic{\BB_1(0)\setminus \BV}$. By \cite[Theorem~2.1]{Peyre-ESAIM}, the $\mathcal{W}^2_2$ distance is dominated by $\|\cdot\|_{H^{-1}(\BB_1(0))}$. More precisely, 
\[
\cW_2\left(\nu_t, \nu\right) \lesssim  \|\nu_t-\nu\|_{H^{-1}(\BB_1(0))} :=\sup_{\|\nabla h\|_{L^2(\BB_1(0))}\leq 1} \left|\int h \dd(\nu_t -\nu) \right|\,.
\]
Now we consider the cases $d=2$ and $d\geq 3$ separately. 

{\it Case $d=2$.} We  proceed by analyzing the relation between $\|\cdot\|_{H^{-1}}$ and the Orlicz spaces. Fix $\phi(t) = e^{t^2}-1$. By Trudinger's theorem \cite[Theorem~8.27]{AdamsFournier2003Sobolev}, we have the inequality 
\begin{equation}\label{eq:h_L_varphi}
    \|h\|_{L_\phi(\BB_1(0))} \lesssim  \| h\|_{W^{1,2}(\BB_1(0))}. 
\end{equation}
By the Poincaré inequality \citep[p.~291]{Evans-PDE}, this implies
\begin{equation*}
     \inf_{m\in \R}\|h-m\|_{L_\phi (\BB_1(0))} \lesssim \inf_{m\in \R} \| h-m\|_{W^{1,2}(\BB_1(0))}\lesssim \|\nabla h\|_{L^2(\BB_1(0))},
\end{equation*}
where $\|\cdot\|_{L_\phi(\BB_1(0))}$ is the Orlicz norm, 
$ \|g\|_{L_{\phi}(\BB_1(0))}  := \inf\left\{ k > 0 : \int_{\BB_1(0)} \phi\left(|g(z)|/k \right) \dd z \le 1 \right\}.
$
Recall that $f_\nu$ denotes the Lebesgue density of $\nu$. It follows that
\begin{align*}
  \|\nu_t - \nu \|_{H^{-1}}&=  \sup_{\|\nabla h\|_{L^2(\BB_1(0))}
  \leq 1 } \int_{\BB_1(0)} h (f_{\nu_t}-f_{\nu})\dd z\\&\lesssim \sup_{\inf_{m\in \R}\| h-m\|_{L_\phi(\BB_1(0))} \leq 1 } \int_{\BB_1(0)} h(z) (f_{\nu_t}(z)-f_{\nu}(z))\dd z\\
    &= \sup_{\| h\|_{L_\phi(\BB_1(0))} \leq 1 } \int_{\BB_1(0)} h(z) (f_{\nu_t}(z)-f_{\nu}(z))\dd z,%
\end{align*}
where the last equality follows from the fact that $\int_{\BB_1(0)} (f_{\nu_t}-f_{\nu})\dd z=0$. 

Let $ \phi^*(u) = \sup_{v\in \R} \left\{ uv - \phi(v) \right\}$ be the convex conjugate of $\phi$. 
From \cite[Section~15.3]{Rubshtein2016Orliczdual}, we deduce 
\begin{equation*}
    \sup_{\| h\|_{L_\phi (\BB_1(0))} \leq 1 } \int_{\BB_1(0)} h (f_{\nu_t}-f_{\nu})\dd z\leq 2\|f_{\nu_t}-f_{\nu}\|_{L_{\phi^*}(\BB_1(0))}. 
\end{equation*}

Therefore, it remains to show that 
$$  \|f_{\nu_t}-f_{\nu}\|_{L_{\phi^*}(\BB_1(0))} \lesssim t|\log(t)|^{\frac{1}{2}}. $$
Note that $$L^\infty(\BB_1(0))\ni f_{\nu_t} - f_\nu = \begin{cases}
    -\frac{1}{\cL_d(\BB_1(0))},&\quad {\rm in} \,\, \BV\\
    \frac{t}{(1-t)\cL_d(\BB_1(0))},&\quad {\rm in} \,\, \BB_1(0)\setminus \BV
\end{cases},
$$
after applying triangle inequality, it suffices to prove 
\begin{equation}
    \label{eq:claim-d=2}
     \|{\bf 1}_{\mathbb{V}}\|_{L_{\phi^*}(\BB_1(0))} \lesssim t|\log(t)|^{\frac{1}{2}} .
\end{equation}
By definition 
$$ \|{\bf 1}_{\mathbb{V}}\|_{L_{\phi^*}(\BB_1(0))}  = \inf\left\{ k > 0 : \int_{\BB_1(0)} \phi^*\left(\frac{{\bf 1}_{\mathbb{V}}}{k}  \right) \dd z \le 1 \right\}.
$$
As $\cL_d(\BV) =  t \cL_d(\BB_1(0))$, we have
\begin{align*}
       \|{\bf 1}_{\mathbb{V}}\|_{L_{\phi^*}(\BB_1(0))} & =  \inf\left\{ k > 0 :  \cL_d(\BV) \phi^*\left(\frac{1}{k} \right) \le 1 \right\}= \inf\left\{ k > 0 :  \phi^*\left(\frac{1}{k} \right) \le \frac{1 }{t \cL_d(\BB_1(0))}\right\}.
\end{align*}
On the other hand, since $ \phi'(v)=0$ if and only if $v=0$, it follows that $\phi^*(0)= \phi(0)=0$.  Furthermore, since $\phi^*$ is convex, satisfies $\phi^*(0)=0$, and is
even, it is non-decreasing on $[0,\infty)$. Thus, for all sufficiently small $t>0$ such that 
$$\frac{1 }{t \cL_d(\BB_1(0))} \geq \phi^*(1+e^{1/4}),$$
we have
\begin{equation}
    \label{eq:Orlitz-bound-1}
     \|{\bf 1}_{\mathbb{V}}\|_{L_{\phi^*}(\BB_1(0))} =\inf\left\{ k \in \left(0, \frac{1}{1+e^{\frac{1}{4}}}\right] :  \phi^*\left(\frac{1}{k} \right) \le \frac{1 }{t \cL_d(\BB_1(0))}\right\}.
\end{equation}

Let us give an upper bound on
$ \phi^*(u) $ for $u\geq 1+e^{1/4}$. Since $\phi$ is strictly convex and differentiable, $\phi^*$  is also strictly convex and differentiable. Hence,  $ \argmax_{v\in \R} \left\{ uv - \phi(v) \right\} $ is the singleton $\{v_u\}$ where  
$ 2 v_u e^{v_u^2} = u$. Furthermore, as $  v\mapsto 2 v e^{v^2} $ is monotone increasing and $u\geq  e^{1/4}$, it must follow that $v_u\geq \frac{1}{2}$. Thus, we derive 
$ e^{v_u^2} = u/(2v_u) \leq u $, which implies that 
$ v_u \leq \sqrt{\log{(u)}} $. As a consequence, it holds that 
$$ \phi^*(u)= u v_u - e^{v_u^2}+1 \,\leq\, u\sqrt{\log(u)} +1 \leq  2u\sqrt{\log(u)},\quad u\geq 1+e^{1/4}.$$
This, combined with \eqref{eq:Orlitz-bound-1}, shows that
\begin{equation}\label{eq:L_varphi_upper_bound}
    \|{\bf 1}_{\mathbb{V}}\|_{L_{\phi^*}(\BB_1(0))}  \leq \inf\left\{ k\in \left(0, \left(1+e^{\frac{1}{4}} \right)^{-1}\right] :  \frac{1}{k}\sqrt{\log\left(\frac{1}{k}\right)} \le \frac{1}{2t\cL_d(\BB_1(0))}\right\}.
\end{equation}
Let $C = 1/(2\cL_d(\BB_1(0)))$. There exists $k_t\in \left(0, \left(1+e^{\frac{1}{4}} \right)^{-1}\right]$ such that 
$ \frac{1}{k_t}\sqrt{\log\left(\frac{1}{k_t}\right)}= \frac{C}{t}$. Thus, the strict monotonicity of $k\mapsto \frac{1}{k}\sqrt{\log\left(\frac{1}{k}\right)}$ implies that
$ \|{\bf 1}_{\mathbb{V}}\|_{L_{\phi^*}(\BB_1(0))} \leq k_t$. Moreover, as $k_t<(1+e^{1/4})^{-1}$, we derive that $\frac{1}{k_t}\lesssim \frac{1}{t}$,
which implies that 
\begin{equation}\label{eq:kt-bound}
    \frac{1}{k_t}\sqrt{\log\left(\frac{1}{t}\right)} \gtrsim   \frac{1}{k_t}\sqrt{\log\left(\frac{1}{k_t}\right)}= \frac{C}{t}. 
\end{equation}
Therefore, by rearranging the terms in~\eqref{eq:kt-bound}, we conclude that, for $t$ small enough,
$$  \|{\bf 1}_{\mathbb{V}}\|_{L_{\phi^*}(\BB_1(0))}\leq  {k_t}\lesssim {t \sqrt{\log\left(\frac{1}{t}\right)}}, $$
which proves \eqref{eq:claim-d=2}.

{\it Case $d\geq 3$}. The proof follows the same strategy. Instead of applying Trudinger's theorem and the Orlicz norm in~\eqref{eq:h_L_varphi}, we use the Sobolev inequality \citep[p.~284]{Evans-PDE} to obtain
$$ \|h\|_{L^{\frac{2 d}{d-2}}(\BB_1(0))}\lesssim \| h\|_{W^{1,2}(\BB_1(0))}.$$
Arguing as in the case $d=2$, it suffices to show that
$$ \|{\bf 1}_{\mathbb{V}}\|_{L^{\frac{2 d}{d+2}}(\BB_1(0))} \lesssim t^{\frac{d+2}{2 d}},$$
which follows by a direct computation.
\end{proof}

\subsection{Uniform $L^q$ Estimates}\label{sec:Lq-estimate}
Recall that~\Cref{thm:L2-estimate} shows
\begin{equation}\label{eq:L2-estimate2}
   \left\|\BQ_{P_t}-\BQ_{P} \right\|_{L^2(\mu)}^2 \lesssim  \begin{cases}
      t^{1+\frac{2}{d}} & {\rm if} \ d\geq 3,\\
t^{2}{|\log(t)|}     & {\rm if}\ d=2
  \end{cases} 
\end{equation}
for sufficiently small $t$. However, we are interested in the limit of the potential $\xi_t$, which is characterized by\footnote{Here, $\xi_t$ is defined up to an additive constant.}
\begin{equation}\label{eq:xi-t}
    \nabla \xi_t:= \frac{\BQ_{P_t}-\BQ_{P}}{t}.
\end{equation}
Dividing both sides of~\eqref{eq:L2-estimate2} by $t^2$ leads to a divergent right-hand side as $t\to 0$. Therefore, to analyze the limiting behavior of $\xi_t$ as $t\to 0$, we establish the following result: the (centered) $L^q(\Omega_\mu)$ norm of the potential $\xi_t$ is uniformly bounded for small $t$. This will be the key ingredient in proving the $L^q$ convergence of $\xi_t$ in \Cref{Thm:InfluenceQuantiles-main}.
\begin{proposition}\label{Prop:Lq-estimate}
For any $q \in \left(1,\; \frac{d(d+2)}{d^{2}+2d-4}\right)$, there exists $t_0>0$ such that
    $$  \sup_{t\in (0,t_0]} \inf_{m_t\in \R}\| \xi_t-m_t\|_{L^q(\Omega_\mu)}  <\infty.$$
\end{proposition}
To prove \Cref{Prop:Lq-estimate}, we first establish a linearization with respect to a test function. 

\begin{lemma}\label{lemma:test-linearization}
Fix $\eta\in (0,1]$. 
For any $g\in \cC^{1,\eta}(\Omega_\mu)$, 
$$  \bigg|  g(\mathbf{F}_{P}(x_0))-\int_{\Omega_\mu} g  \dd\mu -  \int_{\Omega_\mu} \langle   \nabla g  , [\nabla\BQ_{P}]^{-1} \nabla \xi_t\rangle  \dd\mu \bigg| \lesssim \begin{cases}
    \| \nabla g\|_{_{\CC^\eta(\Omega_\mu)}}  t^{\frac{(d+2)(1+\eta)}{2 d}-1}&\quad {\rm if}\,\, d\geq 3\\
    \| \nabla g\|_{\CC^\eta(\Omega_\mu)} t^{\eta} |\log(t)|^{(1+\eta)/2}&\quad {\rm if}\,\, d=2.
\end{cases}$$
\end{lemma}
\begin{proof}
We will start by showing the following claim: for any $f\in \cC^{1,\eta}(\Omega_P)$, 
\begin{equation}\label{eq:linearization-f}
    \bigg|  f(x_0)-\int_{\Omega_P} f \dd P -  \int_{\Omega_\mu} \langle \nabla f(\BQ_{P}), \nabla \xi_t\rangle  \dd\mu \bigg| \leq  \begin{cases}
    \| \nabla f\|_{\CC^\eta(\Omega_P)} t^{\frac{(d+2)(1+\eta)}{2 d}-1}&\quad {\rm if}\,\, d\geq 3,\\
    \| \nabla f\|_{\CC^\eta(\Omega_P)} t^{\eta} |\log(t)|^{(1+\eta)/2}&\quad {\rm if}\,\, d=2.
\end{cases}
\end{equation}
Our desired result can then be concluded by taking $f=g \circ \mathbf{F}_{P}$. Note that $f\in \CC^{1,\eta}(\Omega_P)$, $\nabla f=  [\nabla\mathbf{F}_{P}] \nabla g (\mathbf{F}_{P})$, and $\|\nabla f\|_{\CC^{\eta}(\Omega_P)}\lesssim \|\nabla g\|_{\CC^{\eta}(\Omega_\mu)}$. 
Furthermore, we have
$$\nabla f(\BQ_{P})=  [\nabla\mathbf{F}_{P}(\BQ_{P})] \nabla g =  [\nabla\BQ_{P}]^{-1} \nabla g.$$
Hence by the symmetry of $ [\nabla\BQ_{P}]^{-1}$, we get 
$$ \int_{\Omega_\mu} \langle \nabla f(\BQ_{P}), \nabla \xi_t\rangle  \dd\mu=  \int_{\Omega_\mu} \langle   \nabla g  , [\nabla\BQ_{P}]^{-1} \nabla \xi_t\rangle  \dd\mu,$$
which concludes the result.

To show~\eqref{eq:linearization-f}. Fix a test function $f\in \cC^{1,\eta}(\Omega_P)$. Then by the push-forward, it follows that   
    \begin{align*}
        t \left(f(x_0)-\int_{\Omega_P} f \dd P\right)=  \int_{\Omega_P}  f \dd (P_t-P)=& \int_{\Omega_\mu} \left( f(\BQ_{P_t})-f(\BQ_{P}) \right) \dd\mu.
    \end{align*} 
By fundamental theorem of calculus, we have
\begin{align*}
    & f(\BQ_{P_t})-f(\BQ_{P}) -  \langle \nabla f(\BQ_{P}), \BQ_{P_t} - \BQ_P\rangle  \\
    =& \int_0^1 \inner{\nabla f(\BQ_P + s(\BQ_{P_t} - \BQ_P))-\nabla f(\BQ_P))}{\BQ_{P_t} - \BQ_P}\dd s\\
    \leq & \|\nabla f\|_{\CC^\eta(\Omega_P)}\|\BQ_{P_t} - \BQ_P\|^{1+\eta}.
\end{align*}
Therefore, as $\eta \in (0,1]$,
\begin{align*}
  \bigg|  f(x_0)-\int_{\Omega_P} f \dd P -  \int_{\Omega_\mu} \langle \nabla f(\BQ_{P}), \nabla \xi_t\rangle  \dd\mu \bigg| &\leq \frac{1}{t} \| \nabla f\|_{\CC^\eta(\Omega_P)}  \|\BQ_{P_t}-\BQ_{P} \|_{L^{1+\eta}(\mu)}^{1+\eta}\\
    &\leq \frac{1}{t} \| \nabla f\|_{\CC^\eta(\Omega_P)} \|\BQ_{P_t}-\BQ_{P} \|_{L^{2}(\mu)}^{1+\eta}.
\end{align*}
The claim follows from \Cref{thm:L2-estimate}.
\end{proof}

We now prove \Cref{Prop:Lq-estimate}. 
\begin{proof}[Proof of \Cref{Prop:Lq-estimate}] 
Define $${\rm sgn}(s)= \begin{cases}
    1& {\rm if}\ s>0,\\
    0& {\rm if}\ s=0,\\
    -1& {\rm if}\ s<0.
\end{cases}$$
Given $q>1$ whose value will be specified at the end, denote by $\beta=q-1$, and let $p$ be the H\"older conjugate of $q$.  Take $m_t \in \R$ such that 
$$ \int_{\Omega_\mu} {\rm sgn}(\xi_t-m_t) |\xi_t-m_t|^{\beta}\dd z=0.$$
Assume that $\xi_t\neq m_t$ as otherwise the result is trivial and define the function 
$$ g_\beta= \frac{{\rm sgn}(\xi_t-m_t) |\xi_t-m_t|^{\beta}}{\|\xi_t-m_t\|_{L^{q}(\Omega_\mu)}^{q/p}}. $$
Note that $g_\beta$ satisfies \eqref{eq: compatibility condition-general} as $\int_{\Omega_\mu} g_\beta\dd z= 0$ and  $g_\beta \in {L^p}(\Omega_\mu) $ with 
\begin{align*}
    \| g_\beta\|_{L^p(\Omega_\mu)}&= \left(  \int_{\Omega_\mu} \frac{|\xi_t-m_t|^{(q-1)p}}{\|\xi_t-m_t\|_{L^{q}(\Omega_\mu)}^{q}}\dd z\right)^{\frac{1}{p}}\\
    &=  \frac{1}{\|\xi_t-m_t\|_{L^{q}}^{q/p}}\left(  \int_{\Omega_\mu} {|\xi_t-m_t|^{q}}\dd z \right)^{\frac{1}{p}} = 1
\end{align*}
as $(q-1)p = q$.
Set $h_\beta =\BL^{-1}(g_\beta,0)$ where $\BL$ is defined in \eqref{eq:operator-L} and satisfies that $\int_{\Omega_\mu} h_\beta \dd z =0$. Note that, by \Cref{lemma: existence and uniqueness linearized MA}, $\|h_\beta\|_{W^{2,p}(\Omega_\mu)} \leq C $, for some constant $C>0$ independent of $t$. Furthermore, Morrey's inequality (see \cite[Theorem 7.26]{GilbargTrudinger.Book}) yields 
 \begin{equation}
     \label{eq:Morrey-ineq}
     \|h_\beta\|_{ \cC^{1,\eta}(\Omega_\mu)} \lesssim  \|h_\beta\|_{W^{2,p}(\Omega_\mu)}, \quad \text{for}\quad   p>d \quad   \text{and}\quad  \eta\leq 1- \frac{d}{p} . 
 \end{equation}
On the other hand, for $\eta>\frac{2d}{d+2}-1$, applying
 \Cref{lemma:test-linearization} with $g=h_\beta$ yields that 
 $$  \bigg|\int_{\Omega_\mu} \langle   \nabla h_\beta  , [\nabla\BQ_{P}]^{-1} \nabla \xi_t\rangle \dd \mu \bigg| \lesssim \|\nabla h_\beta\|_{\CC^{\eta}(\Omega_\mu)} + \|h_\beta\|_{L^\infty(\Omega_\mu)}<\infty,$$
where in the last inequality we used \eqref{eq:Morrey-ineq}. 
Using integration by parts and recalling that $\BL(h_\beta)= (g_\beta, 0)$, it follows that 
\begin{align*}
\int_{\Omega_\mu} \langle    \nabla h_\beta  , [\nabla\BQ_{P}]^{-1}\nabla \xi_t\rangle \dd \mu
    &=\int_{\Omega_\mu} \langle  [\nabla\BQ_{P}]^{-1}  \nabla h_\beta  , \nabla \xi_t\rangle \dd \mu\\
    &= \int_{\Omega_\mu} \langle  [\nabla\BQ_{P}]^{-1}  \nabla h_\beta  , \nabla (\xi_t-m_t)\rangle f_\mu \dd z\\
    &=\int_{\partial \Omega_\mu}  \langle  f_\mu[\nabla\BQ_{P}]^{-1}  \nabla h_\beta ,   n_{\mu} \rangle  (\xi_t-m_t) \dd \mathcal{H}^{d-1}\\
    &\hspace{1cm}-\int_{\Omega_\mu}  {\rm div} (f_\mu [\nabla\BQ_{P}]^{-1} \nabla h_\beta  ) (\xi_t-m_t)\dd z\\
    &= -\| \xi_t-m_t\|_{L^q(\Omega_\mu)}, 
\end{align*}
where we recall that $n_\mu$ is the outward unit normal to $\Omega_\mu$. Choosing $\eta \in \left(\frac{2d}{d+2}-1, 1\right)$ and  $p=  \frac{d}{1-\eta} >d$, we conclude the bound for any $q = \frac{p}{p-1} \in \left(1, \frac{d(d+2)}{d^{2}+2d-4}\right)
$.
\end{proof}

\subsection{Local H\"older estimates}\label{sec:holder}
In this section, we establish uniform $\cC_{\rm loc}^{2,\alpha}$ estimates of $\xi_t$ on $\Omega_\mu\setminus \BF_P(x_0)$. Excluding the point $\BF_P(x_0)$ is necessary due to the expected pole-type singularity suggested in~\Cref{sec:Explicit}. Recall that $K_{t,x_0} := (\partial \varphi_{t})^{-1}(x_0)\subset \Omega_\mu$, that $\BQ_{P_t}$ pushes $\frac{1}{1-t} {\bf 1}_{\R^d\setminus K_{t,x_0}}\dd \mu$ forward to $P$, and that $\xi_t$ is defined in~\eqref{eq:xi-t}.
\begin{proposition}\label{prop:Uniform-estimate}
 Under the setting of \Cref{Thm:InfluenceQuantiles-main}, for every open set $U\Subset \Omega_\mu\setminus\{\BF_P(x_0)\}$, there exists $t_0>0$ and $C>0$ such that
  $$  \sup_{t\in (0,t_0)} \inf_{m_t\in \R} \|\xi_t -m_t\|_{\cC^{2,\alpha}(U)} \leq C.  $$
\end{proposition}
 As a consequence of~\Cref{prop:shape}, we derive the following result, which is used in the proof of \Cref{prop:Uniform-estimate}.

\begin{lemma}\label{lemma:Interior-time-depend}
    Under the setting of \Cref{prop:Uniform-estimate}, %
    there exists $$t_0=t_0\!\left(
d,\alpha,\lambda,\Lambda,
\Omega_\mu,\Omega_P,
\|f_\mu\|_{\cC^{1,\alpha}(\Omega_\mu)},
\|f_P\|_{\cC^{1,\alpha}(\Omega_P)}
,U\right)>0$$ and $$C=C\!\left(
d,\alpha,\lambda,\Lambda,
\Omega_\mu,\Omega_P,
\|f_\mu\|_{\cC^{1,\alpha}(\Omega_\mu)},
\|f_P\|_{\cC^{1,\alpha}(\Omega_P)}
,U\right)>0$$ such  that, for every $t\in (0,t_0)$,
$$  \|\BQ_{P_t}\|_{\cC^{1,\alpha}(U)}\leq C \quad \text{and}\quad \frac{1}{C} {\rm I}_d  \leq  \nabla \BQ_{P_t}(z)  \leq C{\rm I}_d \quad \text{for all}\,\,z\in U.$$  
\end{lemma}
\begin{proof}[Proof of \Cref{lemma:Interior-time-depend}]
    Recall that $\BQ_{P_t}=\nabla\varphi_t$ for all $t\geq0$, where
    $\varphi_t:\R^d\to\R$ is a convex function.
    In what follows, we sketch the argument based on the ``sections'' to
    show that $\varphi_t\in\cC^{2,\alpha}(U)$ and that
    $\|\varphi_t\|_{\cC^{2,\alpha}(U)}$ is uniformly bounded for
    $t\leq t_0$.
    The bound depends on the fixed interior set $U$, the regularity data
    of $\mu$ and $P$, and $d,\lambda,\Lambda,\alpha$.
    For the detailed argument, we refer the reader to
    \cite[Section~4]{figalli2017monge}, which provides a comprehensive
    overview of the related techniques.

    As $U$ is open and $U\Subset \Omega_\mu\setminus\{\BF_P(x_0)\}$, there exists an open set
    \[
        U\Subset V\Subset
        \Omega_\mu\setminus\{\BF_P(x_0)\}.
    \]
    Note that \Cref{prop:shape} implies that
    $V\cap K_{t,x_0}=\emptyset$ for $t$ small enough, where
    \[
        K_{t,x_0}:=(\partial\varphi_{t,x_0})^{-1}(x_0).
    \]
    Furthermore, $\nabla\varphi_t$ pushes
    $\mu|_{\R^d\setminus K_{t,x_0}}$ forward to $(1-t)P$.
    The measures $(1-t)P$ and $\mu$ have densities bounded away from
    zero and bounded above on their supports, and the target $(1-t)P$ is
    supported on the convex set $\Omega_P$. Hence, Caffarelli's interior
    regularity theorem \cite{Caffarelli.92} implies that
    $\partial\varphi_t$ is single-valued in $V$ for $t$ small enough.

    By decreasing $t_0$ if necessary, we assume throughout that
    $t_0\leq 1/2$.
    Take $z\in U$ and set $p_t=\nabla\varphi_t(z)$. Define the affine
    function
    \[
        l_{t,z,\vae}(y)
        =
        \varphi_t(z)+\inner{p_t}{y-z}+\vae,
    \]
    with the convention
    \[
        l_{0,z,\vae}(y)
        =
        \varphi(z)+\inner{\nabla\varphi(z)}{y-z}+\vae.
    \]
    Define
    \[
        S_{t,z,\vae}
        =
        \left\{
        y\in\Omega_\mu:
        \varphi_t(y)<l_{t,z,\vae}(y)
        \right\}.
    \]
    By \Cref{thm:Cafarrelli}, there exist constants
    $0<m\leq M<\infty$ such that
    \[
        m\,\mathrm{Id}
        \preceq D^2\varphi
        \preceq M\,\mathrm{Id}
        \qquad\text{on }\Omega_\mu.
    \]
    Fix $\vae>0$ sufficiently small that
    $2\sqrt{\frac{\vae}{m}}
    <\operatorname{dist}(\overline U,\partial V)$.
    We claim that
    \begin{equation}
        \label{eq:section_shape}
        S_{0,z,\vae/2}
        \subset
        S_{t,z,\vae}
        \subset
        S_{0,z,2\vae},
        \qquad
        z\in U,\quad t\leq t_0.
    \end{equation}
    By definition,
    \[
        S_{t,z,\vae}
        =
        \left\{
        y\in\Omega_\mu:
        \varphi_t(y)-\varphi_t(z)
        -\inner{p_t}{y-z}<\vae
        \right\}.
    \]
    By \cite[Theorem~2.8]{delBarrioLoubes.19}, $\varphi_t$ converges
    pointwise to $\varphi$ in $\Omega_\mu$. Since the image of the
    subgradient $\partial\varphi_t(\Omega_\mu)$ is bounded,
    $\{\varphi_t\}_{t\leq t_0}$ is uniformly Lipschitz and, under the
    chosen normalization, converges uniformly to $\varphi$ on
    $\Omega_\mu$ by the Arzelà--Ascoli theorem. Furthermore,
    $\nabla\varphi_t$ converges uniformly to $\nabla\varphi$ on $V$ by
    \cite[Theorem~7.17]{RockafellarWets1998}.
    For
    \[
        D_t(z,y)
        :=
        \varphi_t(y)-\varphi_t(z)
        -\inner{\nabla\varphi_t(z)}{y-z},
    \]
    the preceding uniform convergences and the boundedness of
    $\Omega_\mu$ imply, after decreasing $t_0$ if necessary, that
    \[
        \sup_{\substack{z\in U\\y\in\Omega_\mu}}
        |D_t(z,y)-D_0(z,y)|
        \leq\frac{\vae}{2},
        \qquad t\leq t_0.
    \]
    This proves \eqref{eq:section_shape}, uniformly in $z\in U$.
    Moreover, the bounds on $D^2\varphi$ give
    \[
        \frac{m}{2}\|y-z\|^2
        \leq D_0(z,y)
        \leq\frac{M}{2}\|y-z\|^2.
    \]
    Consequently,
    \begin{equation}
        \label{eq:uniform-section-balls}
        \BB_{\sqrt{\frac{\vae}{M}}}(z)
        \subset
        S_{0,z,\vae/2}
        \subset
        S_{t,z,\vae}
        \subset
        S_{0,z,2\vae}
        \subset
        \BB_{2\sqrt{\frac{\vae}{m}}}(z)
        \Subset V,
    \end{equation}
    uniformly for $z\in U$ and $t\leq t_0$.

    Since $\varphi_t$ is convex, John's lemma
    (\Cref{lemma:john}) gives an affine map
    $L_{t,z,\vae}$ such that
    \[
        \BB_1(0)
        \subset
      L_{t,z,\vae}(S_{t,z,\vae})
        \subset
        d\,\BB_1(0), 
    \]
    where 
    $A_{t,z,\vae}:=\nabla L_{t,z,\vae}$ satisfies
    $\det A_{t,z,\vae}>0$.
    The inclusions \eqref{eq:uniform-section-balls} give uniform control
    of both $A_{t,z,\vae}$ and
$A_{t,z,\vae}^{-1}$. Indeed, for every
    $e\in\R^d$ with $\|e\|=1$,
    \[
        \sqrt{\frac{\vae}{M}}
        \|A_{t,z,\vae}e\|
        =
        \left\|L_{t,z,\vae}
        \left(
    z+\frac{\sqrt{\frac{\vae}{M}}}{2}e
        \right)
        - {L_{t,z,\vae}}
        \left(
        z-\frac{\sqrt{\frac{\vae}{M}}}{2}e
        \right)
        \right\|
        \leq 2d,
    \]
    and hence
    \(
        \|A_{t,z,\vae}\|
        \leq\frac{2d}{\sqrt{\frac{\vae}{M}}}.
    \)
    On the other hand, since
    \(
        \BB_1(0)
        \subset
        L_{t,z,\vae}(S_{t,z,\vae}),
    \)
    there exist $y_t^0,y_t^e\in S_{t,z,\vae}$ such that
    \(L_{t,z,\vae}({y_t^0})=0
    \)
    and
    \(L_{t,z,\vae}({y_t^e})
        =\frac{e}{2}.
    \)
    Therefore,
    \[
        \frac12
        \|A_{t,z,\vae}^{-1}e\|
        =
        \|y_t^e-y_t^0\|
        \leq
        \operatorname{diam}(S_{t,z,\vae})
        \leq
        {4\sqrt{\frac{\vae}{m}}},
    \]
    so that
    \(
        \|A_{t,z,\vae}^{-1}\|
        \leq
        {8\sqrt{\frac{\vae}{m}}}.
    \)
    In particular,
    \begin{equation}
        \label{eq:affine-uniform}
        {
        \left(8\sqrt{\frac{\vae}{m}}\right)^{-d}}
        \leq
        \det{A_{t,z,\vae}}
        \leq
        \left(
        \frac{2d}{\sqrt{\frac{\vae}{M}}}
        \right)^d.
    \end{equation}
    Furthermore,
    \begin{equation}
        \label{eq:normalized-interior}
        \operatorname{dist}
        \left(
       L_{t,z,\vae}(z),
        \partial L_{t,z,\vae}(S_{t,z,\vae})
        \right)
        \geq
        \frac{
        \sqrt{\frac{\vae}{M}}
        }{
        \|A_{t,z,\vae}^{-1}\|
        }
        \geq
        \frac{
        \sqrt{\frac{\vae}{M}}
        }{
        8\sqrt{\frac{\vae}{m}}
        }
        =:\theta_0>0.
    \end{equation}
Define
    \begin{equation}
        \label{eq:u_t}
        \begin{aligned}
            u_t(y)
            &=
            \bigl(\det{A_{t,z,\vae}}\bigr)^{2/d}
            \left(
            \varphi_t-l_{t,z,\vae}
            \right)
            \bigl(
          L_{t,z,\vae}^{-1}(y)
            \bigr)
            \\ &=\bigl(\det{A_{t,z,\vae}}\bigr)^{2/d}
            \left(
            \varphi_t
            \bigl(
            {L_{t,z,\vae}^{-1}}(y)
            \bigr)
            -\varphi_t(z)
            -\inner{p_t}{
            {L_{t,z,\vae}^{-1}}(y)-z
            }
            -\vae
            \right).
        \end{aligned}
    \end{equation}
    Set
    \(
        Z_{t,\vae}
        :=
        {L_{t,z,\vae}}(S_{t,z,\vae}).
    \)
    Then $u_t$ solves, in the Alexandrov sense,
    \begin{equation}
        \label{eq:normalized-MA}
        \begin{cases}
            \displaystyle
            \det(D^2u_t)
            =
            \frac{
            f_\mu\circ
            {L_{t,z,\vae}^{-1}}
            }{
            (1-t)f_P\circ\nabla\varphi_t\circ
            {L_{t,z,\vae}^{-1}}
            }
            &\text{in }Z_{t,\vae},
            \\[3mm]
            u_t=0
            &\text{on }\partial Z_{t,\vae}.
        \end{cases}
    \end{equation}
    See \cite[Definition~2.1 and Theorem~4.23]{figalli2017monge}.
    Such $u_t$ is referred to as the \textit{normalized solution}.

    Since $\mu$ and $P$ are regular, and $t_0\leq1/2$, the right-hand
    side of \eqref{eq:normalized-MA} is bounded between
    $\frac{\lambda}{\Lambda}$ and $\frac{2\Lambda}{\lambda}$.
    Hence, \cite[Corollary~4.11]{figalli2017monge} shows that $u_t$ is
    strictly convex on $Z_{t,\vae}$. In view of \eqref{eq:u_t}, this
    shows that $\varphi_t$ is strictly convex in $S_{t,z,\vae}$.
    Since $z\in U$ is arbitrary, $\varphi_t$ is strictly convex in $U$. Hence, \cite[Corollary~4.21]{figalli2017monge} yields
    \(u_t\in\cC_{\rm Loc}^{1,\gamma}(Z_{t,\vae})\)
    for some $0<\gamma\leq\alpha$ depending only on
    $d,\lambda,\Lambda$.
    Differentiating \eqref{eq:u_t} gives
    \begin{equation}
        \label{eq:gradient-change-variables}
        \nabla\varphi_t
        \bigl(
        {L_{t,z,\vae}^{-1}}(y)
        \bigr)
        =
        p_t
        +
        \bigl(\det{A_{t,z,\vae}}\bigr)^{-2/d}
        {A_{t,z,\vae}}^\top
        \nabla u_t(y).
    \end{equation}
    Therefore, the regularity of $f_\mu$ and $f_P$, together with
    \eqref{eq:affine-uniform} and the local
    $\cC^{1,\gamma}$ estimate for $u_t$, implies that the right-hand
    side of \eqref{eq:normalized-MA} belongs to
    \(\cC_{\rm Loc}^{0,\gamma}(Z_{t,\vae})\).
    Thus, \cite[Corollary~4.43]{figalli2017monge} shows that
    \(u_t\in\cC_{\rm Loc}^{2,\gamma}(Z_{t,\vae})\).
    It follows from \eqref{eq:u_t} that
    $\varphi_t\in\cC_{\rm Loc}^{2,\gamma}(U)$. Since
    $f_\mu,f_P\in\cC^{1,\alpha}$, the right-hand side of
    \eqref{eq:normalized-MA} is then locally $\cC^{0,\alpha}$.
    Iterating the preceding argument gives
    \(\varphi_t\in\cC_{\rm Loc}^{2,\alpha}(U)\).

    To obtain an estimate uniform in $t$, note first that the normalized
    domains satisfy $\BB_1(0)\subset Z_{t,\vae}\subset d\BB_1(0),$
    and, by \eqref{eq:normalized-interior}, the point
    \(
        y_{t,z}
        :=
        {L_{t,z,\vae}}(z)
    \)
    remains a fixed positive distance from $\partial Z_{t,\vae}$.
    The estimates in
    \cite[Corollary~4.43 and Lemma~4.41]{figalli2017monge}, together
    with the bootstrap above, therefore yield
    \begin{equation}
        \label{eq:uniform-normalized-estimate}
        \|u_t\|_{\cC^{2,\alpha}
        (\BB_{\theta_0/2}(y_{t,z}))}
        \leq C,
    \end{equation}
    where $C$ is independent of $t\leq t_0$ and $z\in U$.
    Differentiating \eqref{eq:u_t} twice gives
    \begin{equation}
        \label{eq:hessian-change-variables}
        D^2\varphi_t(x)
        =
        \bigl(\det{A_{t,z,\vae}}\bigr)^{-2/d}
        {A_{t,z,\vae}}^\top
        D^2u_t
        \bigl(
        {L_{t,z,\vae}}(x)
        \bigr)
        {A_{t,z,\vae}}.
    \end{equation}
    Let
    \(
        {
        C_A
        :=
        \sup_{\substack{z\in U\\0\leq t\leq t_0}}
        \|A_{t,z,\vae}\|}
        <\infty
    \)
    and
    \(
        \eta_0:=\frac{\theta_0}{2C_A}.
    \)
    If $x\in\BB_{\eta_0}(z)$, then
    \(
        {L_{t,z,\vae}}(x)
        \in
        \BB_{\theta_0/2}
        \bigl(
        {L_{t,z,\vae}}(z)
        \bigr).
    \)
    Consequently, \eqref{eq:affine-uniform},
    \eqref{eq:uniform-normalized-estimate}, and
    \eqref{eq:hessian-change-variables} imply
    \[
        \|\varphi_t\|_{\cC^{2,\alpha}(\BB_{\eta_0}(z))}
        \leq C,
    \]
    uniformly in $z\in U$ and $t\leq t_0$. Since $\eta_0$ is independent
    of $z$ and  the closure {$\overline U$ is compact}, a finite-covering
    argument gives
    \[
        \sup_{t\leq t_0}
        \|\varphi_t\|_{\cC^{2,\alpha}(U)}
        \leq C.
    \]
    Here the zeroth- and first-order terms are uniformly controlled by
    the uniform convergence of $\varphi_t$ and the boundedness of
    $\partial\varphi_t(\Omega_\mu)\subset\Omega_P$.

    Finally, on $U$,
    \[
        \det D^2\varphi_t
        =
        \frac{f_\mu}
        {(1-t)f_P\circ\nabla\varphi_t}
        \geq
        \frac{\lambda}{\Lambda}.
    \]
    The uniform upper bound on $D^2\varphi_t$, together with this
    determinant lower bound, implies
    \[
        \frac1C\,\mathrm{Id}
        \preceq
        D^2\varphi_t
        \preceq
        C\,\mathrm{Id}
        \qquad\text{on }U.
    \]
    Since $\BQ_{P_t}=\nabla\varphi_t$, this proves both asserted
    conclusions.
\end{proof}
We are now ready to prove \Cref{prop:Uniform-estimate}. 
\begin{proof}[Proof of \Cref{prop:Uniform-estimate}]
Fix a compact set $K$ such that
\[
U\subset K\subset \Omega_\mu\setminus\{\BF_P(x_0)\}.
\]
By \Cref{prop:shape}, there exists $t_0>0$ such that
\[
K_{t,x_0}\cap K=\emptyset,
\qquad
\text{for every }t\in(0,t_0].
\]
Therefore, for every $t\in(0,t_0]$,
\Cref{lemma:Interior-time-depend} implies that $\BQ_{P_t}$ is a strong
solution of
\begin{equation}\label{eq:MA-t}
\log(\det(\nabla \BQ_{P_t}(z))) = {\log(f_\mu(z))}-\log(1-t)-\log({f_P(\BQ_{P_t}(z)})), \qquad z \in K.
\end{equation}
 Likewise, we have that 
\begin{equation}\label{eq:MA-0}
\log(\det(\nabla \BQ_{P}(z))) = {\log(f_\mu(z))}-\log({f_P(\BQ_{P}(z))}), \qquad z \in K.
\end{equation}
Note that for $A \in \BS_d^+$, 
\[
\frac{\dd}{\dd\epsilon} \bigg|_{\epsilon=0} \log (\det (A+\epsilon B)) = \tr \Big( A^{-1} B \Big). 
\]
Let $A_s= s \nabla \BQ_{P_t}+ (1-s) \nabla \BQ_{P} $, then the above formula and the fundamental theorem of calculus give
$$ \log(\det(\nabla \BQ_{P_t}(z)))-\log(\det(\nabla \BQ_{P}(z))) =  t   \int_{0}^1  \tr \Big( A_s(z)^{-1} \nabla^2 \xi_t(z)   \Big) \dd s .$$
By the linearity of the trace operator, we know that 
\begin{equation}
    \label{eq:strong-linearization}
    \log(\det(\nabla \BQ_{P_t}(z)))-\log(\det(\nabla \BQ_{P}(z))) =    t\cdot  \tr \Big(  \Big(\int_{0}^1 A_s(z)^{-1} \dd s \Big) \nabla^2 \xi_t (z)  \Big)  .
\end{equation}
On the other hand, the same argument shows that 
\begin{equation}
 \label{eq:strong-linearization-2}
  \log({f_P(\BQ_{P_t}(z))}) -\log({f_P(\BQ_{P}(z))})= t  \left\langle \int_{0}^1 \frac{\nabla f_P(\BT_s(z) )}{ f_P(\BT_s(z) )} \dd s, \nabla \xi_t(z) \right\rangle,
\end{equation}
where $\BT_s= s  \BQ_{P_t}+ (1-s) \BQ_{P} $ denotes the interpolated optimal transport map. Combining \eqref{eq:MA-t}, \eqref{eq:MA-0}, \eqref{eq:strong-linearization} and \eqref{eq:strong-linearization-2}, we derive that, for $t$ small enough and $z\in K$,  
\begin{equation}
    \label{Linearized-MA}
  \tr \Big(  \Big(\int_{0}^1 A_s(z)^{-1} \dd s \Big) \nabla^2 \xi_t (z)  \Big)  +   \left\langle \int_{0}^1 \frac{\nabla f_P(\BT_s(z) )}{ f_P(\BT_s(z) )} \dd s, \nabla \xi_t(z) \right\rangle = -\frac{\log(1-t)}{t}.
\end{equation}
Note that the left-hand side of \eqref{Linearized-MA} defines a strongly elliptic operator (of $\xi_t$) by \Cref{lemma:Interior-time-depend}.   
Fix $q \in \left(1,\; \frac{d(d+2)}{d^{2}+2d-4}\right)$ as in \Cref{Prop:Lq-estimate} and choose a family of constants  $\{m_t\}_{t\leq t_0}$ such that $\sup_{t\leq t_0} \|\xi_t-m_t\|_{L^q(\Omega_\mu)} < \infty$. Define $\tilde \xi_t = \xi_t-m_t$. Note that $\tilde\xi_t$ also satisfies~\eqref{Linearized-MA}. Then the local maximum principle \citep[Theorem~9.20]{GilbargTrudinger.Book} implies that for any ball $\BB_{2R}(z') \subset K $ with radius $2R$, 
$$ \sup_{\BB_{R}(z')} |\tilde \xi_t| \lesssim \|\tilde \xi_t\|_{L^q(\BB_{2R}(z'))} + \frac{|\log(1-t)|}{t} \lesssim \|\tilde \xi_t\|_{L^q(\BB_{2R}(z'))}  +1,$$
which, together with \Cref{Prop:Lq-estimate}, yields that 
\begin{equation}
    \label{eq:Estimate-Uniform-1}
     \sup_{\BB_{R}(z')} |\tilde \xi_t| \lesssim 1 .
\end{equation}
Furthermore, Schauder interior estimates \citep[Theorem~6.2]{GilbargTrudinger.Book} state  
$$ \|\tilde \xi_t\|_{\cC^{2,\alpha}(\BB_{R/2}(z')}) \lesssim \sup_{\BB_{R}(z')} |\tilde \xi_t| +\frac{\left|\log(1-t)\right|}{t} \lesssim 1.$$
In view of \eqref{eq:Estimate-Uniform-1} and the compactness of $K$, we conclude the proof.
\end{proof}

\subsection{Proof of~\Cref{Thm:InfluenceQuantiles-main}}\label{sec:proof-main}
Fix $d\geq 3$. When $d=2$, a similar argument as below shows the statement for any $q\in (1,2)$. Since $ \frac{d(d+2)}{d^{2}+2d-4} = 2$ when $d=2$, we conclude the theorem for all $d\geq 2$.

 To show \Cref{Thm:InfluenceQuantiles-main} when $d\geq 3$, we start by arguing that the limit exists in $L^q(\Omega_\mu)$ for $q \in \left(1, \frac{d(d+2)}{d^{2}+2d-4}\right)$.   First, \Cref{Prop:Lq-estimate} gives
\[
\sup_{t\in(0,t_0]}
\inf_{a\in\R}
\|\xi_t-a\|_{L^q(\Omega_\mu)}
<\infty.
\]
We choose the representative $m_t:=\int_{\Omega_\mu}\xi_t\,d\mu$ and set $\widetilde\xi_t:=\xi_t-m_t.$ 
Then $\int_{\Omega_\mu}\widetilde\xi_t\,d\mu=0.$ 
Moreover, since $\mu$ is a probability measure, for every $a\in\R$,
\[
|m_t-a|
=
\left|
\int_{\Omega_\mu}(\xi_t-a)\,d\mu
\right|
\leq
\|\xi_t-a\|_{L^q(\mu)}.
\]
Using the equivalence of the $L^q(\mu)$ and
$L^q(\Omega_\mu)$ norms, we obtain
\[
\|\widetilde\xi_t\|_{L^q(\Omega_\mu)}
\leq
C\inf_{a\in\R}
\|\xi_t-a\|_{L^q(\Omega_\mu)}.
\]
Consequently,
\[
\sup_{t\in(0,t_0]}
\|\widetilde\xi_t\|_{L^q(\Omega_\mu)}
<\infty,
\]
and every $L^q$ limit of $\widetilde\xi_t$ satisfies the normalization $\int_{\Omega_\mu}G_{x_0}\,d\mu=0.$ 
It remains to verify the translation condition in the
Kolmogorov--Riesz--Fr\'echet theorem
\cite[Theorem~4.26]{brezis2011functional}
to show that the sequence $\{\xi_t-m_t\}_{t\leq t_0}$ is relatively compact in $L^q(\Omega_\mu)$. 
For convenience, we extend $\widetilde\xi_t$ by zero to
$\mathbb R^d\setminus\Omega_\mu$. We will prove that
\begin{equation}\label{eq:lq-uniform-continuity}
\begin{aligned}
\lim_{\|h\|\to0}\sup_{t\in(0,t_0]}
\Bigg\{
&\underbrace{
\int_{(\Omega_\mu-h)\cap\Omega_\mu}
|\widetilde\xi_t(z+h)-\widetilde\xi_t(z)|^q\,\dd z
}_{=:I_1(t,h)}
\\
&+
\underbrace{
\int_{\Omega_\mu\setminus(\Omega_\mu-h)}
|\widetilde\xi_t(z)|^q\,\dd z
}_{=:I_2(t,h)}
\\
&+
\underbrace{
\int_{(\Omega_\mu-h)\setminus\Omega_\mu}
|\widetilde\xi_t(z+h)|^q\,\dd z
}_{=:I_3(t,h)}
\Bigg\}
=0.
\end{aligned}
\end{equation}
Indeed, since $\widetilde\xi_t$ vanishes outside $\Omega_\mu$, the
three terms in \eqref{eq:lq-uniform-continuity} give the exact
decomposition
\[
\begin{aligned}
\int_{\mathbb R^d}
|\widetilde\xi_t(z+h)-\widetilde\xi_t(z)|^q\,\dd z
={}&
I_1(t,h)+I_2(t,h)+I_3(t,h).
\end{aligned}
\]
The remaining region, on which both $z$ and $z+h$ lie outside
$\Omega_\mu$, contributes zero. 
Here
\[
\Omega_\mu-h:=\{z-h:\ z\in\Omega_\mu\},
\]
and $t_0$ is as in \Cref{Prop:Lq-estimate} for a slightly larger
$\widetilde q\in
\left(q,\frac{d(d+2)}{d^2+2d-4}\right)$.
By \Cref{Prop:Lq-estimate} and H\"older's inequality with exponent
$\widetilde q$, we have
\begin{align}
\label{eq:I_2}
\sup_{t\in(0,t_0]}I_2(t,h)
\leq &
\mathcal L^d\!\left(
\Omega_\mu\setminus(\Omega_\mu-h)
\right)^l
\sup_{t\in(0,t_0]}
\left(
\int_{\Omega_\mu}
|\widetilde\xi_t(z)|^{\widetilde q}\,\dd z
\right)^{q/\widetilde q},
\end{align}
where $l:=1-\frac{q}{\widetilde q}>0.$ 
The third term satisfies the same estimate with $h$ replaced by
$-h$. Indeed, after the change of variables $y=z+h$,
\[
\begin{aligned}
I_3(t,h)
&=
\int_{\Omega_\mu\setminus(\Omega_\mu+h)}
|\widetilde\xi_t(y)|^q\,\dd y.
\end{aligned}
\]
Consequently,
\begin{align}
\label{eq:I_3}
\sup_{t\in(0,t_0]}I_3(t,h)
\leq{}&
\mathcal L^d\!\left(
\Omega_\mu\setminus(\Omega_\mu+h)
\right)^l
\sup_{t\in(0,t_0]}
\left(
\int_{\Omega_\mu}
|\widetilde\xi_t(y)|^{\widetilde q}\,\dd y
\right)^{q/\widetilde q}.
\end{align}
Since $\mathbf 1_{\Omega_\mu}(\cdot+h)\to
\mathbf 1_{\Omega_\mu}$ in $L^1(\mathbb R^d)$ as $\|h\|\to0$, we have
\[
\mathcal L^d\!\left(
\Omega_\mu\setminus(\Omega_\mu-h)
\right)
+
\mathcal L^d\!\left(
\Omega_\mu\setminus(\Omega_\mu+h)
\right)
\longrightarrow0.
\]
Thus, both $I_2$ and $I_3$ converge to zero uniformly in
$t\in(0,t_0]$.  It remains to control $I_1(t,h)$. As we showed in
\Cref{prop:Example-uniform}, the integrand may be ill-behaved in a
neighborhood of the perturbation point $\mathbf F_P(x_0)$. To resolve
this issue, we separate $I_1(t,h)$ into a contribution away from the
boundary and the pole, and a contribution near these sets. Define
\[
\Omega_\mu^\epsilon
:=
\left\{
x\in\Omega_\mu:
\operatorname{dist}(x,\partial\Omega_\mu)>\epsilon
\ \text{and}\
\|x-\mathbf F_P(x_0)\|>\epsilon
\right\}.
\]
This is the set of points in $\Omega_\mu$ that are at least
$\epsilon$ away from both $\partial\Omega_\mu$ and
$\mathbf F_P(x_0)$. Given $\beta>0$, choose
$\epsilon=\epsilon(\beta)>0$ such that $\mathcal L^d(\Omega_\mu\setminus\Omega_\mu^\epsilon)<\beta.$ 
Then
\begin{equation}\label{eq:lq-uniform-continuity-2}
\begin{aligned}
\int_{(\Omega_\mu-h)\cap\Omega_\mu}
|\widetilde\xi_t(z+h)-\widetilde\xi_t(z)|^q\,\dd z
&\leq
\underbrace{
\int_{(\Omega_\mu-h)\cap\Omega_\mu\cap\Omega_\mu^\epsilon}
|\widetilde\xi_t(z+h)-\widetilde\xi_t(z)|^q\,\dd z
}_{=:A(t,h)}
\\
&\qquad+
\underbrace{
\int_{((\Omega_\mu-h)\cap\Omega_\mu)\setminus\Omega_\mu^\epsilon}
|\widetilde\xi_t(z+h)-\widetilde\xi_t(z)|^q\,\dd z
}_{=:B(t,h)}.
\end{aligned}
\end{equation} 
Assume henceforth that $\|h\|<\frac{\epsilon}{2}.$ 
If $z\in\Omega_\mu^\epsilon$ and $z+h\in\Omega_\mu$, then, for every
$s\in[0,1]$,
$$ \operatorname{dist}(z+sh,\partial\Omega_\mu)
\geq
\operatorname{dist}(z,\partial\Omega_\mu)-s\|h\| >
\frac{\epsilon}{2},$$
and 
$$ \|z+sh-\mathbf F_P(x_0)\|
\geq
\|z-\mathbf F_P(x_0)\|-s\|h\|
>
\frac{\epsilon}{2}.$$
Thus, the whole line segment joining $z$ and $z+h$ is contained in
$\Omega_\mu^{\epsilon/2}$.
Since
\[
\overline{\Omega_\mu^{\epsilon/2}}
\Subset
\operatorname{int}(\Omega_\mu)
\setminus\{\mathbf F_P(x_0)\},
\]
\Cref{prop:Uniform-estimate}, applied on the enlarged interior region
$\Omega_\mu^{\epsilon/2}$, yields constants
$t_\epsilon,C_\epsilon>0$ such that
\[
\sup_{t\in(0,t_\epsilon]}
\|\nabla\widetilde\xi_t\|_
{L^\infty(\Omega_\mu^{\epsilon/2})}
\leq C_\epsilon.
\]
The fundamental theorem of calculus along the segment
$z+s h$, $s\in[0,1]$, therefore gives
\[
|\widetilde\xi_t(z+h)-\widetilde\xi_t(z)|
\leq C_\epsilon\|h\|
\]
for every
\[
z\in
(\Omega_\mu-h)\cap\Omega_\mu\cap\Omega_\mu^\epsilon,
\qquad
t\in(0,t_\epsilon].
\]
Consequently, 
\begin{equation}\label{eq:Ah-0}
\sup_{t\in(0,t_\epsilon]}A(t,h)
\leq
C_\epsilon\mathcal L^d(\Omega_\mu)\|h\|^q,
\qquad
\|h\|<\frac{\epsilon}{2}.
\end{equation}
If $t\in(t_\epsilon,t_0]$, note that
$\widetilde\xi_t$ is
$\operatorname{diam}(\Omega_P)/t$-Lipschitz. Indeed,
\[
\nabla\widetilde\xi_t
=
\nabla\xi_t
=
\frac{\mathbf Q_{P_t}-\mathbf Q_P}{t},
\]
and both $\mathbf Q_{P_t}$ and $\mathbf Q_P$ take values in
$\Omega_P$. Therefore,
\[
\|\nabla\widetilde\xi_t\|_{L^\infty(\Omega_\mu)}
\leq
\frac{\operatorname{diam}(\Omega_P)}{t}.
\]
Hence,
\[
\begin{aligned}
\sup_{t\in(t_\epsilon,t_0]}A(t,h)
&\leq
\sup_{t\in(t_\epsilon,t_0]}
\frac{\operatorname{diam}(\Omega_P)^q}{t^q}
\mathcal L^d(\Omega_\mu)\|h\|^q
\\
&\leq
\frac{\operatorname{diam}(\Omega_P)^q}{t_\epsilon^q}
\mathcal L^d(\Omega_\mu)\|h\|^q.
\end{aligned}
\]
Combining this estimate with \eqref{eq:Ah-0}, we obtain
\begin{equation}\label{eq:Ah}
\sup_{t\in(0,t_0]}A(t,h)
\leq
C_\epsilon\mathcal L^d(\Omega_\mu)\|h\|^q,
\qquad
\|h\|<\frac{\epsilon}{2},
\end{equation}
for a possibly different constant $C_\epsilon$. For the second term in
\eqref{eq:lq-uniform-continuity-2}, using $|a-b|^q\leq2^{q-1}(|a|^q+|b|^q),$ 
we obtain
\begin{align*}
B(t,h)
&\leq
2^{q-1}
\int_{((\Omega_\mu-h)\cap\Omega_\mu)
\setminus\Omega_\mu^\epsilon}
|\widetilde\xi_t(z+h)|^q\,\dd z
+
2^{q-1}
\int_{((\Omega_\mu-h)\cap\Omega_\mu)
\setminus\Omega_\mu^\epsilon}
|\widetilde\xi_t(z)|^q\,\dd z.
\end{align*}
After the change of variables $y=z+h$, the first integral equals
\[
\int_{E_h}|\widetilde\xi_t(y)|^q\,\dd y,
\qquad
E_h
:=
\left(
((\Omega_\mu-h)\cap\Omega_\mu)
\setminus\Omega_\mu^\epsilon
\right)+h.
\]
Moreover,
\[
\mathcal L^d(E_h)
=
\mathcal L^d\!\left(
((\Omega_\mu-h)\cap\Omega_\mu)
\setminus\Omega_\mu^\epsilon
\right)
\leq
\mathcal L^d(\Omega_\mu\setminus\Omega_\mu^\epsilon)
<\beta.
\]
Therefore, applying H\"older's inequality separately to the two
integrals and using \Cref{Prop:Lq-estimate}, we obtain
\begin{align*}
B(t,h)
&\leq
2^{q-1}\mathcal L^d(E_h)^l
\left(
\int_{\Omega_\mu}
|\widetilde\xi_t(y)|^{\widetilde q}\,\dd y
\right)^{q/\widetilde q}+
2^{q-1}
\mathcal L^d(\Omega_\mu\setminus\Omega_\mu^\epsilon)^l
\left(
\int_{\Omega_\mu}
|\widetilde\xi_t(z)|^{\widetilde q}\,\dd z
\right)^{q/\widetilde q} \lesssim\beta^l,
\end{align*}
uniformly for $t\in(0,t_0]$ and
$\|h\|<\epsilon/2$. Combining this estimate with
\eqref{eq:lq-uniform-continuity-2} and \eqref{eq:Ah}, we obtain
\[
\limsup_{\|h\|\to0}
\sup_{t\in(0,t_0]}
I_1(t,h)
\leq C\beta^l.
\]
Since $\beta>0$ is arbitrary, it follows that
\[
\lim_{\|h\|\to0}
\sup_{t\in(0,t_0]}I_1(t,h)=0.
\]
Combining this conclusion with \eqref{eq:I_2} and
\eqref{eq:I_3}, we obtain
\[
\lim_{\|h\|\to0}
\sup_{t\in(0,t_0]}
\int_{\mathbb R^d}
|\widetilde\xi_t(z+h)-\widetilde\xi_t(z)|^q\,\dd z
=0,
\]
which is precisely the translation condition required by the
Kolmogorov--Riesz--Fr\'echet theorem. 
Thus, $\{\widetilde\xi_t\}_{t\leq t_0}$ is relatively compact in
$L^q(\Omega_\mu)$. Therefore, there exists a subsequence
$t_k\downarrow0$ and a function $G_{x_0}\in L^q(\Omega_\mu)$ such that
\[
\widetilde\xi_{t_k}\longrightarrow G_{x_0}
\qquad\text{in }L^q(\Omega_\mu).
\]
Since, by \Cref{lemma:test-linearization},
\[
\lim_{t\downarrow0}
\left|
g_\gamma(\mathbf F_P(x_0))
-
\int_{\Omega_\mu}g_\gamma\,\dd\mu
+
\int_{\Omega_\mu}
{\rm div}\!\left(
f_\mu[\nabla\BQ_P]^{-1}\nabla g_\gamma
\right)
\widetilde\xi_t\,\dd z
\right|
=0,
\]
and
\[
{\rm div}\!\left(
f_\mu[\nabla\BQ_P]^{-1}\nabla g_\gamma
\right)\in L^p(\Omega_\mu),
\qquad
\frac1p+\frac1q=1,
\]
we may pass to the limit along the subsequence by H\"older's
inequality. Hence,
\begin{equation}
\label{eq:Solbable}
\int_{\Omega_\mu}
{\rm div}\!\left(
f_\mu[\nabla\BQ_P]^{-1}\nabla g_\gamma
\right)
G_{x_0}\,\dd z
=
\int_{\Omega_\mu}g_\gamma\,\dd\mu
-
g_\gamma(\mathbf F_P(x_0)).
\end{equation}
This holds for every $g_\gamma\in W^{2,p}(\Omega_\mu)$ satisfying $\int_{\Omega_\mu}g_\gamma\,\dd z=0$ 
and
\[
\left\langle
f_\mu[\nabla\BQ_P]^{-1}\nabla g_\gamma,
n_\mu
\right\rangle
=0
\qquad\text{on }\partial\Omega_\mu
\]
in the trace sense; see~\Cref{sec:solvability}.  It remains to prove uniqueness of solutions to~\eqref{eq:Solbable} in $L^q(\Omega_\mu)$ such that $\int_{\Omega_\mu} G_{x_0} \dd \mu =0$, from which we may conclude that 
\[
\tilde\xi_t \to G_{x_0}\quad {\rm in}\,\, L^q(\Omega_\mu).
\] 
Furthermore, \Cref{prop:Uniform-estimate} and the Arzel\`a--Ascoli theorem show that the convergence is also in $ \cC^{2,\alpha'}_{\rm Loc}({\rm int}(\Omega_\mu) \setminus \{\BF_P(x_0)\}) $ for any $0<\alpha'<\alpha$, and it follows that $G_{x_0}\in \cC^{2,\alpha}_{\rm Loc}({\rm int}(\Omega_\mu) \setminus \{\BF_P(x_0)\}) \cap L^q(\Omega_\mu)$.

We now present the following proposition, which establishes uniqueness of solutions to~\eqref{eq:Solbable}.
\begin{proposition}\label{prop:uniqueness}
Fix $q \in \left(1, \frac{d(d+2)}{d^{2}+2d-4}\right)$. There exists a unique solution $G_{x_0} \in L^q(\Omega_\mu)$ to \eqref{eq:Solbable} such that $\int_{\Omega_\mu} G_{x_0} \dd \mu =0$. 
\end{proposition}
\begin{proof}
    Note that any limit point of the compact sequence $\{\xi_t-m_t\}_{t\leq t_0}$ belongs to $\cC^{2,\alpha}_{\rm Loc}({\rm int}(\Omega_\mu) \setminus \{\BF_P(x_0)\}) \cap  L^q(\Omega_\mu)$ and solves \eqref{eq:Solbable}. This shows the existence. To show the uniqueness, assume that $G$ solves the homogeneous equation 
$$  \int_{\Omega_\mu}  {\rm div} (f_\mu [\nabla\BQ_{P}]^{-1} \nabla g ) G\dd z=  0 $$
for every $g\in {W}^{2,p}(\Omega_\mu)$ (with $p$ being the H\"older conjugate of $q$) such that $\int_{\Omega_\mu} g\dd z = 0$ and $\inner{[\nabla \BQ_P]^{-1}\nabla g}{n_\mu}f_\mu = 0$ on $\partial\Omega_\mu$ in the trace sense. Define 
 $$  \gamma= \frac{{\rm sgn}(G-m) |G-m|^{q-1}}{\|G-m\|_{L^{q}(\Omega_\mu)}^{q/p}}  $$
where $m$ is such that 
$$ \int_{\Omega_\mu} {\rm sgn}(G-m) |G-m|^{q-1}\dd z=0.$$ When $\|G-m\|_{L^{q}(\Omega_\mu)}=0$, the problem is trivial. Then $\gamma\in L^p(\Omega_\mu)$ and $\int_{\Omega_\mu}\gamma \dd z=0$. Therefore, $g_\gamma =\BL^{-1}(\gamma,0)\in W^{2,p}(\Omega_\mu)$ (recall that $\BL$ is defined in \eqref{eq:operator-L}) such that $\int_{\Omega_\mu} g_\gamma \dd z=0$ satisfies 
\[
 \|G- m\|_{L^q(\Omega_\mu)} = \int_{\Omega_\mu}  {\rm div} (f_\mu [\nabla\BQ_{P}]^{-1} \nabla g_\gamma ) G\dd z=  0\,.
\]
This concludes the uniqueness (up to constant) in $L^q(\Omega_\mu)$. The normalization $\int_{\Omega_\mu} G_{x_0} \dd \mu =0$ gives the uniqueness.  
\end{proof} 
We conclude this section by the following lemma, which collects auxiliary results for the proof of~\Cref{thm:Unbounded-IF}.
\begin{lemma}\label{lemma:Lp-estimates-of-G}
Under the setting of \Cref{Thm:InfluenceQuantiles-main}, the following hold.
\begin{enumerate}[label = (\roman*)]
\item  If $z_0:=\mathbf F_P(x_0)\in {\rm int}(\Omega_\mu)$, the limit $G_{x_0}$ solves  
\[
{\rm div} (f_\mu [\nabla\BQ_{P}]^{-1} \nabla u )=\mu -\delta_{z_0}
\]
in the distributional sense.  If $z_0\in\partial \Omega_\mu $, the limit $G_{x_0}$ solves  
\[
{\rm div} (f_\mu [\nabla\BQ_{P}]^{-1} \nabla u )=\mu \qquad \text{in}\ {\rm int}(\Omega_\mu)
\]
in the distributional sense. 
    \item For every
$p\in\left(1,\frac{d}{d-2}\right),$
there exists a constant $C_p>0$, independent of $x_0$, such that
\[
\|G_{x_0}\|_{L^p(\Omega_\mu)}\leq C_p.
\]
\end{enumerate}
\end{lemma}
\begin{proof}
\textit{(i)}. Recall that the potential $G_{x_0}$ of the influence function at $x_0$  solves~\eqref{eq:Solbable}, i.e.,
\begin{equation}\label{eq:transposition-G-boundary}
  \int_{\Omega_\mu}  {\rm div} (f_\mu [\nabla\BQ_{P}]^{-1} \nabla g ) G_{x_0}\dd z=\int_{\Omega_\mu} g \dd(\mu -\delta_{z_0}),
\end{equation}
for all $g\in {W}^{2,p_1}(\Omega_\mu)$, with  
 $p_1>\frac{d(d+2)}{4}$, such that $\int_{\Omega_\mu} g \dd z=0$ and $\langle f_\mu[\nabla\mathbf Q_P]^{-1} \nabla g,n_\mu\rangle=0$ on $\partial\Omega_\mu$. For $g \in \cC_c^\infty(\Omega_\mu)$, we may subtract a constant to obtain $\tilde g$ such that $\int_{\Omega_\mu} \tilde g \dd z=0$. Then
 \begin{align*}
     \int_{\Omega_\mu}  {\rm div} (f_\mu [\nabla\BQ_{P}]^{-1} \nabla g ) G_{x_0}\dd z
     &=  \int_{\Omega_\mu}  {\rm div} (f_\mu [\nabla\BQ_{P}]^{-1} \nabla \tilde g ) G_{x_0}\dd z\\
     &= \int_{\Omega_\mu} \tilde g \dd(\mu -\delta_{z_0})
     = \int_{\Omega_\mu} g \dd(\mu -\delta_{z_0}).
 \end{align*}
Thus, by the definition of the distributional divergence,
 \[
 \int_{\Omega_\mu} {\rm div} (f_\mu [\nabla\BQ_{P}]^{-1} \nabla G_{x_0} ) g \dd z=\int_{\Omega_\mu} g \dd (\mu -\delta_{z_0}) .
 \]
 In other words, $G_{x_0}$ solves, in the distributional sense,
\[
    \BL_1(u):= {\rm div} (f_\mu [\nabla\BQ_{P}]^{-1} \nabla u )=\mu -\delta_{z_0} 
\]
if $z_0\in {\rm int}(\Omega_\mu)$,  and 
\[
{\rm div} (f_\mu [\nabla\BQ_{P}]^{-1} \nabla u )=\mu \qquad \text{in}\ {\rm int}(\Omega_\mu)
\]
otherwise. 
 This concludes  the proof of~(i).

\textit{(ii)}.
 Fix an exponent $p_2$ such that $\frac d2<p_2<p_1,$ 
and define
\[
L^{p_2}_0(\Omega_\mu)
:=
\left\{
h\in L^{p_2}(\Omega_\mu):
\int_{\Omega_\mu}h\,\dd z=0
\right\}.
\]
For every $h\in L^{p_1}_0(\Omega_\mu)$, let $g_h\in W^{2,p_1}(\Omega_\mu)$ be the unique
weak solution of
\[
\begin{cases}
{\rm div}(f_\mu[\nabla\mathbf Q_P]^{-1}\nabla g_h)=h
&\text{in }{\rm int}(\Omega_\mu),\\
\langle f_\mu[\nabla\mathbf Q_P]^{-1}\nabla g_h,n_\mu\rangle=0
&\text{on }\partial\Omega_\mu,\\
\displaystyle\int_{\Omega_\mu}g_h\,\dd z=0.
\end{cases}
\]
As $p_2<p_1$, $h\in L^{p_2}_0(\Omega_\mu)$ as well. By \Cref{lemma: existence and uniqueness linearized MA}, $\|g_h\|_{W^{2,p_2}(\Omega_\mu)}
\leq C\|h\|_{L^{p_2}(\Omega_\mu)}.$ 
Since $p_2>d/2$, the Sobolev embedding gives $W^{2,p_2}(\Omega_\mu)
\hookrightarrow C({\Omega_\mu}),$ 
and hence $\|g_h\|_{L^\infty(\Omega_\mu)}
\leq C\|h\|_{L^{p_2}(\Omega_\mu)}.$ 
Using \eqref{eq:transposition-G-boundary}, we obtain
\begin{align*}
\left|
\int_{\Omega_\mu}hG_{x_0}\,\dd z
\right|
&=
\left|
\int_{\Omega_\mu}g_h\,\dd\mu-g_h(z_0)
\right|\leq
2\|g_h\|_{L^\infty(\Omega_\mu)}\leq
C\|h\|_{L^{p_2}(\Omega_\mu)}.
\end{align*}
The constant is independent of
$z_0\in {\Omega_\mu}$. Because $p_1>p_2$ and $\Omega_\mu$ is bounded,
$L^{p_1}_0(\Omega_\mu)$ is dense in
$L^{p_2}_0(\Omega_\mu)$. Therefore, the functional $h\longmapsto
\int_{\Omega_\mu}hG_{x_0}\,\dd z$ 
admits a unique continuous extension to
$L^{p_2}_0(\Omega_\mu)$. Define
\[
\widetilde G_{x_0}
:=
G_{x_0}
-
\frac{1}{\mathcal L^d(\Omega_\mu)}
\int_{\Omega_\mu}G_{x_0}\,\dd z.
\]
For every $h\in L^{p_1}_0(\Omega_\mu)$,
\[
\int_{\Omega_\mu}h\widetilde G_{x_0}\,\dd z
=
\int_{\Omega_\mu}hG_{x_0}\,\dd z.
\]
For $h\in L^{p_2}(\Omega_\mu)$, set
\[
\Theta(h)
:=
\Theta_0\left(
h-
\frac{1}{\mathcal L^d(\Omega_\mu)}
\int_{\Omega_\mu}h\,\dd z
\right),
\]
where $\Theta_0$ denotes the extension to
$L^{p_2}_0(\Omega_\mu)$. Since the mean-zero projection is bounded
on $L^{p_2}(\Omega_\mu)$, we have
$|\Theta(h)|
\leq C\|h\|_{L^{p_2}(\Omega_\mu)}.$ 
Thus $\Theta\in\bigl(L^{p_2}(\Omega_\mu)\bigr)'.$
Let $q_2$ be the conjugate exponent of $p_2$. By the representation
theorem for the dual of $L^{p_2}$ \citep[Theorem~4.11]{brezis2011functional}, there exists
$H\in L^{q_2}(\Omega_\mu)$ such that
\[
\Theta(h)=\int_{\Omega_\mu}hH\,\dd z
\]
for every $h\in L^{p_2}(\Omega_\mu)$, with $\|H\|_{L^{q_2}(\Omega_\mu)}\leq C.$ 
For every $h\in L^{p_1}(\Omega_\mu)$,
\[
\int_{\Omega_\mu}hH\,\dd z
=
\int_{\Omega_\mu}h\widetilde G_{x_0}\,\dd z.
\]
Hence $\widetilde G_{x_0}=H$, a.e. in  $\Omega_\mu$ 
and therefore $\|\widetilde G_{x_0}\|_{L^{q_2}(\Omega_\mu)}
\leq C.$  It remains to control the constant removed from $G_{x_0}$. Since $\int_{\Omega_\mu}G_{x_0}\,\dd\mu=0,$ 
we have
\[
0
=
\int_{\Omega_\mu}\widetilde G_{x_0}\,\dd\mu
+
\frac{1}{\mathcal L^d(\Omega_\mu)}
\int_{\Omega_\mu}G_{x_0}\,\dd z.
\]
Consequently,
\[
\left|
\frac{1}{\mathcal L^d(\Omega_\mu)}
\int_{\Omega_\mu}G_{x_0}\,\dd z
\right|
\leq
\|f_\mu\|_{L^{p_2}(\Omega_\mu)}
\|\widetilde G_{x_0}\|_{L^{q_2}(\Omega_\mu)}
\leq C.
\]
It follows that $\|G_{x_0}\|_{L^{q_2}(\Omega_\mu)}
\leq C,$ 
uniformly in $x_0$. When $d>2$, the condition $p_2>d/2$ is equivalent to
$q_2<\frac{d}{d-2}.$ 
When $d=2$, every finite $q_2>1$ can be obtained. In particular,
\[
\|G_{x_0}\|_{L^p(\Omega_\mu)}
\leq C_p,
\qquad
1<p<\frac{d}{d-2}=+\infty,
\]
with a constant $C_p$ independent of $x_0$.
\end{proof}
\subsection{Proof of~\Cref{thm:Unbounded-IF}}\label{sec:proof-unboundedness}

We will first establish the local statement~(ii), which will be used to prove~(i).
To avoid confusion in notation, we fix a point of evaluation
$z^*\in {\rm int}(\Omega_\mu)$ instead of $z$ as in the statement.
Pick $z_0=\BF_P(x_0)\in\Omega_\mu$ and set
\[
R:=\|z_0-z^*\|,
\qquad
\delta:={\rm dist}(z^*,\partial\Omega_\mu).
\]
Fix $R_0\in(0,\min\{\delta/2,1\}]$, to be decreased below if necessary,
and assume throughout the proof of~(ii) that $0<R\leq R_0.$ 
Then
\[
\operatorname{dist}(z_0,\partial\Omega_\mu)
\geq
\operatorname{dist}(z^*,\partial\Omega_\mu)-\|z_0-z^*\|
=
\delta-R
\geq
\frac{\delta}{2}.
\]
Thus, throughout the proof of~(ii), $z_0$ is an interior point of
$\Omega_\mu$, and the equation for $G_{x_0}$ has an interior Dirac
source at $z_0$.
In this proof, we denote by $C$ a constant that does not depend on $R$
and may change from line to line. Similarly, the notation $\lesssim$
hides a constant that does not depend on $R$. These constants may depend
on $\delta$. 

\textit{Step~1:} Write
$A_{0}:=f_\mu(z_0)[\nabla \BQ_P(z_0)]^{-1}.$
Define the operator
$\BL^{0}_1(u)=\operatorname{div}(A_0\nabla u).$
A weak solution $\BL^{0}_1(u)=-\delta_{z_0}$ is
\cite[see, e.g.,][Section~2.2.1]{Evans-PDE}
\begin{equation}\label{eq:Phi-x0}
    \Phi_{x_0}(z)=
\begin{cases}
\displaystyle
\frac{1}{(d-2)\,\cH^{d-1}(\cS^{d-1})}\,
\frac{1}{\sqrt{\det(A_0)}}\,
\big( \langle  z-z_0, A_0^{-1} ( z-z_0)\rangle \big)^{\frac{2-d}{2}},&
d\ge3, \\[1.4em]
\displaystyle
-\frac{1}{4\pi\sqrt{\det(A_0)}}\,
\log\!\big( \langle z-z_0, A_0^{-1} ( z-z_0)\rangle \big),
&d=2 .
\end{cases}
\end{equation}
On the other hand, \Cref{lemma:Lp-estimates-of-G}~(i) states that
$G_{x_0}$ solves
\[
{\rm div} (f_\mu [\nabla\BQ_{P}]^{-1} \nabla u )
=
\mu-\delta_{z_0}
\]
in the distributional sense. Therefore, $v:=G_{x_0}-\Phi_{x_0}$ solves
\[
\BL_1(v)
=
\BL_1(G_{x_0}-\Phi_{x_0})
=
f_\mu-[\BL_1-\BL^0_1](\Phi_{x_0})
=
f_\mu-\diver{
[f_\mu[\nabla \BQ_P]^{-1}-A_0]\nabla \Phi_{x_0}
}
=:f_0
\]
in the distributional sense, where $\BL_1$ is strictly elliptic
by~\Cref{thm:Cafarrelli}. We proceed with controlling the regularity of
$v$. In particular, we will show that
\begin{equation}\label{eq:gradient-v-infinity}
    \|\nabla v(z^*)\|
    \lesssim
    \frac{1}{R^{d-1-\alpha}}.
\end{equation}
This, combined with the fact that
$\|\nabla\Phi_{x_0}(z)\|
\asymp\|z-z_0\|^{-(d-1)}$, yields that
\[
\|\nabla G_{x_0}(z^*)\|
\leq
\|\nabla\Phi_{x_0}(z^*)\|
+
\|\nabla v(z^*)\|
\lesssim
\frac{1}{R^{d-1}}
+
\frac{1}{R^{d-1-\alpha}}
\lesssim
\frac{1}{R^{d-1}}.
\]
Moreover, there exist constants $c,C>0$, independent of $R$, such that
\[
\begin{aligned}
\|\nabla G_{x_0}(z^*)\|
&\geq
\|\nabla\Phi_{x_0}(z^*)\|
-
\|\nabla v(z^*)\|
\\
&\geq
\frac{c}{R^{d-1}}
-
\frac{C}{R^{d-1-\alpha}}
=
\frac{c-CR^\alpha}{R^{d-1}}.
\end{aligned}
\]
By decreasing $R_0$ if necessary so that
$CR_0^\alpha\leq c/2$, we obtain
\[
\|\nabla G_{x_0}(z^*)\|
\geq
\frac{c}{2R^{d-1}}.
\]
In view of $R=\|z^*-z_0\|$, we arrive at the desired claim. The rest of the proof is devoted to
establishing~\eqref{eq:gradient-v-infinity}.

\textit{Step~2.1:} We will first focus on the case when $d\geq3$.
To that end, we proceed with showing that
\[
\|v\|_{L^\infty(\BB_{R/4}(z^*))}
\lesssim
\frac{1}{R^{d-2-\alpha}}.
\]
Recall $A_0=(f_\mu[\nabla\BQ_P]^{-1})(z_0)$. It follows by direct
computation that
\begin{equation}\label{eq:f_0}
    \begin{aligned}
    f_0
    =&\,
    f_\mu
    -\Bigl(
    \inner{\nabla f_\mu}
    {[\nabla \BQ_P]^{-1}\nabla \Phi_{x_0}}
    +
    f_\mu
    \inner{\diver{[\nabla\BQ_P]^{-1}}}
    {\nabla \Phi_{x_0}}
    \\
    &\hspace{7cm}
    +
    \tr\bigl(
    [f_\mu[\nabla\BQ_P]^{-1}-A_0]
    \nabla^2\Phi_{x_0}
    \bigr)
    \Bigr),
    \end{aligned}
\end{equation}
where the divergence of a matrix is to be understood row-wise.
Since $f_\mu\in\cC^{1,\alpha}(\Omega_\mu)$ and
$[\nabla\BQ_P]^{-1}\in\cC^{1,\alpha}(\Omega_\mu)$
by~\Cref{thm:Cafarrelli}, we know that
$f_\mu[\nabla\BQ_P]^{-1}\in\cC^{1,\alpha}(\Omega_\mu)$.
Hence, substituting \eqref{eq:Phi-x0} into \eqref{eq:f_0} gives
\[
|f_0(z)|
\lesssim
\frac{\|z-z_0\|}{\|z-z_0\|^d}
\lesssim
\frac{1}{\|z-z_0\|^{d-1}}
\lesssim
\frac{1}{R^{d-1}},
\]
for any $z\in\BB_{R/2}(z^*)$
(so that $\|z-z_0\|\geq R/2$).
Thus,
\begin{equation}\label{eq:f_0-d}
    \|f_0\|_{L^d(\BB_{R/2}(z^*))}
    \lesssim
    \frac{1}{R^{d-2}}.
\end{equation}
Note that $v$ solves $\BL_1(v)=f_0$ in $\Omega_\mu$ in the distributional sense.
Therefore, the local maximum principle
\cite[Theorem~9.20]{GilbargTrudinger.Book} yields, for $q>0$,
\begin{align}
    \|v\|_{L^\infty(\BB_{R/4}(z^*))}
    &\lesssim
    \frac{\|v\|_{L^q(\BB_{R/2}(z^*))}}{R^{\frac{d}{q}}}
    +
    R\|f_0\|_{L^d(\BB_{R/2}(z^*))}
    \notag
    \\
    &\lesssim
    \frac{\|v\|_{L^q(\BB_{R/2}(z^*))}}{R^{\frac{d}{q}}}
    +
    \frac{1}{R^{d-3}}.
    \label{eq:v}
\end{align}
Here we also used \eqref{eq:f_0-d}.

We now control $\|v\|_{L^q(\BB_{R/2}(z^*))}$.
Since, for $\eta\in(0,1)$,
\begin{align*}
\int_{\Omega_\mu}
|f_0(z)|^{\frac{d-\eta}{d-1}}
\dd z
&\lesssim
\int_{\Omega_\mu}
\frac{1}{\|z-z_0\|^{d-\eta}}
\dd z
\\
&\lesssim
\int_0^{{\rm diam}(\Omega_\mu)}
s^{-(1-\eta)}
\dd s
<\infty,
\end{align*}
we know that $\|f_0\|_{L^p(\Omega_\mu)}<\infty$ for
$p<\frac{d}{d-1}$. On the other hand,
\cite[Theorem~12.1]{Taira.2024.book} yields, for $p>1$,
\begin{equation}\label{eq:v-W2p}
\|v\|_{W^{2,p}(\mathbb{B}_{\delta/2}(z^*))}
\lesssim
\|v\|_{L^p(\Omega_\mu)}
+
\|f_0\|_{L^p(\Omega_\mu)}.
\end{equation}
Furthermore, \Cref{lemma:Lp-estimates-of-G}~(ii) shows that
$\|G_{x_0}\|_{L^q(\Omega_\mu)}\leq C$ for
$q<\frac{d}{d-1}$, and direct computation gives
$\|\Phi_{x_0}\|_{L^p(\Omega_\mu)}\leq C$ for
$p<\frac{d}{d-2}$.
Therefore, $v\in L^p(\Omega_\mu)$ for
$p\in(1,\frac{d}{d-1})$, and by~\eqref{eq:v-W2p}, we have
\begin{equation}
    \label{eq:bound-v-W-2-p}
    \|v\|_{W^{2,p}(\mathbb{B}_{\delta/2}(z^*))}
    \leq C,
    \qquad
    p\in\left(1,\frac{d}{d-1}\right).
\end{equation}
By the Sobolev embedding theorem
\cite[Corollary~9.15]{brezis2011functional}, we know that, for $d\geq3$,
\[
\|v\|_{L^q(\mathbb{B}_{\delta/2}(z^*))}
\lesssim
\|v\|_{W^{2,p}(\mathbb{B}_{\delta/2}(z^*))},
\qquad
q<\frac{dp}{d-2p},
\]
with the convention $1/0=\infty$. For $p<\frac{d}{d-1}$ and using
\eqref{eq:bound-v-W-2-p}, we conclude that
\[
\|v\|_{L^q(\BB_{R/2}(z^*))}
\leq C,
\qquad
q<\frac{d}{d-3},
\qquad
d\geq3.
\]
Recall that $\alpha\in(0,1)$. Taking
$q=\frac{d}{d-2-\alpha}$ in~\eqref{eq:v}, we then derive
by~\eqref{eq:f_0-d} that
\begin{equation}\label{eq:v_infty}
    \|v\|_{L^\infty(\BB_{R/4}(z^*))}
    \lesssim
    \frac{1}{R^{d-2-\alpha}}
    +
    \frac{1}{R^{d-3}}
    \lesssim
    \frac{1}{R^{d-2-\alpha}},
    \qquad d\geq3,
\end{equation}
where we used $R\leq R_0\leq1$ in the last inequality.

\textit{Step~2.2:}
We are now ready to prove~\eqref{eq:gradient-v-infinity} when $d\geq3$,
thus concluding the proof of~(ii). Applying Schauder's interior estimates
\cite[Theorem~6.2]{GilbargTrudinger.Book} yields that
\begin{equation}\label{eq:gradient-v}
    \frac{R}{8}
    \|\nabla v\|_{L^\infty(\BB_{R/8}(z^*))}
    \lesssim
    \|v\|_{L^\infty(\BB_{R/4}(z^*))}
    + R^2\|f_0\|_{L^\infty(\BB_{R/4}(z^*)} + R^{2+\alpha}
    [f_0]_{\cC^\alpha(\BB_{R/4}(z^*))}.
\end{equation}
By direct computation again, as in~\eqref{eq:f_0}, we have
\[
\|f_0\|_{L^\infty(\BB_{R/4}(z^*)}\lesssim \frac{1}{R^{d-1}}\quad {\rm and}\quad [f_0]_{\cC^\alpha(\BB_{R/4}(z^*))}
\lesssim
\frac{R^\alpha}{R^{d+\alpha}}
=
\frac{1}{R^d}.
\]
Combining this with~\eqref{eq:v_infty} and~\eqref{eq:gradient-v}
shows that
\[
\begin{aligned}
\|\nabla v\|_{L^\infty(\BB_{R/8}(z^*))}
&\lesssim
\frac{1}{R^{d-1-\alpha}}
+ R\frac{1}{R^{d-1}} + 
R^{1+\alpha}\frac{1}{R^d}
\\
&\lesssim
\frac{1}{R^{d-1-\alpha}}.
\end{aligned}
\]
This proves~\eqref{eq:gradient-v-infinity} when $d\geq3$.

\textit{Step~3:}
Finally, we establish~\eqref{eq:gradient-v-infinity} when $d=2$.
Redefine $\tilde v=v-v(z^*)$ and note that $\tilde v$ solves the same
equation and has the same gradient.
Choose
\[
p\in\left(\frac{2}{2-\alpha},2\right).
\]
Applying the Sobolev embedding theorem
\cite[Corollary~9.15]{brezis2011functional} and
\eqref{eq:bound-v-W-2-p} gives
\begin{equation}\label{eq:gradient-v-infinity-d=2}
    \|\tilde v\|_{\cC^{0,\alpha}(\BB_{\delta/2}(z^*))}
    \lesssim
    \|v\|_{W^{2,p}(\BB_{\delta/2}(z^*))}
    \leq C.
\end{equation}
 Since $\tilde v(z^*)=0$, we obtain
\[
\|\tilde v\|_{L^\infty(\BB_{R/4}(z^*))}
\leq
\frac{R^\alpha}{4^\alpha}
[\tilde v]_{\cC^{0,\alpha}(\BB_{\delta/2}(z^*))}
\lesssim
R^\alpha.
\]
Moreover, the preceding estimate for $f_0$, with $d=2$, gives
\[
\|f_0\|_{\cC^\alpha(\BB_{R/4}(z^*))}
\lesssim
\frac{1}{R^2}.
\]
Applying~\eqref{eq:gradient-v} to $\tilde v$ and recalling that
$\nabla\tilde v=\nabla v$, we obtain
\[
\begin{aligned}
\frac{R}{8}
\|\nabla v\|_{L^\infty(\BB_{R/8}(z^*))}
&\lesssim
R^\alpha
+
R^{2+\alpha}\frac{1}{R^2}\lesssim
R^\alpha.
\end{aligned}
\]
Consequently,
\[
\|\nabla v(z^*)\|
\lesssim
\frac{1}{R^{1-\alpha}},
\]
which proves~\eqref{eq:gradient-v-infinity} when $d=2$.
Finally, by the continuity of $\BF_P$ at $\BQ_P(z^*)$, the condition
$R\leq R_0$ holds whenever $x_0$ belongs to a sufficiently small
neighborhood of $\BQ_P(z^*)$. This completes the proof of~(ii).

Now we prove {\it (i)} by contradiction. Suppose not. Then there exists
a sequence $\{x_n\}_n\subset\Omega_P$ such that
\begin{equation}\label{eq:unbounded-IF}
    \|\BI(x_n;\BQ_P(z^*))\|
    \|z^*-\BF_P(x_n)\|^{d-1}
    \to\infty.
\end{equation}
As $\Omega_\mu$ is bounded and closed, the sequence
$z_n=\BF_P(x_n)\in\Omega_\mu$ admits a limit
$z_\infty\in\Omega_\mu$ along a subsequence. If $z_\infty=z^*$, then
\Cref{thm:Unbounded-IF}~(ii) yields a contradiction.
Hence $z_\infty\neq z^*$. Choose $s>0$ such that
\[
2s<
\min\bigl\{
\delta,\|z_\infty-z^*\|
\bigr\}.
\]
Then, for all sufficiently large $n$,
$\|z_n-z^*\|>s$. In particular,
$\BB_s(z^*)\Subset\Omega_\mu$ and
$\BB_s(z^*)\cap\{z_n\}=\emptyset$ for all sufficiently large $n$.
Thus,
\[
\BL_1(G_{x_n})
=
{\rm div}
\bigl(
f_\mu[\nabla\BQ_P]^{-1}\nabla G_{x_n}
\bigr)
=
f_\mu
\qquad\text{in }\BB_s(z^*).
\]
This interior equation is valid regardless of whether
$z_n=\BF_P(x_n)$ belongs to ${\rm int}(\Omega_\mu)$ or to
$\partial\Omega_\mu$. Indeed, the ball
$\BB_s(z^*)\Subset{\rm int}(\Omega_\mu)$ is separated from $z_n$.
When $z_n\in{\rm int}(\Omega_\mu)$, the interior Dirac source is
supported outside $\BB_s(z^*)$; when
$z_n\in\partial\Omega_\mu$, the singularity is at the boundary and does not affect the equation on this
interior ball. In both cases we can use \Cref{lemma:Lp-estimates-of-G}.     
By the local maximum principle
\cite[Theorem~9.20]{GilbargTrudinger.Book}, together with the uniform $L^p(\Omega_\mu)$ estimate from
\Cref{lemma:Lp-estimates-of-G}~(ii), we have
\[
\|G_{x_n}\|_{L^\infty(\BB_{s/2}(z^*))}
\lesssim
\|G_{x_n}\|_{L^p(\BB_s(z^*))}
+
\|f_\mu\|_{L^d(\BB_s(z^*))}
\lesssim
1,
\qquad
p\in\left(1,\frac{d}{d-1}\right).
\] 
Applying \cite[Corollary~6.3]{GilbargTrudinger.Book} yields
\begin{equation}\label{eq:gradient-G}
    \frac{s}{4}
    \|\nabla G_{x_n}\|_{L^\infty(\BB_{s/4}(z^*))}
    \lesssim
    \|G_{x_n}\|_{L^\infty(\BB_{s/2}(z^*))}
    +
    \|f_\mu\|_{\cC^\alpha(\Omega_\mu)}
    \lesssim
    1,
\end{equation}
where the underlying constant is independent of $n$. Since
$\BI(x_n;\BQ_P(z^*))=\nabla G_{x_n}(z^*)$ and
$\lim_n\|z_n-z^*\|=\|z_\infty-z^*\|$, we derive
by~\eqref{eq:gradient-G} that
\[
\limsup_n
\|\BI(x_n;\BQ_P(z^*))\|
\|z_n-z^*\|^{d-1}
\lesssim
\|z_\infty-z^*\|^{d-1}
<\infty.
\]
This contradicts~\eqref{eq:unbounded-IF}.

\subsection{Proof of~\Cref{prop:shape}}\label{sec:proof-shape}
We begin the proof with a few definitions. 
An \emph{ellipsoid} $\cE$ is defined as the image of the unit ball $\BB_1(0)$ under an affine transformation, i.e.,  $\cE = m + A(\BB_1(0))$ with $A$ symmetric and positive definite.  For $r > 0$, we write $r \cE$ as the \emph{dilation} of $\cE$ with center $m$, that is $r \cE = m + r A(\BB_1(0))$. Intuitively, the dilation of a set means keeping the center but changing the diameter of the set.  The \emph{barycenter} of a Borel set $S$ with nonempty interior is $\int_S z \dd z/\cL_d(S)$. 
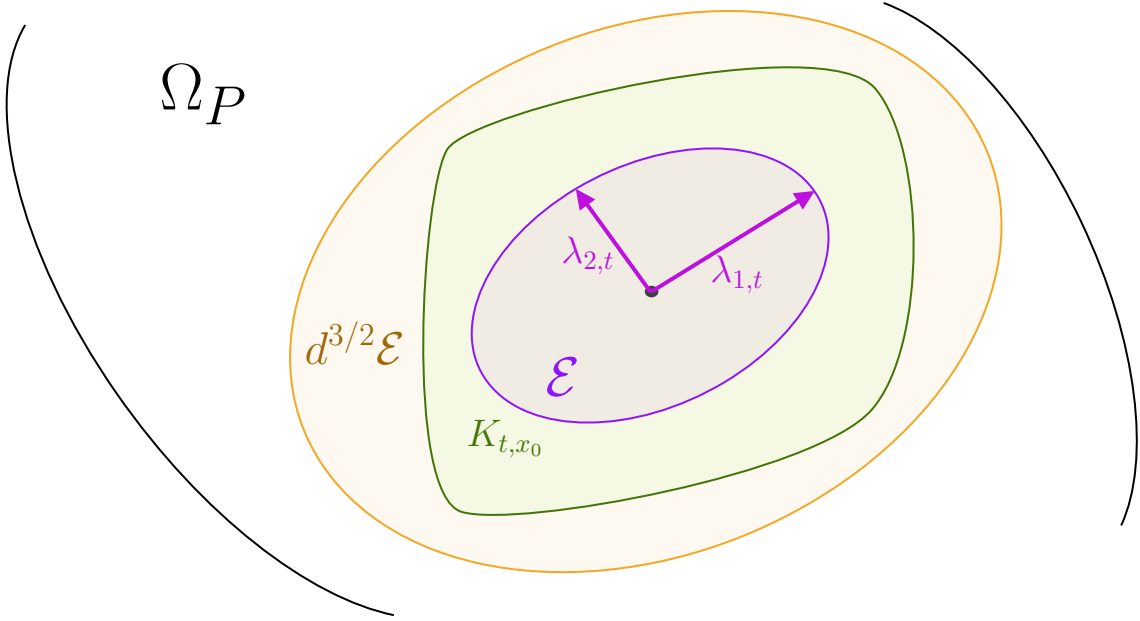
\begin{figure}[!ht]
    \centering
\tikzset{every picture/.style={line width=0.75pt}} %

\begin{tikzpicture}[x=0.65pt,y=0.65pt,yscale=-1,xscale=1]
\draw  [color={rgb, 255:red, 245; green, 166; blue, 35 }  ,draw opacity=1 ][fill={rgb, 255:red, 245; green, 166; blue, 35 }  ,fill opacity=0.07 ] (188.97,270.51) .. controls (136.11,203.06) and (174.91,100.39) .. (275.63,41.2) .. controls (376.35,-17.98) and (500.85,-11.28) .. (553.71,56.18) .. controls (606.56,123.64) and (567.76,226.3) .. (467.04,285.49) .. controls (366.32,344.68) and (241.82,337.97) .. (188.97,270.51) -- cycle ;
\draw  [draw opacity=0] (225.44,350.73) .. controls (176.54,340.72) and (114.18,292.88) .. (65.22,222.77) .. controls (6.95,139.34) and (-13.44,51.44) .. (11.97,8.98) -- (151.68,168.04) -- cycle ; \draw   (225.44,350.73) .. controls (176.54,340.72) and (114.18,292.88) .. (65.22,222.77) .. controls (6.95,139.34) and (-13.44,51.44) .. (11.97,8.98) ;  
\draw  [draw opacity=0] (509.11,-3.78) .. controls (547.84,10.23) and (592.62,57.93) .. (623.53,123.08) .. controls (656.34,192.24) and (664.03,261.05) .. (646.64,298.8) -- (545.55,160.08) -- cycle ; \draw   (509.11,-3.78) .. controls (547.84,10.23) and (592.62,57.93) .. (623.53,123.08) .. controls (656.34,192.24) and (664.03,261.05) .. (646.64,298.8) ;  
\draw  [fill={rgb, 255:red, 0; green, 0; blue, 0 }  ,fill opacity=1 ] (371.34,163.35) .. controls (371.34,161.81) and (372.85,160.57) .. (374.71,160.57) .. controls (376.57,160.57) and (378.09,161.81) .. (378.09,163.35) .. controls (378.09,164.88) and (376.57,166.12) .. (374.71,166.12) .. controls (372.85,166.12) and (371.34,164.88) .. (371.34,163.35) -- cycle ;
\draw  [color={rgb, 255:red, 65; green, 117; blue, 5 }  ,draw opacity=1 ][fill={rgb, 255:red, 220; green, 248; blue, 191 }  ,fill opacity=0.25 ] (256.38,81) .. controls (270.97,58) and (475.23,11) .. (504.41,46) .. controls (533.59,81) and (534.81,195) .. (501.98,232) .. controls (469.15,269) and (298.77,300.48) .. (264.9,291) .. controls (231.03,281.52) and (241.79,104) .. (256.38,81) -- cycle ;
\draw  [color={rgb, 255:red, 144; green, 19; blue, 254 }  ,draw opacity=1 ][fill={rgb, 255:red, 182; green, 107; blue, 248 }  ,fill opacity=0.08 ] (279.98,216.12) .. controls (256.19,185.21) and (278.94,134.95) .. (330.78,103.86) .. controls (382.63,72.76) and (443.94,72.62) .. (467.73,103.54) .. controls (491.52,134.45) and (468.78,184.72) .. (416.94,215.81) .. controls (365.09,246.9) and (303.78,247.04) .. (279.98,216.12) -- cycle ;
\draw [color={rgb, 255:red, 189; green, 16; blue, 224 }  ,draw opacity=1 ][line width=1.5]    (374.71,163.35) -- (465.81,107.15) ;
\draw [shift={(469.21,105.05)}, rotate = 148.33] [fill={rgb, 255:red, 189; green, 16; blue, 224 }  ,fill opacity=1 ][line width=0.08]  [draw opacity=0] (11.61,-5.58) -- (0,0) -- (11.61,5.58) -- cycle    ;
\draw [color={rgb, 255:red, 189; green, 16; blue, 224 }  ,draw opacity=1 ][line width=1.5]    (374.19,163.72) -- (333.13,107.09) ;
\draw [shift={(330.78,103.86)}, rotate = 54.05] [fill={rgb, 255:red, 189; green, 16; blue, 224 }  ,fill opacity=1 ][line width=0.08]  [draw opacity=0] (11.61,-5.58) -- (0,0) -- (11.61,5.58) -- cycle    ;

\draw (88,30) node [anchor=north west][inner sep=0.75pt]  [font=\Huge] [align=left] {$\displaystyle \Omega _{P}$};
\draw (311.57,200) node [anchor=north west][inner sep=0.75pt]  [font=\huge,color={rgb, 255:red, 144; green, 19; blue, 254 }  ,opacity=1 ] [align=left] {$\displaystyle \mathcal{E}$};
\draw (407.18,140.91) node [anchor=north west][inner sep=0.75pt]  [font=\Large,color={rgb, 255:red, 189; green, 16; blue, 224 }  ,opacity=1 ] [align=left] {$\displaystyle \lambda _{1,t}$};
\draw (320.8,129.91) node [anchor=north west][inner sep=0.75pt]  [font=\Large,color={rgb, 255:red, 189; green, 16; blue, 224 }  ,opacity=1 ] [align=left] {$\displaystyle \lambda _{2,t}$};
\draw (265.94,236) node [anchor=north west][inner sep=0.75pt]  [font=\Large,color={rgb, 255:red, 65; green, 117; blue, 5 }  ,opacity=1 ] [align=left] {$\displaystyle K_{t,x_{0}}$};
\draw (172.11,179) node [anchor=north west][inner sep=0.75pt]  [font=\LARGE,color={rgb, 255:red, 160; green, 103; blue, 9 }  ,opacity=1 ] [align=left] {$\displaystyle d^{3/2}\mathcal{E}$};

\end{tikzpicture}\caption{Intuition behind \Cref{prop:shape}. Inside the set $K_{t,x_0} $ (green) we have the ellipsoid $ \mathcal{E} $ (purple) and outside its dilation   $ d^{3/2}\mathcal{E} $ (orange). The goal in \Cref{prop:shape} is to show that the axes $\lambda_{1,t}\geq \lambda_{2,t}$ of  $\mathcal{E} $  decrease with rate $t^{\frac{1}{d}}$ for $d\geq 3$. }
    \label{fig:inturion}
\end{figure}

The following is the celebrated John's Lemma. See \cite[Theorem 1.8.2]{Gutirrez2016} and \cite[Lemma~A.13]{figalli2017monge}. 
\begin{lemma}[John's lemma]\label{lemma:john}
Let $K$ be an open bounded convex set. There exists an ellipsoid $\mathcal{E}$ such that 
    $$ \mathcal{E}\subset  K \subset d^{3/2} \mathcal{E}. $$
The center of $\mathcal{E}$ can be taken as the center of mass of $K$. 
\end{lemma}
Using John's lemma, we are now ready to prove \Cref{prop:shape}. Recall that $K_{t,x_0} := (\partial \varphi_{t,x_0})^{-1}(x_0)$ is the preimage in $\Omega_\mu$, thus $\cL_d(K_{t,x_0})\asymp t$ as $\mu$ is regular.
\begin{proof}[Proof of~\Cref{prop:shape}]
By John's lemma, there exists an ellipsoid $\cE_t=Z_{t, x_0}+ A_t (\BB_1(0))$ and a constant $C>1$ such that 
 $$  \mathcal{E}_t \subset  K_{t,x_0} \subset C \mathcal{E}_t. $$
Note that, as  $\BQ_{P}^{-1}=\mathbf{F}_{P} $, we have $\| \BQ_{P}(x)- \BQ_{P}(x')\|\geq L^{-1}\|x-x'\|$. Furthermore,  using that  $f_\mu\geq \lambda$ on $\Omega_\mu$, we have 
    \begin{align*}
   \left\|\BQ_{P_t}-\BQ_{P} \right\|_{L^2(\mu)}^2 &\geq \int_{K_{t,x_0}} \| \BQ_{P}(z) -x_0\|^2 \mu(\rd z) \\
   &=\int_{K_{t,x_0}} \| \BQ_{P}(z) -\BQ_{P}(\mathbf{F}_{P}(x_0))\|^2 \mu(\rd z)\\
   &\geq  L^{-2}  \int_{K_{t,x_0}} \| z -\mathbf{F}_{P}(x_0)\|^2 \mu(\rd z)\\
   & \geq  L^{-2} \lambda  \int_{ K_{t,x_0}} \| z -\mathbf{F}_{P}(x_0)\|^2 \dd z\geq  L^{-2}  \lambda \int_{K_{t,x_0}} \| z -Z_{t, x_0}\|^2 \dd z\,,
\end{align*}
where the last line follows as the barycenter minimizes the function $y \mapsto \int_{K_{t,x_0}} \|z-y\|^2\dd z$. Changing  variables yields
\begin{align*}
   \int_{ K_{t,x_0}} \| z -Z_{t, x_0}\|^2 \dd z\geq    \int_{ \cE_t} \| z -Z_{t, x_0}\|^2 \dd z &=\det(A_t) \int_{ \BB_1(0)} \| A_t(z)\|^2 \dd z.
\end{align*}
Let $\lambda_{1,t}\geq  \dots \geq \lambda_{d,t}>0$ be the eigenvalues with associated (orthonormal) eigenvectors $e_1, \dots, e_d$  of the matrix $A_t$. As
$$ \|A_t(z)\|^2= \sum_{i=1}^d  \lambda_{i,t}^{2}(\langle z ,e_i \rangle)^2  \geq \lambda_{1,t}^{2}(\langle z ,e_1 \rangle)^2, $$
we conclude that
\begin{equation}\label{eq:lambda_1t}
     \left\|\BQ_{P_t}-\BQ_{P} \right\|_{L^2(\mu)}^2 \gtrsim \lambda_{1,t}^{2} \det(A_t) \int_{ \BB_1(0)} (\langle z ,e_1 \rangle)^2\dd z\gtrsim  \det(A_t) \lambda_{1,t}^{2} .
\end{equation} 
By definition, $ \cL_d(K_{t,x_0}) \asymp \det(A_t)\cL_d(\BB_1(0))$, then $\det(A_t)\asymp t $ as $\cL_d(K_{t,x_0})\asymp \mu(K_{t,x_0}) = P_t(\{x_0\})= t$.  For $d\geq 3$,  \Cref{thm:L2-estimate} and~\eqref{eq:lambda_1t} imply $ \lambda_{1,t} \lesssim  t^{\frac{1}{d}}$ for $t$ sufficiently small. Moreover, since $\det(A_t) \leq \lambda_{d,t}\lambda_{1,t}^{d-1}$, we see that $\lambda_{d,t} \gtrsim t^{\frac{1}{d}}$. Therefore, there exists $R>0$ such that
\[
\BB_{R^{-1}  t^{\frac{1}{d}}}(Z_{t, x_0})\subset  K_{t,x_0}  \subset \BB_{R t^{\frac{1}{d}}}(Z_{t, x_0}).
\]
Finally, as $\mu(K_{t,x_0}) =t$, we note that 
 $$  \left\|\BQ_{P_t}-\BQ_{P} \right\|_{L^2(\mu)}^2 \geq  L^{-2} \,t \cdot \inf_{w\in K_{t,x_0}}\| \mathbf{F}_P (x_0)-w \|^2 $$
 and \Cref{thm:L2-estimate} concludes the proof for $d\geq 3$. 
The case $d=2$ follows from a similar argument and the details are omitted.   

\end{proof}

\subsection{Proof of \Cref{prop:Example-uniform}}
   \textit{(i)}. It is easy to check that $T_t$ pushes ${\cal U}(\BB_1(0))$ forward to $(1-t) {\cal U}(\BB_1(0)) + t \delta_0 $. To check the optimality, it suffices to note that $T_t$ is the gradient of the function 
    $ f(z)=g(\|z\|) $ with 
    $$ g(s)= \begin{cases}
    0&,\quad s\leq t^\frac{1}{d}\\
        \int_{t^\frac{1}{d}}^s \left(\dfrac{r^d - t}{1 - t}\right)^{\frac{1}{d}} \dd r&,\quad s>t^\frac{1}{d}
    \end{cases}. $$
   Since $s\mapsto g'(s)\geq 0$ is non-decreasing, $g$ 
    is convex, and therefore $f$ is convex as well. Hence, (i)~follows from Brenier's theorem~\cite[Theorem~2.32]{Villani2003}.  
    
    \textit{(ii)} Note that
 \begin{align}
     \int_{\BB_1(0)} \|T_t(z) -z\|^2 \dd z = \underbrace{\int_{\BB_{t^{\frac{1}{d}}}(0)} \|z\|^2 \dd z}_{=: A}+  \underbrace{\int_{\BB_1(0)\setminus \BB_{t^{\frac{1}{d}}}(0)} \|T_t(z)-z\|^2 \dd z}_{=: B}. 
 \end{align}   
A standard calculation shows that
$A = C_d t^{1+\frac{2}{d}}$, where the constant $C_d$ depends only on the dimension.  By setting $r = \|z\|$, the coarea formula gives:
\begin{align*}
    B&=  \cH^{d-1}(\cS^{d-1}) \underbrace{\int_{t^{\frac{1}{d}}}^1  \left( r- \left(\dfrac{r^d  - t}{1 - t}\right)^{\frac{1}{d}}\right)^2 r^{d-1} \dd r}_{=: I(t)}, 
\end{align*}
where we recall that $\cH^{d-1}(\cS^{d-1})$ is the surface area of $\cS^{d-1}$.
We now upper bound $I(t)$. First, perform the change of variables \(u=r^d\), so that
\[
I(t)=\frac{1}{d}\int_t^1 
\left(u^{\frac{1}{d}}-\left(\frac{u-t}{1-t}\right)^{\frac{1}{d}}\right)^2 \dd u.
\]
We divide the integral in two parts: $I(t) = I_1(t) + I_2(t)$, where
$$ I_1(t) =  \frac{1}{d}\int_{t}^{2t} 
\left(u^{\frac{1}{d}}-\left(\frac{u-t}{1-t}\right)^{\frac{1}{d}}\right)^2 \dd u $$
and 
$$ I_2(t) =  \frac{1}{d}\int_{2t}^{1} 
\left(u^{\frac{1}{d}}-\left(\frac{u-t}{1-t}\right)^{\frac{1}{d}}\right)^2 \dd u. $$
Since for $u \in [t, 2t]$, 
$$  \left(u^{\frac{1}{d}}-\left(\frac{u-t}{1-t}\right)^{\frac{1}{d}}\right)^2  \leq t^{\frac{2}{d}}, $$
we have
$I_1(t)\lesssim t^{1+\frac{2}{d}}   $. 
Define the function 
$F_u(t)= \left(\frac{u-t}{1-t}\right)^{\frac{1}{d}}. $
Note that $F_u(0) = u^{\frac{1}{d}}$ and thus 
\[
I_2(t) = \frac{1}{d}\int_{2t}^1 (F_u(0) - F_u(t))^2\dd u.
\]
Note first that
$$ F_u'(t)
= \frac{u-1}{d}\,(u-t)^{\frac{1}{d}-1}(1-t)^{-\frac{1}{d}-1}. $$
Then $F_u'(0)=\frac{u-1}{d}u^{\frac{1}{d}-1}$ 
and for $u\geq 2 t$,
$$ \sup_{s\leq t}|F_u'(s)|
\lesssim  |t-u|^{\frac{1}{d}-1}, $$
so that the mean value theorem yields
$$ I_2(t) \lesssim \frac{t^2}{d}\int_{2t}^{1} |t-u|^{2(\frac{1}{d}-1)} \dd u  . $$
Hence, changing variables gives 
$$
I_2(t)\lesssim t^2 \int_t^{1-t} 
u^{\frac{2}{d}-2} \dd u \leq t^2 \int_t^{1} 
u^{\frac{2}{d}-2} \dd u \lesssim \begin{cases}
      t^{1+\frac{2}{d}} & {\rm if} \ d\geq 3,\\
t^{2}{|\log(t)|}     & {\rm if}\ d=2.
  \end{cases} 
$$
 For $d=2$, we observe that  $$ I_2(t) = 
\frac{t^{2}}{8\sqrt{1-t}}
\left[
\log\!\left(\frac{2 - t + 2\sqrt{1-t}}{t}\right)
- 2\sqrt{1-t}
\right] \asymp t^2 |\log(t)|.
 $$
Then the upper bound follows in view of $A=C_d t^{1+\frac{2}{d}}$.

\textit{(iii)}
For any $x\in \BB(0,1)$, by Proposition~1 in \cite{Wirth2020}, the system
\[
\Delta G = 1- \cL_d(\BB(0,1))\delta_{x},\quad \frac{\partial G}{\partial n_\mu} = 0.
\]
admits a solution $\tilde \cG(x, \cdot)\in \cC^{\infty}_{\rm Loc}(\BB(0,1)\setminus\{x\})$, unique up to an additive constant depending on $x$, which is given by 
\begin{equation}\label{eq:influence-function-uniform}
    \tilde \cG(x,z)
=
 \Gamma(z-x)
-
\frac{1}{d}
\int_{0}^{\|x\|}\left(
\frac{\inner{z}{\frac{x}{\|x\|}} - \dfrac{1}{s}}
{\left\| s z - \dfrac{x}{\|x\|} \right\|^{d}} + \frac{1}{s}\right)\dd s + \frac{1}{2d} \|z\|^{2},
\end{equation}
where 
\[
\Gamma(z)
=
\begin{cases}
- |z|, & \text{if } d = 1, \\[6pt]
-\dfrac{1}{2} \log(\|z\|), & \text{if } d = 2, \\[6pt]
\dfrac{1}{(d-2)d\|z\|^{d-2}}, & \text{if } d > 2.
\end{cases}
\]
Define $G_x = \tilde \cG(x,\cdot) - \int_{\BB(0,1)} \tilde \cG(x,z)\dd z$. Then $G_x$ satisfies $\int_{\BB(0,1)} \cG_x(z)\dd z=0$ and solves~\eqref{eq:PDE-main} in the distributional sense. By the uniqueness of the weak solution (see~\Cref{Thm:InfluenceQuantiles-main}), we know that the influence function satisfies $\BI(x;z) =  \nabla_2 G_x(z) = \nabla_2 \tilde \cG(x,z)$. The desired formula follows from a direct computation of $\nabla_2 \tilde \cG(x,z)$.

\subsection{Proof of~\Cref{coro:IF-property}}

We begin with the proof of the relation~\eqref{eq:symmetry}. 
\begin{lemma}[Symmetry of the Green potential]\label{lem:green-symmetry}
Under the setting of~\cref{Thm:InfluenceQuantiles-main}, for any
$x\in \operatorname{int}(\Omega_P)$ and $z\in \operatorname{int}(\Omega_\mu)$ such that
$x\neq \BQ_P(z)$, it holds that
\[
    G_x(z)=G_{\BQ_P(z)}(\BF_P(x)).
\]
\end{lemma}

\begin{proof}
Let $\rho_1^\vae,\rho_2^\vae\in \cC_c^\infty({\Omega_\mu})$ be nonnegative mollifiers with unit mass, supported in disjoint balls around $\BF_P(x)$ and $z$, respectively. Fix $p>d$. Let $u_1^\vae,u_2^\vae \in W^{2,p}(\Omega_\mu)$ solve
\begin{equation}\label{eq:mollified-pde}
        \operatorname{div}\!\left(f_\mu[\nabla \BQ_P]^{-1}\nabla u_1^\vae\right)
    =
    f_\mu-\rho_1^\vae,
    \qquad
    \operatorname{div}\!\left(f_\mu[\nabla \BQ_P]^{-1}\nabla u_2^\vae\right)
    =
    f_\mu-\rho_2^\vae,
\end{equation}
respectively, with boundary condition
\[
    \inner{f_\mu[\nabla \BQ_P]^{-1}\nabla u_1^\vae}{n_\mu}
    =
    \inner{f_\mu[\nabla \BQ_P]^{-1}\nabla u_2^\vae}{n_\mu}
    =
    0
    \qquad\text{on }\partial\Omega_\mu,
\]
and normalization
\begin{equation}\label{eq:normalization-green-symmetry}
        \int_{\Omega_\mu}u_1^\vae\dd\mu
    =
    \int_{\Omega_\mu}u_2^\vae\dd\mu
    =
    0.
\end{equation}
Using integration by parts, the boundary condition, and the symmetry of
$f_\mu[\nabla \BQ_P]^{-1}$, we obtain
\[
\begin{aligned}
    \int_{\Omega_\mu}
    u_1^\vae
    \operatorname{div}\!\left(f_\mu[\nabla \BQ_P]^{-1}\nabla u_2^\vae\right)\dd z
    &=
    -\int_{\Omega_\mu}
    \inner{f_\mu[\nabla \BQ_P]^{-1}\nabla u_1^\vae}
    {\nabla u_2^\vae}\dd z  \\
    &=
    \int_{\Omega_\mu}
    u_2^\vae
    \operatorname{div}\!\left(f_\mu[\nabla \BQ_P]^{-1}\nabla u_1^\vae\right)\dd z .
\end{aligned}
\]
Substituting the equations for $u_1^\vae$ and $u_2^\vae$ gives
\[
    \int_{\Omega_\mu}u_1^\vae(f_\mu-\rho_2^\vae)\dd z
    =
    \int_{\Omega_\mu}u_2^\vae(f_\mu-\rho_1^\vae)\dd z.
\]
By~\eqref{eq:normalization-green-symmetry}, this reduces to
\begin{equation}\label{eq:symmetry-equality}
        \int_{\Omega_\mu}u_1^\vae\rho_2^\vae\dd z
    =
    \int_{\Omega_\mu}u_2^\vae\rho_1^\vae\dd z.
\end{equation}

We will show that $u_1^\vae \to G_x$ along some subsequence uniformly on local balls around $z$ as $\vae\downarrow 0$. Then, by a similar argument, we have $u_2^\vae \to G_{\BQ_P(z)}$ along some subsequence uniformly on local balls around $\BF_P(x)$ as $\vae\downarrow 0$. Choose $r>0$ such that $\BB_r(z)\Subset \operatorname{int}(\Omega_\mu)$, $\BB_r(\BF_P(x))\Subset \operatorname{int}(\Omega_\mu)$, $\BB_r(z)\cap \BB_r(\BF_P(x))=\emptyset$. The signed measures $(f_\mu-\rho_1^\vae)\dd z$ have total mass zero and uniformly bounded total variation. Hence, by Stampacchia's estimates for homogeneous Neumann problem from~\cref{lem:neumann-stampacchia}~(i), together with
the normalization~\eqref{eq:normalization-green-symmetry},
\[
    \sup_{\vae>0}\|u_1^\vae\|_{W^{1,q}(\Omega_\mu)}<\infty
\]
for any $q \in \left(1,\; \frac{d(d+2)}{d^{2}+2d-4}\right)$. 

Take $\vae$ sufficiently small, $\operatorname{supp}(\rho_1^\vae)\cap \BB_r(z)=\emptyset$.
Therefore
\[
    \operatorname{div}\!\left(f_\mu[\nabla \BQ_P]^{-1}\nabla u_1^\vae\right)
    =
    f_\mu
    \qquad\text{in }\BB_r(z).
\]
The local maximum principle~\citep[Theorem~9.20]{GilbargTrudinger.Book} gives
\[
    \sup_{\vae>0}\|u_1^\vae\|_{L^\infty(\BB_{r/2}(z))}\lesssim \frac{\|u_1^\vae\|_{L^q(\BB_{r}(z))}}{r^{\frac{d}{q}}} + r\|f_\mu\|_{L^d(\BB_r(z))} \lesssim 1.
\]
By the interior Schauder estimates~\citep[Theorem~6.2]{GilbargTrudinger.Book},
\begin{equation}\label{eq:uniform-C2alpha}
        \sup_{\vae>0}\|u_1^\vae\|_{\cC^{2,\alpha}(\BB_{r/4}(z))}\lesssim 1.
\end{equation}
By~\cref{lem:neumann-stampacchia}~(ii)--(iv), $u_1^{\vae_n} \to G_x$ strongly in $L^q(\Omega_\mu)$ along some subsequence $\vae_n \downarrow 0$, together with~\eqref{eq:uniform-C2alpha}, the convergence can be upgraded to $\cC^{2,\beta}(\BB_{r/4}(z))$ for every
$0<\beta<\alpha$. Passing, if necessary, to a further subsequence, again denoted by
$\vae_n$, we may also assume that $u_2^{\vae_n}$ converges to
$G_{\BQ_P(z)}$ in
$\cC^{2,\beta}(\BB_{r/4}(\BF_P(x)))$ for every $0<\beta<\alpha$. 
Therefore,
\[
\lim_{n} \int_{\Omega_\mu}u_1^{\vae_n}\rho_2^{\vae_n}\dd z =
    G_{x}(z)\qquad {\rm and} \qquad
\lim_{n} \int_{\Omega_\mu}u_2^{\vae_n}\rho_1^{\vae_n}\dd z =
    G_{\BQ_P(z)}(\BF_P(x)),
\]
and thus~\eqref{eq:symmetry-equality} gives
\[
     G_x(z)=G_{\BQ_P(z)}(\BF_P(x)).
\]
This proves the result.
\end{proof}

\begin{proof}[Proof of~\Cref{coro:IF-property}]
We prove each point separately.

\textit{Proofs of (i) and (ii)}. The claims follow from~\Cref{thm:Unbounded-IF} by direct computation.

\textit{Proof of (iii)}. Recall the bivariate function in~\eqref{rmk:green-function},  $\cG:\Omega_P \times \Omega_\mu \to \R$ such that $\cG(x,z) = G_x(z)$ where $\cG_x:\Omega_\mu\to \R$ is given in~\Cref{Thm:InfluenceQuantiles-main}. Let $h\in \R^d$ such that $\|h\|=1$. For $t$ sufficiently small such that $z+th\in \Omega_\mu$, write 
\[
H_t(x) := \frac{1}{t}\left(\cG(x,z + th) - \cG(x,z)\right)\quad \text{and}\quad H(x) := \inner{\nabla G_x(z)}{h},
\]
where $H(x)$ is defined arbitrarily on the set $\{x\in\Omega_P: \BF_P(x) = z\}$ since the set has measure zero under $P$. The claim follows once we have shown that
\begin{equation}\label{eq:dct}
 \lim_{t\downarrow0} \int_{\Omega_P} H_t \dd P = \int_{\Omega_P} H\dd P.
\end{equation}
Indeed, using the relation $\cG(x,z)= \cG(\BQ_P(z),\BF_P(x))$ from~\cref{lem:green-symmetry} in the third line and the fact that $\int_{\Omega_\mu}G_x \dd \mu=0$ in the last line, we obtain that
\begin{align*}
\int_{\Omega_P} \inner{\BI(x;\BQ_P(z))}{h}\,P(\rd x)
&= \int_{\Omega_P} \inner{\nabla G_x(z)}{h}\,P(\rd x)=
\int_{\Omega_P} H(x)\,P(\rd x)  = \lim_{t\downarrow0} \int_{\Omega_P} H_t(x) \, P(\rd x)\\
 &=\lim_{t\downarrow0} \frac{1}{t}\int_{\Omega_P}\left(\cG(x,z + th) - \cG(x,z)\right)\,P(\rd x) \\
& = \lim_{t\downarrow0}\frac{1}{t} \int_{\Omega_P}\left(\cG(\BQ_P(z + th), \BF_P(x)) - \cG(\BQ_P(z),\BF_P(x))\right)\,P(\rd x) \\
& = \lim_{t\downarrow0} \frac{1}{t} \int_{\Omega_P}\left(G_{\BQ_P(z + th)}(\BF_P(x)) - G_{\BQ_P(z)}(\BF_P(x))\right)\,P(\rd x) \\
 & =  \lim_{t\downarrow0} \frac{1}{t} \int_{\Omega_\mu}\left(G_{\BQ_P(z + th)}(\tilde z) - G_{\BQ_P(z)}(\tilde z)\right)\,\mu(\rd \tilde z) =0.
\end{align*}
The rest of the proof is devoted to establishing~\eqref{eq:dct}. To that end, we will rely on a localization argument. Choose $R>0$ small enough and define 
\[
A_R := \{x\in \Omega_P: \|\BF_P(x) -z \|\leq R\}\setminus \{\BQ_P(z)\}.
\]
By~\Cref{thm:Unbounded-IF}~(i), there exists $C>0$ such that we have 
\[
|H(x)| = |\inner{\nabla G_x(z)}{h}| \leq \|\nabla G_x(z)\|\leq \frac{C}{\|\BF_P(x) -z\|^{d-1}} \leq CR^{1-d},\quad \forall x\in A_R^c.
\]
On the other hand, choose $t$ small enough such that $\|\BF_P(x) - (z+th)\| \geq R/2$, then $\|\BF_P(x) - (z+sth)\| \geq R/2$ for all $s\in [0,1]$. Hence, by the mean-value theorem, we conclude that
\[
|H_t(x)| \leq \int_0^1 \left|\inner{\nabla G_x(z + sth)}{h}\right|\dd s \leq \int_0^1 \frac{C}{\|\BF_P(x) -(z + sth)\|^{d-1}} \dd s\leq C2^{d-1}R^{1-d}.
\]
Therefore, applying the dominated convergence theorem shows
\begin{equation}\label{eq:dct-ARc}
     \lim_{t\downarrow0}\int_{A_R^c} H_t \dd P = \int_{A_R^c} H\dd P.
\end{equation}
We proceed with controlling the integrals over $A_R$. Since $f_\mu \leq \Lambda$, we derive by \Cref{thm:Unbounded-IF}~(i) that
\begin{align*}
    \left|\int_{A_R} H \dd P \right|&\leq \int_{A_R} \frac{C}{\|\BF_P(x)-z\|^{d-1}} P(\rd x)\\
    & \leq \int_{\{\|\tilde z -z\|\leq R\}} \frac{C}{\|\tilde z - z\|^{d-1}} \mu(\rd \tilde z)\leq C_1\int_0^R \frac{1}{r^{d-1}}d r^{d-1}\dd r\leq C_1 R,
\end{align*}
where the constant $C_1>0$ depends on $d, C, \Lambda$. 
Moreover, for $|t|\leq R/2$, the mean-value theorem and
\Cref{thm:Unbounded-IF}~(i)  give
\[
    |H_t(x)|
    \leq
    C\int_0^1
    \frac{ds}{
    \|\BF_P(x)-(z+sth)\|^{d-1}
    }.
\]
Therefore, using $(\BF_P)_\#P=\mu$, Fubini--Tonelli's theorem, and
$f_\mu\leq\Lambda$, we obtain
\begin{align*}
    \int_{A_R}|H_t(x)|\,dP(x)
    &\leq
    C\int_0^1
    \int_{\{\|\widetilde z-z\|\leq R\}}
    \frac{\mu(\rd \widetilde z)\dd s}{
    \|\widetilde z-(z+sth)\|^{d-1}
    }
    \\
    &\leq
    C\Lambda\int_0^1
    \int_{\BB_{3R/2}(z+sth)}
    \frac{\dd \widetilde z\dd s}{
    \|\widetilde z-(z+sth)\|^{d-1}
    }
    \\
    &\leq
    C_2\int_0^{3R/2}
    r^{-(d-1)}r^{d-1}\dd r
    \leq C_3R,
\end{align*}
where $C_2,C_3>0$ are independent of $t$ and $R$.
Hence, we obtain that 
\[
\limsup_{t\to 0}\left|\int_{A_R} (H_t - H) \dd P\right| \leq (C_1+C_3)R.
\]
In view of~\eqref{eq:dct-ARc}, we obtain that
\[
\limsup_{t\to 0}\left|\int_{\Omega_P} (H_t - H) \dd P\right| \leq (C_1+C_3)R.
\]
Letting $R\to 0$ yields~\eqref{eq:dct}.
\end{proof}

\subsection{Proof of~\cref{prop:asym-linear-form-test}}\label{sec:proof-prop:asym-linear-form-test}

Following the same argument as in the proof of~\cref{lemma:test-linearization} and write $f = g\circ \BF_P$, we have
\begin{equation}\label{eq:tylor-linear-form-2}
\left|
\int_{\Omega_\mu}
\left\langle
\BQ_{P_n}-\BQ_P,
[\nabla\BQ_P]^{-1}\nabla g
\right\rangle
\dd\mu
-
(P_n-P)(f)
\right|
\lesssim
\|g\|_{\CC^{1,\eta}(\Omega_\mu)}
\|\BQ_{P_n}-\BQ_P\|_{L^2(\mu)}^{1+\eta}.
\end{equation}
Since $P$ is absolutely continuous with respect to Lebesgue measure
and $\partial\Omega_P$ has Lebesgue measure zero, we have
$P(\partial\Omega_P)=0$. Moreover, by~\cref{thm:Cafarrelli},
$\BF_P$ maps ${\rm int}(\Omega_P)$ onto
${\rm int}(\Omega_\mu)$. Hence, with probability one,
$X_i\in{\rm int}(\Omega_P)$ and
$\BF_P(X_i)\in{\rm int}(\Omega_\mu)$ for every $i=1,\ldots,n$; we work on this event. 

In view of $\BI(X_i;\BQ_P(z))=\nabla G_{X_i}(z)$ and
$\dd\mu=f_\mu\dd z$, applying \cref{lem:neumann-stampacchia}~(iv) with $a=\BF_P(X_i)$ and $g\in\cC^{1,\eta}(\Omega_\mu)$
implies that
\begin{equation}\label{eq:IF-test-identity}
\int_{\Omega_\mu}
\left\langle
\BI(X_i;\BQ_P(\cdot)),
[\nabla\BQ_P]^{-1}\nabla g
\right\rangle
\dd\mu
=
g(\BF_P(X_i))
-
\int_{\Omega_\mu}g\,\dd\mu.
\end{equation}
Averaging \eqref{eq:IF-test-identity} over $i=1,\ldots,n$ gives
\[
\int_{\Omega_\mu}
\left\langle
\frac1n\sum_{i=1}^n
\BI(X_i;\BQ_P(\cdot)),
[\nabla\BQ_P]^{-1}\nabla g
\right\rangle
\dd\mu
=
\frac1n\sum_{i=1}^n g(\BF_P(X_i))
-
\int_{\Omega_\mu}g\,\dd\mu
=
(P_n-P)(f).
\]
Combining this identity with
\eqref{eq:tylor-linear-form-2} yields
\begin{multline*}
\left|
\int_{\Omega_\mu}
\left\langle
\BQ_{P_n}-\BQ_P,
[\nabla\BQ_P]^{-1}\nabla g
\right\rangle
\dd\mu
-
\int_{\Omega_\mu}
\left\langle
\frac1n\sum_{i=1}^n
\BI(X_i;\BQ_P(\cdot)),
[\nabla\BQ_P]^{-1}\nabla g
\right\rangle
\dd\mu
\right|
\\
\lesssim
\|g\|_{\CC^{1,\eta}(\Omega_\mu)}
\|\BQ_{P_n}-\BQ_P\|_{L^2(\mu)}^{1+\eta}.
\end{multline*}
Recall that
$\|\BQ_{P_n}-\BQ_P\|_{L^2(\mu)}
\lesssim\cW_2(P_n,P)$ by
\cite[Theorem~6]{Manole.2024.Aos}, while
\[
\E[\cW_2^2(P_n,P)]
\lesssim
\alpha(n,d)
:=
\begin{cases}
n^{-1}\log(n),& d=2,\\
n^{-2/d},& d\geq3,
\end{cases}
\]
by \cite[Corollary~8]{Manole.2024.Aos} when $d>2$ and by
\cite{Ambrosio-EJS} when $d=2$. Since $\eta\in(0,1]$, Jensen's
inequality gives
\[
r_n
\E\!\left[
\|\BQ_{P_n}-\BQ_P\|_{L^2(\mu)}^{1+\eta}
\right]
\lesssim
r_n
\alpha(n,d)^{(1+\eta)/2}
=o(1).
\]
Markov's inequality therefore yields
\[
r_n
\left|
\int_{\Omega_\mu}
\left\langle
\BQ_{P_n}-\BQ_P,
[\nabla\BQ_P]^{-1}\nabla g
\right\rangle
\dd\mu
-
\int_{\Omega_\mu}
\left\langle
\frac1n\sum_{i=1}^n
\BI(X_i;\BQ_P(\cdot)),
[\nabla\BQ_P]^{-1}\nabla g
\right\rangle
\dd\mu
\right|
=
o_{\PP}(1).
\]

If $d=2,3$ and
$\eta\in((d-2)/2,1]$, then
$\sqrt n\,\alpha(n,d)^{(1+\eta)/2}=o(1)$. Hence the same argument
gives
\[
\sqrt n
\left|
\int_{\Omega_\mu}
\left\langle
\BQ_{P_n}-\BQ_P,
[\nabla\BQ_P]^{-1}\nabla g
\right\rangle
\dd\mu
-
\int_{\Omega_\mu}
\left\langle
\frac1n\sum_{i=1}^n
\BI(X_i;\BQ_P(\cdot)),
[\nabla\BQ_P]^{-1}\nabla g
\right\rangle
\dd\mu
\right|
=
o_{\PP}(1).
\]
Finally, by \eqref{eq:IF-test-identity},
\[
Y_i
:=
\int_{\Omega_\mu}
\left\langle
\BI(X_i;\BQ_P(\cdot)),
[\nabla\BQ_P]^{-1}\nabla g
\right\rangle
\dd\mu
=
g(\BF_P(X_i))
-
\int_{\Omega_\mu}g\,\dd\mu.
\]
Since $(\BF_P)_\#P=\mu$, the random variables
$\BF_P(X_i)$ are i.i.d. with law $\mu$. Therefore,
\[
\E[Y_i]=0,
\qquad
{\rm Var}(Y_i)
=
{\rm Var}_\mu(g(Z))
=
\int_{\Omega_\mu}g^2\,\dd\mu
-
\left(
\int_{\Omega_\mu}g\,\dd\mu
\right)^2
=:\sigma_g^2<\infty.
\]
The classical central limit theorem gives
\[
\frac1{\sqrt n}\sum_{i=1}^nY_i
\overset{d}{\longrightarrow}
\cN(0,\sigma_g^2).
\]
Combining this convergence with the preceding $o_{\PP}(1)$ estimate
and Slutsky's theorem yields
\[
\sqrt n
\int_{\Omega_\mu}
\left\langle
\BQ_{P_n}-\BQ_P,
[\nabla\BQ_P]^{-1}\nabla g
\right\rangle
\dd\mu
\overset{d}{\longrightarrow}
\cN(0,\sigma_g^2),
\qquad d=2,3.
\]

\appendix
\section{Background on PDE}\label{sec:solvability}

We define the (Besov) boundary space of traces   
$$\mathcal{B}^{k-\frac{1}{p},p}(\partial\Omega) := \left\{ \phi \in \mathcal{D}'(\partial\Omega):  \text{there exists a function } u \in W^{k,p}(\Omega) \text{ with } \gamma_0( u) = \phi \right\},$$
where $\gamma_0$ denotes the trace operator and $\mathcal{D}'(\partial\Omega)$ the space of distributions on $\partial\Omega$. Recall that 
the trace operator $\gamma_0$ is the unique bounded, linear operator
\[
\gamma_0 : W^{k,p}(\Omega) \to B^{k-\frac{1}{p},p}(\partial\Omega)
\]
defined for smooth functions $u \in C^\infty(\overline{\Omega})$ by the restriction
\[
\gamma_0(u) = u|_{\partial\Omega},
\]
and extended to the entire Sobolev space $W^{k,p}(\Omega)$ by continuity (we refer to \cite[Chapter~7.3]{Taira.2024.book} for further details). We define the norm 
$$ \| \phi \|_{B^{k-\frac{1}{p},p}(\partial\Omega)} = \inf \left\{ \| u \|_{W^{k,p}(\Omega)} : u \in W^{k,p}(\Omega),\ \gamma_0 u = \phi \text{ on } \partial\Omega \right\}, $$
 which  makes $ \mathcal{B}^{k-\frac{1}{p},p}(\partial\Omega)$ a Banach space.

Define the operator
\begin{equation}\label{eq:operator-L}
   \begin{aligned}
 \BL :  W^{2,p}(\Omega_\mu) &\to   L^p(\Omega_\mu) \times  \mathcal{B}^{1-\frac{1}{p},p}(\partial\Omega_\mu)\\
\xi&\mapsto \left(\begin{array}{c}
 \BL_1(\xi)\\
  \BL_2(\xi)
 \end{array} \right)= \left(\begin{array}{c}
 {\rm div} (f_\mu [\nabla\BQ_{P}]^{-1} \nabla \xi ) 
 \\
  \langle  [\nabla\BQ_{P}]^{-1}  \nabla \xi,   n_{\mu} \rangle f_\mu
 \end{array} \right),
\end{aligned} 
\end{equation}
which is bounded by the trace theorem (see \cite[Theorem~7.8]{Taira.2024.book}). Here $n_{\mu} $ denotes  the outer
normal unit vector field at $\partial \Omega_\mu$. 
Furthermore, $ \BL $  is uniformly elliptic  with strictly oblique boundary condition. We call $\xi\in  W^{2,p}(\Omega_\mu)$ a strong solution to the system 
\[
  \BL(\xi) = (f,g)
\]
if $\BL_1(\xi) =f$ holds almost everywhere in $\Omega_\mu$ and $\BL_2(\xi) =g$ holds in the trace sense. We now present a result that characterizes the invertibility of the operator $\BL$ and provides an \emph{a priori} estimate for the strong solution~$\xi$.

\begin{lemma}\label{lemma: existence and uniqueness linearized MA}
    Let $p\in (1, 
    \infty)$. For any $f \in L^{p}({\Omega}_\mu)$ and $g\in  \mathcal{B}^{1-\frac{1}{p},p}(\partial\Omega_\mu)$, the system 
 \begin{align}\label{eq: linearized MongeAmpere general}
  \BL(\xi) = (f,g)
\end{align}
admits a unique  (up to additive shifts) strong solution  $ \xi \in W^{2,p}({\Omega}_\mu) $ for all $p\in (1, \infty)$ if
\begin{equation}\label{eq: compatibility condition-general}
\int_{\Omega_\mu}  f\dd z = \int_{\partial \Omega_\mu} g  \dd \cH^{d-1}.
\end{equation}
Furthermore, in that case, there exists a constant $C$ depending on $\alpha, d, p, \mu, P$ such that 
\begin{equation}
    \label{eq:estimate-Calderon}
    \inf_{a\in \R}\| \xi-a\|_{W^{2,p}(\Omega_\mu)} \leq C \left(\|f\|_{L^p(\Omega_\mu)}+ \| g\|_{\cB^{1-\frac{1}{p},p}(\partial \Omega_\mu)}\right). 
\end{equation}
\end{lemma}
The proof follows essentially the same technique as in \cite[Theorem~5.6]{gonzalez2024linearization}, which establishes solvability of the system in the classical sense.
\begin{remark}
In the proof, we showed that there exists a unique solution $\xi \in W^{2,p}(\Omega_\mu)$ satisfying $\int_{\partial \Omega_\mu} \xi \dd \cH^{d-1} = 0$. However, since the operator $\BL$ acts on $\nabla \xi$, adding a constant to $\xi$ yields another solution of the same system. In particular, we may instead normalize $\xi$ by imposing $\int_{\Omega_\mu} \xi \dd z = 0$.
\end{remark}

\begin{proof}
We prove uniqueness and existence separately.

{\it Uniqueness.}
We first consider $p\geq 2$. The case $1<p<2$ will be treated after
existence has been proved for the conjugate exponent. 
If $\BL(\xi)=0$ for $\xi\in W^{2,p}(\Omega_\mu)$ and $p\geq 2$, then for all $\psi\in {\cC^\infty(\Omega_\mu)},$ 
\begin{align*}
\int_{\Omega_\mu}
\left\langle
\nabla\psi,[\nabla\BQ_P]^{-1}\nabla\xi
\right\rangle
\dd\mu
={}&
-\int_{\Omega_\mu}
\psi\,{\rm div}\!\left(
f_\mu[\nabla\BQ_P]^{-1}\nabla\xi
\right)\dd z
\\
&+
\int_{\partial\Omega_\mu}
\psi
\left\langle
[\nabla\BQ_P]^{-1}\nabla\xi,n_\mu
\right\rangle
f_\mu\dd\cH^{d-1}
=0.
\end{align*}
Since $p\geq2$ and $\mu$ is equivalent to the Lebesgue measure on
$\Omega_\mu$, it follows by approximation that
\[
\int_{\Omega_\mu}
\left\langle
\nabla\xi,[\nabla\BQ_P]^{-1}\nabla\xi
\right\rangle
\dd\mu
=0.
\]
Since the eigenvalues of $[\nabla\BQ_P]^{-1}$ are uniformly bounded away
from zero and {bounded above}, it follows that
\[
\int_{\Omega_\mu}\|\nabla\xi\|^2\dd\mu=0,
\]
which implies, by the Poincaré inequality
\citep[p.~291]{Evans-PDE}, that $\xi$ is constant.

{\it Existence.}
Assume first that $p\geq2$.
We define an auxiliary operator
\[
\mathbb{M}:
W^{2,p}(\Omega_\mu)
\longrightarrow
L^p(\Omega_\mu)
\times
\mathcal{B}^{1-\frac1p,p}(\partial\Omega_\mu)
\]
as
\[
\mathbb{M}(\xi)
=
\left(
\begin{array}{c}
{\rm div}\!\left(
f_\mu[\nabla\BQ_P]^{-1}\nabla\xi
\right)
\\[1mm]
\left\langle
[\nabla\BQ_P]^{-1}\nabla\xi,n_\mu
\right\rangle f_\mu
+
\gamma_0\xi
\end{array}
\right)
=
{\BL(\xi)}
+
(0,\gamma_0\xi).
\]
As the eigenvalues of $[\nabla\BQ_P]^{-1}$ and $f_\mu$ are uniformly
bounded away from zero and bounded above, the boundary
condition is strictly oblique; see
\cite[Lemma~5.9]{gonzalez2024linearization}. Hence,
\cite[Theorem~16.2]{Taira.2024.book} implies that $\mathbb{M}$ admits a
bounded inverse. We rewrite $\BL(\xi)=(f,g)$ as
\begin{equation}
\label{eq:Solvable}
\xi-\mathbb{M}^{-1}(0,\gamma_0\xi)
=
\mathbb{M}^{-1}(f,g).
\end{equation}
  We show now that for 
  $$ (f,g)\in \mathcal{Y}_{p,0}:=\left\{ (f,g)\in L^p(\Omega_\mu)
\times
\mathcal{B}^{1-\frac1p,p}(\partial\Omega_\mu) \ \text{satisfying \eqref{eq: compatibility condition-general}} \right\}$$
  the equation \eqref{eq:Solvable}
 is uniquely solvable for some $\xi \in W^{2,p}(\Omega_\mu)$ such that $
\int_{\partial \Omega_\mu} \xi \dd \cH^{d-1} = 0$. 
The trace operator
\[
\gamma_0:W^{2,p}(\Omega_\mu)
\longrightarrow
\mathcal B^{2-\frac1p,p}(\partial\Omega_\mu)
\]
is bounded. Furthermore,  by the definition of the boundary Besov spaces, the
Rellich--Kondrachov theorem \cite[Theorem~7.6]{Taira.2024.book}, and the
trace theorem \cite[Theorem~7.8]{Taira.2024.book},  the embedding
\[
\mathcal B^{2-\frac1p,p}(\partial\Omega_\mu)
\hookrightarrow
\mathcal B^{1-\frac1p,p}(\partial\Omega_\mu)
\]
is compact.  Since
\[
\mathbb M^{-1}:
L^p(\Omega_\mu)\times
\mathcal B^{1-\frac1p,p}(\partial\Omega_\mu)
\longrightarrow W^{2,p}(\Omega_\mu)
\]
is bounded by \cite[Theorem~16.2]{Taira.2024.book}, we  conclude that the operator
$$ W^{2,p}(\Omega_\mu)\ni \xi \mapsto \mathbb{M}^{-1}(0,\gamma_0\xi)\in W^{2,p}(\Omega_\mu) $$ is compact.
Now we need to check that $$\mathbb{M}^{-1}(f,g)\in \mathcal{X}_p:= \left\{ \xi \in W^{2,p}(\Omega_\mu): 
\int_{\partial \Omega_\mu} \xi \dd \cH^{d-1} = 0\right\}$$ for  $(f,g)\in \mathcal{Y}_{p,0}$. Since the map
\[
u\longmapsto
\int_{\partial\Omega_\mu}\gamma_0u\,d\mathcal H^{d-1}
\]
is continuous on \(W^{2,p}(\Omega_\mu)\), the space
\(\mathcal X_p\) is a closed subspace (hence Banach) of \(W^{2,p}(\Omega_\mu)\). If $(f,g)\in\mathcal{Y}_{p,0}$  and
$u=\mathbb{M}^{-1}(f,g)$, then in the strong sense 
\[
\operatorname{div}\!\left(
f_\mu[\nabla\BQ_P]^{-1}\nabla u
\right)=f
\]
and
\[
\left\langle
f_\mu[\nabla\BQ_P]^{-1}\nabla u,n_\mu
\right\rangle+\gamma_0u=g.
\]
Integrating the first identity and using the divergence theorem and the
second identity, we obtain
\[
\int_{\partial\Omega_\mu}\gamma_0u\,d\mathcal H^{d-1}
=
\int_{\partial\Omega_\mu}g\,d\mathcal H^{d-1}
-
\int_{\Omega_\mu}f\,dz
=0.
\]
Hence $\mathbb{M}^{-1}(f,g)\in\mathcal{X}_p$.  
For every $\xi\in\mathcal{X}_p$, the pair
$(0,\gamma_0\xi)$ belongs to $\mathcal{Y}_{p,0}$. Therefore, $\mathbb{M}^{-1}(0,\gamma_0\xi)\in\mathcal{X}_p$ and the Fredholm alternative \citep[Theorem~6.6]{brezis2011functional} implies that
$$\xi\mapsto \xi-\mathbb{M}^{-1}(0,\gamma_0\xi) $$ is  invertible on  $\mathcal{X}_p$ 
 if and only if  
the equation $ u-\mathbb{M}^{-1}(0,\gamma_0 u) =0 $ admits only the zero solution $u=0$. Note that  if $ u-\mathbb{M}^{-1}(0,\gamma_0 u) =0 $, then 
$$ \BL(u)+ (0,\gamma_0 u)= \mathbb{M}(u)=(0,\gamma_0 u), $$
and thus $\BL(u)=0$. By the uniqueness argument above, $u$ is constant; since $u\in\mathcal{X}_p$, it follows that $u=0$.  Therefore, $\xi\mapsto \xi-\mathbb{M}^{-1}(0,\gamma_0\xi) $ is  invertible on  $\mathcal{X}_p$. Since  $\mathbb{M}^{-1}(f,g)\in \mathcal{X}_p$ for  $(f,g)\in \mathcal{Y}_{p,0}$, we deduce that   \eqref{eq:Solvable} is uniquely solvable for some $\xi\in W^{2,p}(\Omega_\mu)$ satisfying $\int_{\partial \Omega_\mu} \xi \dd \cH^{d-1} = 0$. This concludes the existence for $p\geq 2$.

Now let $1<p<2$ and let
\[
p'=\frac{p}{p-1}>2.
\]
Suppose that $\xi\in W^{2,p}(\Omega_\mu)$ satisfies $\BL(\xi)=0$. For every
$h\in L^{p'}(\Omega_\mu)$ satisfying
\[
\int_{\Omega_\mu}h\,\dd z=0,
\]
the result already proved at exponent $p'>2$ provides
$v\in W^{2,p'}(\Omega_\mu)$ such that
\[
\BL(v)=(h,0).
\]
Using Green's formula, the symmetry of $[\nabla\BQ_P]^{-1}$, and the
homogeneous boundary conditions for $\xi$ and $v$, we obtain
\begin{align*}
\int_{\Omega_\mu}\xi h\,\dd z
&=
-\int_{\Omega_\mu}
\left\langle
f_\mu[\nabla\BQ_P]^{-1}\nabla\xi,\nabla v
\right\rangle
\dd z
\\
&=
\int_{\Omega_\mu}
v\,{\rm div}\!\left(
f_\mu[\nabla\BQ_P]^{-1}\nabla\xi
\right)\dd z
=0.
\end{align*}
Hence $\xi$ annihilates every zero-mean function in
$L^{p'}(\Omega_\mu)$, and therefore $\xi$ is constant. Repeating the
Fredholm argument above at exponent $p$ gives existence and uniqueness for every
$p\in(1,\infty)$. 

Finally, by the bounded inverse theorem,
\[
\inf_{a\in\mathbb{R}}
\|\xi-a\|_{W^{2,p}(\Omega_\mu)}
\leq
C\left(
\|f\|_{L^p(\Omega_\mu)}
+
\|g\|_{\mathcal{B}^{1-\frac1p,p}(\partial\Omega_\mu)}
\right).
\]
\end{proof}

The following result is an adaptation of Stampacchia's truncation argument~\citep{Stampacchia1965,MasoMuratOrsinaPrignet1999} to the homogeneous Neumann problem.
\begin{lemma}\label{lem:neumann-stampacchia}
Under the setting of~\Cref{Thm:InfluenceQuantiles-main}, fix
$a\in\operatorname{int}(\Omega_\mu)$ and $0<\varepsilon_0<\operatorname{dist}
    (a,\partial\Omega_\mu)$. 
For every $\varepsilon\in(0,\varepsilon_0]$, let
$\rho_\varepsilon\in\cC_c^\infty
(\Omega_\mu)$ be nonnegative and satisfy
\[
    \int_{\Omega_\mu}\rho_\varepsilon\dd z=1,
    \qquad
    \operatorname{supp}(\rho_\varepsilon)
    \subset\BB_\varepsilon(a),
\]
so that $ (f_\mu-\rho_\varepsilon)\dd z \to
    \mu-\delta_a$ weakly as signed measures as $\vae \to 0$.
Fix $p>d$. For each $\vae>0$, let
$u_\varepsilon\in W^{2,p}(\Omega_\mu)$ be the unique normalized
solution (see~\cref{lemma: existence and uniqueness linearized MA}) of
\[
\begin{cases}
{\rm div}\!\left(
f_\mu[\nabla\BQ_P]^{-1}\nabla u_\varepsilon
\right)
=f_\mu-\rho_\varepsilon
&\text{in }\operatorname{int}(\Omega_\mu),\\[1mm]
\displaystyle
\inner{
f_\mu[\nabla\BQ_P]^{-1}\nabla u_\varepsilon
}{
n_\mu
}
=0
&\text{on }\partial\Omega_\mu,\\[1mm]
\displaystyle
\int_{\Omega_\mu}u_\varepsilon\dd\mu=0.
\end{cases}
\]
Then, for any $r\in(1,d/(d-1))$, the following holds:
\begin{enumerate}[label = (\roman*)]
    \item $\sup_{0<\varepsilon\leq\varepsilon_0}
    \|u_\varepsilon\|_{W^{1,r}(\Omega_\mu)}
    <\infty$.
    \item There exists $\varepsilon_k \to 0$ and $u\in W^{1,r}(\Omega_\mu)$ such that
\begin{equation}\label{eq:neumann-stampacchia-convergence}
\begin{aligned}
    u_{\varepsilon_k}
    \rightharpoonup u
    &&\text{weakly in }W^{1,r}(\Omega_\mu),\qquad
    u_{\varepsilon_k}\longrightarrow u
    &&\text{strongly in }L^r(\Omega_\mu)
\end{aligned}
\end{equation}
\item The limit $u$ satisfies $\int_{\Omega_\mu}u\dd\mu=0$ and
\begin{equation}\label{eq:neumann-transposition-limit}
    \int_{\Omega_\mu}
    u\,
    {\rm div}\!\left(
        f_\mu[\nabla\BQ_P]^{-1}\nabla\varphi
    \right)\dd z
    =
    \int_{\Omega_\mu}\varphi\dd\mu-\varphi(a)
\end{equation}
for every $\varphi\in W^{2,r'}(\Omega_\mu)$, where
$r'=r/(r-1)>d$, satisfying
\[
    \inner{
        f_\mu[\nabla\BQ_P]^{-1}\nabla\varphi
    }{
        n_\mu
    }
    =0
    \qquad\text{on }\partial\Omega_\mu
\]
in the trace sense. 
\item If $1<r<\frac{d(d+2)}{d^2+2d-4}$, then $u=G_{\BQ_P(a)}$ a.e.~in $\Omega_\mu$. Consequently, 
\[
\int_{\Omega_\mu}
\left\langle
f_\mu[\nabla\BQ_P]^{-1}\nabla G_{\BQ_P(a)},
\nabla g
\right\rangle
\dd z
= g(a) - \int_{\Omega_\mu}g\,\dd\mu
\]
for all $g\in \cC^{1,\eta}(\Omega_\mu)$, $\eta \in (0,1]$.
\end{enumerate}
\end{lemma}

\begin{proof}
The compatibility condition in~\cref{lemma: existence and uniqueness linearized MA} holds because $\int_{\Omega_\mu}
(f_\mu-\rho_{\varepsilon})\,\dd z
=0$, and thus $u_\vae$ is well-defined.

\textit{(i)}. Let $m_\varepsilon$ be a median of
$u_\varepsilon$ and set $v_\varepsilon:=u_\varepsilon-m_\varepsilon$.
For $k>0$, let $T_k(s):=\max\{-k,\min\{s,k\}\}$.
Since the equation and the homogeneous conormal boundary condition are
unchanged by the addition of constants, $v_\varepsilon$ satisfies the
same equation as $u_\varepsilon$. Testing with
$T_k(v_\varepsilon)$ and using the uniform ellipticity of
$f_\mu[\nabla\BQ_P]^{-1}$ from
\Cref{defn:regular-measures} and \cref{thm:Cafarrelli}, we obtain
\begin{align}
    \int_{\{|v_\varepsilon|<k\}}
    |\nabla v_\varepsilon|^2\dd z
    &\lesssim
    \int_{\Omega_\mu}
    |T_k(v_\varepsilon)|
    (f_\mu+\rho_\varepsilon)\dd z
    \leq 2k.
\label{eq:neumann-stampacchia-truncation}
\end{align}
In what follows, $C$ denotes a generic constant which may change from
line to line and is independent of $k$ and $\varepsilon$.

We first record the zero-median form of the Poincar\'e inequality used
below. If $w \in W^{1,q}(\Omega_\mu)$ has median zero, then its Lebesgue mean satisfies
\[
\left|
\frac{1}{\mathcal L^d(\Omega_\mu)}
\int_{\Omega_\mu}w\dd z
\right|
\leq
\frac{2}{\mathcal L^d(\Omega_\mu)}
\int_{\Omega_\mu}
\left|
w-
\frac{1}{\mathcal L^d(\Omega_\mu)}
\int_{\Omega_\mu}w\dd z
\right|\dd z.
\]
Consequently, the usual Poincar\'e inequality~\citep[p.~291]{Evans-PDE} yields
\begin{equation}\label{eq:zero-median-poincare}
    \|w\|_{L^q(\Omega_\mu)}
    \lesssim
    \|\nabla w\|_{L^q(\Omega_\mu)},
    \qquad 1\leq q<\infty.
\end{equation}

Assume that $d>2$. Since $T_k(v_\varepsilon)$ has median zero,
H\"older's inequality, the Sobolev inequality~\citep[Section~5.6.1, Theorem~1]{Evans-PDE}, and Poincar\'e inequality give
\begin{align*}
\|T_k(v_\varepsilon)\|_{L^{2d/(d-2)}(\Omega_\mu)}
&\leq
\left\|
T_k(v_\varepsilon)
-
\frac{1}{\mathcal L^d(\Omega_\mu)}
\int_{\Omega_\mu}T_k(v_\varepsilon)\dd z
\right\|_{L^{2d/(d-2)}(\Omega_\mu)}
\\
&\quad+
(\mathcal L^d(\Omega_\mu))^{(d-2)/(2d)}
\left|
\frac{1}{\mathcal L^d(\Omega_\mu)}
\int_{\Omega_\mu}T_k(v_\varepsilon)\dd z
\right|
\\
&\leq
C
\left\|
T_k(v_\varepsilon)
-
\frac{1}{\mathcal L^d(\Omega_\mu)}
\int_{\Omega_\mu}T_k(v_\varepsilon)\dd z
\right\|_{L^{2d/(d-2)}(\Omega_\mu)}
\\
&\leq
C\|\nabla T_k(v_\varepsilon)\|_{L^2(\Omega_\mu)}.
\end{align*}
Since $|T_k(v_\varepsilon)|=k$ on
$\{|v_\varepsilon|\geq k\}$, it follows from
\eqref{eq:neumann-stampacchia-truncation} that
\begin{align*}
k^{2d/(d-2)}
\mathcal L^d\bigl(\{|v_\varepsilon|\geq k\}\bigr)
&\leq
\int_{\{|v_\varepsilon|\geq k\}}
|T_k(v_\varepsilon)|^{2d/(d-2)}\dd z
\\
&\leq
\|T_k(v_\varepsilon)\|_{L^{2d/(d-2)}(\Omega_\mu)}^{2d/(d-2)}
\leq
Ck^{d/(d-2)}.
\end{align*}
Hence,
\begin{equation}\label{eq:neumann-distribution-u}
    \mathcal L^d
    \bigl(\{|v_\varepsilon|\geq k\}\bigr)
    \leq
    Ck^{-d/(d-2)}.
\end{equation}
For $s,k>0$, Chebyshev's inequality,
\eqref{eq:neumann-stampacchia-truncation}, and
\eqref{eq:neumann-distribution-u} yield
\begin{align*}
\mathcal L^d
\bigl(\{|\nabla v_\varepsilon|>s\}\bigr)
&\leq
\mathcal L^d
\bigl(\{|v_\varepsilon|>k\}\bigr)
+
\frac{1}{s^2}
\int_{\{|v_\varepsilon|\leq k\}}
|\nabla v_\varepsilon|^2\dd z
\\
&\leq
Ck^{-d/(d-2)}
+
\frac{Ck}{s^2}.
\end{align*}
Choosing $k=s^{(d-2)/(d-1)}$ gives
\[
    \mathcal L^d
    \bigl(\{|\nabla v_\varepsilon|>s\}\bigr)
    \leq
    Cs^{-d/(d-1)}.
\]
As a consequence, for every $r\in(1,d/(d-1))$, $\|\nabla v_\varepsilon\|_{L^r(\Omega_\mu)}\leq C_r$.
Applying \eqref{eq:zero-median-poincare} gives
\begin{equation}\label{eq:neumann-W1r-v-dgeq3}
    \|v_\varepsilon\|_{W^{1,r}(\Omega_\mu)}
    \leq C_r.
\end{equation}

When $d=2$, for every finite $s_0>2$, the Rellich--Kondrachov theorem
\citep[Theorem~7.6]{Taira.2024.book}, together with
\eqref{eq:zero-median-poincare}, gives
\[
    \|T_k(v_\varepsilon)\|_{L^{s_0}(\Omega_\mu)}
    \leq
    C_{s_0}
    \|\nabla T_k(v_\varepsilon)\|_{L^2(\Omega_\mu)}.
\]
Arguing as above, we obtain
\[
    \mathcal L^2
    \bigl(\{|v_\varepsilon|\geq k\}\bigr)
    \leq
    C_{s_0}k^{-s_0/2}
\]
and hence
\[
    \mathcal L^2
    \bigl(\{|\nabla v_\varepsilon|>s\}\bigr)
    \leq
    C_{s_0}k^{-s_0/2}
    +
    \frac{Ck}{s^2}.
\]
Choosing $k=s^{4/(s_0+2)}$ gives a uniform weak
$L^{2s_0/(s_0+2)}$ estimate for $\nabla v_\varepsilon$. Since
$2s_0/(s_0+2)\uparrow2$ as $s_0\to\infty$, we conclude that
\[
    \|\nabla v_\varepsilon\|_{L^r(\Omega_\mu)}
    \leq C_r,
    \qquad 1<r<2.
\]
Applying \eqref{eq:zero-median-poincare} again gives
\begin{equation}\label{eq:neumann-W1r-v-dgeq2}
    \|v_\varepsilon\|_{W^{1,r}(\Omega_\mu)}
    \leq C_r.
\end{equation}

Finally, since
$\int_{\Omega_\mu}u_\varepsilon\dd\mu=0$ and $\mu$ is a probability
measure,
\[
    m_\varepsilon
    =
    -\int_{\Omega_\mu}v_\varepsilon\dd\mu.
\]
By H\"older's inequality and the boundedness of $f_\mu$,
\[
    |m_\varepsilon|
    \leq
    \|f_\mu\|_{L^{r'}(\Omega_\mu)}
    \|v_\varepsilon\|_{L^r(\Omega_\mu)}
    \leq C_r.
\]
Combining with~\eqref{eq:neumann-W1r-v-dgeq3} and~\eqref{eq:neumann-W1r-v-dgeq2}, this proves~(i).

\textit{(ii)}. By weak compactness in the reflexive space
$W^{1,r}(\Omega_\mu)$ and the Rellich--Kondrachov theorem
\citep[Theorem~7.6]{Taira.2024.book}, after passing to a subsequence,
there exists $u\in W^{1,r}(\Omega_\mu)$ such that
\[
    u_{\varepsilon_k}
    \rightharpoonup u
    \quad\text{weakly in }W^{1,r}(\Omega_\mu),
    \qquad
    u_{\varepsilon_k}
    \longrightarrow u
    \quad\text{strongly in }L^r(\Omega_\mu).
\]

\textit{(iii)}. Let $\varphi\in W^{2,r'}(\Omega_\mu)$ satisfy
\[
    \inner{
        f_\mu[\nabla\BQ_P]^{-1}\nabla\varphi
    }{
        n_\mu
    }
    =0
    \qquad\text{on }\partial\Omega_\mu
\]
in the trace sense. Integrating by parts and using the symmetry of the
coefficient matrix gives
\[
\begin{aligned}
\int_{\Omega_\mu}
u_{\varepsilon_k}
{\rm div}\!\left(
f_\mu[\nabla\BQ_P]^{-1}\nabla\varphi
\right)\dd z
&=
\int_{\Omega_\mu}\varphi\dd\mu
-
\int_{\Omega_\mu}\varphi\rho_{\varepsilon_k}\dd z.
\end{aligned}
\]
Since $r'>d$, Morrey's inequality
\citep[Theorem~7.26]{GilbargTrudinger.Book} gives $W^{2,r'}(\Omega_\mu)
    \hookrightarrow
    \cC(\Omega_\mu)$.
Moreover, ${\rm div}\!\left(
    f_\mu[\nabla\BQ_P]^{-1}\nabla\varphi
    \right)
    \in L^{r'}(\Omega_\mu)$.
Therefore, the strong $L^r$ convergence and the weak convergence of
$\rho_{\varepsilon_k}\dd z$ allow us to pass to the limit and obtain
\[
\int_{\Omega_\mu}
u\,
{\rm div}\!\left(
f_\mu[\nabla\BQ_P]^{-1}\nabla\varphi
\right)\dd z
=
\int_{\Omega_\mu}\varphi\dd\mu-\varphi(a).
\]
The strong $L^r$ convergence and the boundedness of $f_\mu$ also imply
\[
    \int_{\Omega_\mu}u\dd\mu
    =
    \lim_{k\to\infty}
    \int_{\Omega_\mu}u_{\varepsilon_k}\dd\mu
    =0.
\]
This proves \eqref{eq:neumann-transposition-limit}.

\textit{(iv)}. Since $u$ satisfies the identity~\eqref{eq:Solbable} and the normalization
$\int_{\Omega_\mu} u\,\dd\mu=0$ that characterize $G_{\BQ_P(a)}$, \cref{prop:uniqueness} yields that $u=G_{\BQ_P(a)}$ a.e.~in $\Omega_\mu.$  Combining with (ii)~yields
\[
\nabla u_{\varepsilon_k}
\rightharpoonup
\nabla G_{\BQ_P(a)}
\quad\text{weakly in }L^r(\Omega_\mu).
\]
Since $g\in\cC^{1,\eta}(\Omega_\mu)$, we have
$g\in W^{1,r'}(\Omega_\mu)$. 
Testing the weak equation for
$u_{\varepsilon_k}$ against $g$ gives
\[
-\int_{\Omega_\mu}
\left\langle
f_\mu[\nabla\BQ_P]^{-1}\nabla u_{\varepsilon_k},
\nabla g
\right\rangle
\dd z
=
\int_{\Omega_\mu}g\,\dd\mu
-
\int_{\Omega_\mu}
g(z)\rho_{\varepsilon_k}(z)\,\dd z.
\]
Because
$f_\mu[\nabla\BQ_P]^{-1}\nabla g\in L^{r'}(\Omega_\mu)$, the weak
$L^r$ convergence of the gradients allows passage to the limit.
The continuity of $g$ and the weak convergence of
$\rho_{\varepsilon_k}\dd z$ then give
\[
-\int_{\Omega_\mu}
\left\langle
f_\mu[\nabla\BQ_P]^{-1}\nabla G_{\BQ_P(a)},
\nabla g
\right\rangle
\dd z
=
\int_{\Omega_\mu}g\,\dd\mu
-
g(a),
\]
as desired. 
\end{proof}

\bibliographystyle{apalike}
\bibliography{ref}
\end{document}